\documentclass[10pt,a4paper]{amsart}
\usepackage[utf8]{inputenc}
\usepackage[T1]{fontenc}
\usepackage[french]{babel}
\usepackage{amsmath,amsfonts,amssymb,amsthm}
\usepackage{stmaryrd} 
\usepackage{array}
\usepackage{enumerate}
\usepackage{hyperref}
\newcommand{\Q}{\mathbb{Q}}
\newcommand{\N}{\mathbb{N}}
\newcommand{\Z}{\mathbb{Z}}

\newcommand{\F}{\mathbb{F}}

\newcommand{\Gal}{\mathrm{Gal}}

\theoremstyle{plain}
\newtheorem{thm}{Théorème}[section]
\newtheorem{prop}[thm]{Proposition}
\newtheorem{cor}[thm]{Corollaire}
\newtheorem{lmm}[thm]{Lemme}

\newtheorem*{lmm*}{Lemme}
\newtheorem*{thm*}{Théorème}
\newtheorem*{claim*}{Fait}
\newtheorem*{hyp*}{Hypothèse}

\theoremstyle{plain}
\newtheorem{rqu}[thm]{Remarque}

\newtheorem{claim}[thm]{Fait}
\newtheorem{hyp}[thm]{Hypothèse}

\theoremstyle{plain}
\newtheorem{defn}[thm]{Définition}
\newtheorem{ex}[thm]{Exemple}
\usepackage[all]{xy}
\usepackage{tikz}
\usetikzlibrary{shadows,decorations.pathmorphing}
\usepackage{comment}
\usepackage{ifthen}

\begin{document}

\title{Calcul des groupes de ramification de certaines extensions radicales.}
\date{}
\author{Arnaud Plessis}
\address{Université de Caen Normandie}
\email{arnaud.plessis@unicaen.fr}
\maketitle

\begin{abstract}
Pour un premier $p\neq 2$, nous calculons la suite des groupes de ramification d'une extension galoisienne, radicale et finie $L/F$ où $F/\Q_p$ est une extension non-ramifiée finie.
Tout d'abord, on les calcule dans le cas où $\Gal(L/F)$ est d'exposant une puissance de $p$, par un théorème de Hecke et quelques récurrences. 
Enfin, nous traitons le cas général. 

Pour le cas $p=2$, la même méthode peut s'appliquer, à condition de faire une hypothèse supplémentaire. 
\end{abstract}

\begin{abstract}
For a rational prime $p\neq 2$, we compute the sequence of ramification groups of a Galois, radical and finite extension $L/F$ where $F/\Q_p$ is an unramified finite extension.
First, we compute it in the case where the exponent of $\Gal(L/F)$ is a power of $p$, by a Hecke's theorem and some inductions,. 
Finally, we deal with the general case. 

For the case $p=2$, the same method can work, provided that we make an additional hypothesis. 
\end{abstract}

\tableofcontents

\section{Introduction.}

Soit $p$ un nombre premier.
Tous les corps considérés dans cette introduction seront, sauf mention explicite du contraire, des extensions finies de $\Q_p$.  
Le calcul explicite des groupes de ramification d'une extension finie et galoisienne $L/K$ est en général difficile. 
Ils sont connus dans peu de cas. 
Par exemple : si $L/K$ est modérément ramifiée (cf \cite[Chapitre IV, Corollaire 3]{Corpslocaux}), si $K=\Q_p$ et $L$ est une extension cyclotomique (cf \cite[Chapitre IV,Proposition 18]{Corpslocaux}), si $L$ est le compositum de toutes les extensions de degré $p$ sur $K$ et $K$ est quelconque (cf \cite[Theorem 11]{CapuanoLauraDelCorsoAnoteonupperramificationjumpsinbelianextensionsofexponent}).
On sait également traiter le cas où $K=\Q_p$ et $L=\Q_p(\zeta_{p^r}, a^{1/p^s})$ avec $p\neq 2$, $r\geq s\geq 0$ deux entiers et $a\in\Z\subset \Z_p$ tel que $v_p(a)\in\{0,1\}$ (cf \cite[Theorem 5.8, Theorem 6.5]{viviani}).

Dans cet article, nous nous proposons de généraliser le travail de Viviani (cf \cite{viviani} en étudiant les groupes de ramification d'une extension radicale, finie et galoisienne $L/F$ où $F/\Q_p$ est une extension non-ramifiée, i.e. $L=F(\zeta_m, a_1^{1/m},\dots,a_n^{1/m})$ où $n,m\in\N^*$ et où $a_1,\dots,a_n\in F^*$.
Il s'agit des extensions galoisiennes non abéliennes les plus simples.
Ce sont donc les "premières" extensions qui ne soient pas classifiées par la théorie des corps de classes locaux.

Dans son article, Viviani s'est restreint au cas où $v_p(a)\in\{0,1\}$. 
La raison est que, sous cette hypothèse, il est possible de trouver une uniformisante de $\Q_p(\zeta_p, a^{1/p^s})/\Q_p(\zeta_p, a^{1/p^{s-1}})$. 
Cette uniformisante permet d'en déduire le saut du groupe de Galois de cette extension. 
Par récurrence, il calcule ainsi les groupes de ramifications de $\Gal(\Q_p(\zeta_{p^r}, a^{1/p^s})/\Q_p)$.

Dans notre approche, on ne cherchera pas à déterminer explicitement une uniformisante de $\Q_p(\zeta_p, a^{1/p^s})/\Q_p(\zeta_p, a^{1/p^{s-1}})$, mais à calculer directement ce saut par un théorème de Hecke (cf \cite[Théorème 10.2.9]{Advancedtopicsincomputationalnumbertheory}).
Ce théorème n'imposant pas de restriction sur $v_p(a)$, cela nous permettra de calculer explicitement le saut de $\Gal(\Q_p(\zeta_p, a^{1/p^s})/\Q_p(\zeta_p, a^{1/p^{s-1}}))$ quelque soit la valeur de $v_p(a)$.

Par un certain algorithme, on sera ainsi en mesure de calculer explicitement les groupes de ramification de $\Gal(\Q_p(\zeta_{p^r}, a^{1/p^s})/\Q_p)$ pour tout $a\in\Q_p$ et tous $r,s\in\N$ avec $r\geq s$.

Soient $t$ un entier premier à $p$ et $a\in\Q_p$. 
Alors, l'extension $\Q_p(\zeta_t, a^{1/t})/\Q_p$ est modérément ramifiée. 
Ainsi, son premier groupe de ramification est trivial. 
Cela nous permettra donc de compléter l'étude de Viviani (i.e. le cas $F=\Q_p$ et $n=1$) en calculant les groupes de ramification de $\Gal(\Q_p(\zeta_m, a^{1/m})/\Q_p)$ où $m=p^{v_p(m)}t$. 

Si $p\neq 2$, une récurrence sur $n\geq 2$ permet d'exprimer les sauts de $\Gal(K/F)$ en fonction de ceux de $\Gal(K'/F)$ où $K=F(\zeta_{p^r}, a_1^{1/p^s},\dots,a_n^{1/p^s})$ et où $K'=F(\zeta_{p^r}, a_2^{1/p^s},\dots,a_n^{1/p^s})$. 
Le cas $p=2$ nécessite plus d'attention car les mêmes arguments que le cas $p\neq 2$ peuvent s'appliquer, à condition de faire une hypothèse technique supplémentaire. 
En généralisant l'algorithme dans le cas $n=1$, on pourra alors en déduire la suite des groupes de ramification de $\Gal(K/F)$. 

Pour les mêmes raisons que le cas $n=1$, on pourra alors en déduire les groupes de ramification de $\Gal(L/F)$ où $L=F(\zeta_m, a_1^{1/m},\dots,a_n^{1/m})$ avec $m=p^{v_p(m)}t$ où $t$ est un entier premier à $p$. \\

On peut trouver au moins deux applications de nos résultats. 
La première concerne les courbes de Tate.
Soient $F/\Q_p$ une extension non-ramifiée et $E/F$ une courbe de Tate (pour la définition, cf \cite[Appendix C.14]{Thearithmeticofellipticcurves}). 
Un théorème de Tate (cf \cite[Appendix C, Theorem 14.1 (a)]{Thearithmeticofellipticcurves}) montre qu'il existe $q\in F^*$ et un morphisme (explicite) $\phi : \overline{F}^*/q^\Z \to E(\overline{F})$ tels que pour toute extension algébrique $L/F$, on ait $\phi \; : L^*/q^\Z \simeq E(L)$.

Soit $m\geq 1$ un entier. 
Fixons une racine $m$-ième de l'unité et de $q$ que l'on note respectivement $\zeta_m$ et $q^{1/m}$. 
De l'isomorphisme ci-dessus, on en déduit que le groupe $E[m]$ des points de $m$-torsion est égal à $E[m]=\langle \phi(\zeta_m), \phi(q^{1/m})\rangle$. 
Soient $P_1,\dots,P_n\in E(F)$. 
Pour tout $i$, notons $a_i\in F^*/q^\Z$ tel que $P_i=\phi(a_i)$. 
L'expression de $\phi$ montre que $F(E[m], [1/m]P_1,\dots,[1/m]P_n)=F(\zeta_m, q^{1/m}, a_1^{1/m},\dots, a_n^{1/m})$ (ici, on a confondu la classe $a_i\in F^*/q^\Z$ avec son représentant canonique $a_i\in F^*$).
Par conséquent, notre étude permet de calculer la suite des groupes de ramifications du groupe de Galois (sur $F$) du corps engendré par les points de torsion d'ordre $m$ et les racines $m$-ièmes d'un nombre fini de points $F$-rationnels d'une courbe de Tate.    

La seconde application concerne la théorie des hauteurs.
Soient $\Gamma$ un groupe de type fini et $S$ un ensemble fini de premiers rationnels. 
Notons \[\Gamma_{\mathrm{div}, S}=\left\{ g\in \overline{\Q} \mid \exists n\in\N^*, g^n\in\Gamma \;\text{et} \; (p\mid n \Rightarrow p\in S)\right\}.\]
Comme l'a remarqué Amoroso dans \cite{OnaconjofRemondtada}, la connaissance du dernier groupe de ramification des extensions considérées dans \cite{viviani} permet de minorer uniformément la hauteur de Weil de l'ensemble $\Q(\langle 2\rangle_{\mathrm{div}, \{3\}})\backslash \langle 2\rangle_{\mathrm{div},\{3\}}$.
Comme notre étude généralise les résultats de Viviani utilisé dans \cite{OnaconjofRemondtada}, il devrait être possible de minorer uniformément la hauteur de Weil des éléments de l'ensemble $K(\Gamma_{\mathrm{div}, S})\backslash \Gamma_{\mathrm{div}, S}$ où $K$ est un corps de nombres, où $\Gamma\subset K^*$ est un groupe de type fini et où $S$ est un ensemble fini de premiers impairs (cette restriction vient de notre hypothèse supplémentaire lorsque $p=2$) ne divisant pas le discriminant de $K$. 
Cela fait l'objet d'un prochain travail en cours. \\

\textbf{Remerciements :} 
Je souhaite remercier F. Amoroso qui a su simplifier ce présent travail. 
Je souhaite également remercier I. Del Corso pour nos différents échanges sur ce sujet et pour des suggestions qui m'ont permises de traiter le cas $p=2$.  
Enfin, je souhaite remercier P. Satgé pour les nombreuses discussions, fortes enrichissantes, que l'on a eues ensemble.  

\section{Résultats principaux et plan de l'article.} \label{résultats principaux}

 Dans cet article, fixons un premier $p$.
  Pour une extension galoisienne et finie $L/K$ de corps locaux contenant $\Q_p$ et de groupe de Galois $G$, on appelle \textit{$i$-ème groupe de ramification} le groupe: 
\begin{equation*}
G_i:=\{\sigma\in G\;\vert\;\forall\;x\in \mathcal{O}_L,\;\sigma(x)\equiv x \mod \mathfrak{p}_L^{i+1}\}
\end{equation*} 
où $\mathfrak{p}_L$ désigne l'idéal premier de $L$.
Il n'est pas dur de remarquer que $G_i$ est un sous-groupe normal de $G$ pour tout $i$ et que $G_j=\{1\}$ pour tout $j$ assez grand. 
On appellera saut de $G$ (ou saut de $L/K$) tout entier $t$ tel que $G_t\neq G_{t+1}$. 

Dans la suite, $F/\Q_p$ désignera une extension finie non ramifiée (mais pas fixée).
                 
\begin{defn} \label{defn de L r s}
Soient $a_1,\dots,a_n\in F$.
 Soient $r,s_1,\dots,s_n\in \N$ avec $r\geq 1$.
  On note $L_{r,s_1,\dots,s_n}$ le corps $F(\zeta_{p^r}, a_1^{1/p^{s_1}},\dots, a_n^{1/p^{s_n}})$. 
 \end{defn}

On se propose de calculer les groupes de ramification de $G:=\Gal(L_{r,s_1,\dots,s_n}/F)$ où $r\geq \max\{s_1,\dots,s_n\}$ et où $F$ et $a_1,\dots,a_n$ vérifient les conditions ci-dessous : 

\begin{hyp} \label{hypothèse 3}
  \begin{enumerate} [i)]
  \item $a_1,\dots,a_n\in\mathcal{O}_F\backslash\mathcal{O}_F^p$;
  \item $[L_{r,s_1,\dots,s_n} : F(\zeta_{p^r})]=p^{\sum_{i=1}^n s_i}$ ;
  \item $p\mid v_p(a_i)$ pour $i=1,\dots,n-1$;
  \item $L_{r,s_1,\dots,s_n}/F$ est totalement ramifiée.
  \end{enumerate}
 \end{hyp}

Même si le corps $L_{r,s_1,\dots,s_n}$ dépend du choix des $a_1,\dots,a_n$, on verra que tous les résultats que nous énoncerons sur $L_{r,s_1,\dots,s_n}$ ne dépendent pas du choix des $a_1,\dots,a_n$, mais uniquement de $v_F(a_n)$, si ces derniers vérifient l'hypothèse \ref{hypothèse 3}, ce qui justifie la notation.

Notons $\delta_{i,j}$ le symbole de Kronecker.
Soient $r\geq 1$, $s_1,\dots, s_n$ des entiers positifs et $k\in\{1,\dots,n+1\}$ tels que $s_{k-1}\geq 1$. 
Pour tout $i$, posons 
\begin{equation} \label{définition de r k}
r(k):=r-\delta_{1,k}=\begin{cases} 
r-1 \; & \text{si} \; k=1 \\
r \; & \text{sinon} 
\end{cases}
\end{equation}
et 
\begin{equation} \label{définition de s k}
 s_i(k):=s_i-\delta_{i+1,k}=\begin{cases} 
s_i-1 \; & \text{si} \; k=i+1 \\
s_i \; & \text{sinon} 
\end{cases}.
\end{equation}

Remarquons que si $r\geq 2$ ou $k\geq 2$ alors, $L_{r,s_1,\dots,s_n}/L_{r(k),s_1(k),\dots,s_n(k)}$ est de degré $p$.
Sinon, $L_{r,s_1,\dots,s_n}/L_{r(k),s_1(k),\dots,s_n(k)}$ est l'extension $L_{1,s_1,\dots,s_n}/L_{0,s_1,\dots,s_n}$ qui est de degré $p-1$.
Dans les deux cas, $L_{r,s_1,\dots,s_n}/L_{r(k),s_1(k),\dots,s_n(k)}$ a un unique saut.
 
\begin{defn} \label{définition  du saut} 
 L'unique saut de l'extension $L_{r,s_1,\dots,s_n}/L_{r(k),s_1(k),\dots,s_n(k)}$ sera noté $t_{n,k}(r,s_1,\dots,s_n)$. 
 \end{defn}
 
Dans la proposition  \ref{Théorème fondamental} $iv)$ et la proposition \ref{Théorème fondamental 2.} $iv)$, on montrera que ce saut ne dépend pas du choix des $a_1,\dots,a_n$, mais uniquement de $v_F(a_n)$, si ces derniers vérifient l'hypothèse \ref{hypothèse 3},  ce qui justifie cette notation.

Dans la section \ref{calcul des sauts}, on montrera les deux théorèmes ci-dessous : 

\begin{thm} \label{calcul des sauts, cas où la valuation est divisible par p}
Soient $F$ et $a_1,\dots,a_n$ vérifiant l'hypothèse \ref{hypothèse 3}. 
 Supposons de plus que $p\mid v_F(a_n)$ et que $s_1\leq \dots \leq s_n$.
Alors, les sauts de $G$, différents de $0$, sont les 

\begin{enumerate} [i)]
\item $t_{n,1}(s,s_1,\dots,s_n)$ avec $s\in \{s_n+1,\dots,r\}$;
\item $t_{n,1}(s_n'+1,s_1',\dots,s_n')$ où $s_n'\in\{1,\dots,s_n-1\}$ et où $s_i'=\min\{s_i,s_n'\}$;
\item $t_{n,n+1}(s_n',s_1',\dots,s_n')$ où $s_n'\in\{1,\dots,s_n\}$ et où $s_i'=\min\{s_i,s_n'\}$.
\end{enumerate}

De plus, si on note $u_j$ le $(n+1)$-uplet tel que $t_j=t_{n,k}(u_j)$ désigne le $(j+1)$-ième plus petit saut de $G$ alors, 

\begin{enumerate} [i)]
\item $G_0=G$;
\item $G_{t_j}=\Gal(L_{r,s_1,\dots,s_n}/L_{u_{j-1}})$ pour $j\geq 1$.
\end{enumerate}

\end{thm}

\begin{thm} \label{calcul des sauts, cas où la valuation n'est pas divisible par p}
Soient $F$ et $a_1,\dots,a_n$ vérifiant l'hypothèse \ref{hypothèse 3}.
Supposons de plus que $p\nmid v_F(a_n)$ et que $s_1\leq \dots\leq s_{n-1}$.
Alors, les sauts de $G$, différents de $0$, sont les 

\begin{enumerate} [i)]
\item $t_{n,n+1}(r,s_1,\dots,s_n) $ si $r=s_n$;
\item $t_{n,1}(s_n'+1,s_1',\dots,s_{n-1}', \min\{s_n'-1,s_n\})$ où $s_n'\in\{1,\dots,r\}$ et $s_i'=\min\{s_i,s_n'\}$;
\item $t_{n,n}(s_{n-1}',s_1',\dots,s_{n-1}',s_n)$ où $s_{n-1}'\in\{s_n+2,\dots,s_{n-1}\}$ et $s_i'=\min\{s_i,s_{n-1}'\}$;
\item $t_{n,n+1}(s_n'+1,s_1',\dots,s_n')$ où $s_n'\in\{1,\dots,s_n\}$ et $s_i'=\min\{s_i,s_n'+1\}$.
\end{enumerate}

De plus, si on note $u_j$ le $(n+1)$-uplet tel que $t_j=t_{n,k}(u_j)$ désigne le $(j+1)$-ième plus petit saut de $G$ alors,

\begin{enumerate} [i)]
\item $G_0=G$;
\item $G_{t_j}=\Gal(L_{r,s_1,\dots,s_n}/L_{u_{j-1}})$ pour $j\geq 1$.
\end{enumerate}
\end{thm}

Remarquons que si $p\mid v_F(a_n)$ (resp. $p\nmid v_F(a_n)$) alors, on peut permuter convenablement les $(a_1^{1/p^{s_1}}, \dots, a_n^{1/p^{s_n}})$ (resp. $(a_1^{1/p^{s_1}}, \dots, a_{n-1}^{1/p^{s_{n-1}}})$) tels que $s_1\leq \dots \leq s_n$ (resp. $s_1\leq \dots \leq s_{n-1}$). \\

\textit{Plan de l'article :}

Dans la section \ref{section rappels}, on calculera $G_0$ et $G_1$. 
On fera également quelques rappels et préliminaires.
 
Dans la section \ref{section avec tous les calculs chiants}, si $p\neq 2$, on montrera que le calcul des groupes de ramification de toute extension radicale, galoisienne et finie de groupe de Galois d'exposant une puissance une $p$ se ramène au calcul des groupes de ramification d'une extension $L_{r,s_1,\dots,s_n}/F$ où $r\geq \max\{s_1,\dots,s_n\}$ et où $F$ et $a_1,\dots,a_n$ vérifient les conditions de l'hypothèse \ref{hypothèse 3}. 
Dans le cas $p=2$, on montrera, à condition de faire une hypothèse technique supplémentaire, que notre problème initial se ramène également à calculer les groupes de ramification d'une extension $L_{r,s_1,\dots,s_n}/F$ où $r\geq \max\{s_1,\dots,s_n\}$ et où $F$ et $a_1,\dots,a_n$ vérifient les conditions de l'hypothèse \ref{hypothèse 3}. 

Dans la section \ref{cas n=1}, on calculera les sauts $t_{1,k}(r,s)$ avec $k\in\{1,2\}$ et $r,s\in\N^*$.

Dans la section \ref{Relations de récurrence entre les}, on établira, à partir du cas $n=1$ traité dans la section \ref{cas n=1}, certaines relations de récurrence entre les $t_{n,k}(r,s_1,\dots,s_n)$.

Dans la section \ref{calculs chiants}, on utilisera les relations de récurrence que l'on a ainsi obtenues afin de calculer explicitement les nombres $t_{n,k}(r,s_1,\dots,s_n)$.

Dans la section \ref{calcul des sauts}, le calcul des nombres $t_{n,k}(r,s_1,\dots,s_n)$ de la section précédente, couplé avec le corollaire \ref{utile pour déterminer la suite des groupes de rami}, nous permettra ainsi de montrer le théorème \ref{calcul des sauts, cas où la valuation est divisible par p} et le théorème \ref{calcul des sauts, cas où la valuation n'est pas divisible par p}.

Enfin, la section \ref{Une généralisation des Théorème} sera consacrée à une généralisation naturelle du théorème \ref{calcul des sauts, cas où la valuation est divisible par p} et du théorème \ref{calcul des sauts, cas où la valuation n'est pas divisible par p} en considérant non plus une extension radicale, galoisienne et finie de groupe de Galois d'exposant une puissance de $p$, mais d'exposant quelconque.

  \section{Rappels et préliminaires.} \label{section rappels}
Dans la suite, tous les corps considérés seront des extensions finies de $\Q_p$. 
Pour une extension galoisienne $L/K$ de groupe de Galois $G$, il est en général difficile de calculer les $G_i$ (et même les sauts). 
Cependant, il existe une expression explicite de $G_0$ (cf \cite[Chapitre IV, Corollaire]{Corpslocaux}) et de $G_1$ (cf \cite[Chapitre IV, Corollaire 1, Corollaire 3]{Corpslocaux}), à savoir que $G_0=\Gal(L/K^{nr})$ et $G_1=\Gal(L/K^{mr})$ où $K^{nr}$ (resp. $K^{mr}$) désigne l'extension maximale non-ramifiée (resp. modérément ramifiée) de $K$ contenue dans $L$.
Dans notre situation, cela donne

\begin{prop} \label{calculs de 0 et 1}
Soient $r,s_1,\dots,s_n\geq 1$ avec $r\geq \max\{s_1,\dots,s_n\}$ et $F, a_1,\dots,a_n$ vérifiant les conditions de l'hypothèse \ref{hypothèse 3}.
 On a \[G_0=G \; \text{et} \; G_1=\Gal(L_{r,s_1,\dots,s_n}/F(\zeta_p))\] où $G=\Gal(L_{r,s_1,\dots,s_n}/F)$.
\end{prop}
   
 \begin{proof}
Le fait que $G_0=G$ découle du $iv)$ de l'hypothèse \ref{hypothèse 3}.
Comme $F/\Q_p$ est non-ramifiée et que $\Q_p(\zeta_{p^r})/\Q_p$ est totalement ramifiée, on en déduit que $[F(\zeta_{p^r}) : F]=(p-1)p^{r-1}$.  
Ainsi, $[L_{r,s_1,\dots,s_n} : F]=(p-1)p^{r-1+\sum_{i=1}^n s_i}$ d'après le $ii)$ de l'hypothèse \ref{hypothèse 3}.
Il s'ensuit donc que $F(\zeta_p)/F$ est la sous-extension maximale modérément ramifiée de $L_{r,s_1,\dots,s_n}/F$. 
Ainsi, $G_1=\Gal(L_{r,s_1,\dots,s_n}/F(\zeta_p))$. 
   \end{proof}

En revanche, si $H$ est un sous-groupe de $G$, il est facile de calculer les $H_i$ en fonction des $G_i$. 
Plus précisément, on a que $H_i=G_i \cap H$ pour tout $i\geq 0$ (cf \cite[Chapitre IV, Proposition 2]{Corpslocaux}).   
     
Il est également possible, dans le cas où $H$ est un sous groupe normal de $G$, de calculer explicitement les quotients $(G/H)_i$.
Mais avant d'expliciter la formule, étendons de manière naturelle la notion de groupes de ramification en posant, pour tout réel $u\geq -1$, $G_u=G_i$ où $i=\lceil u\rceil$ (la partie entière supérieure de $u$). On peut ainsi définir la fonction $\phi_{L/K}$ de Herbrand, valable pour tout $u\geq -1$, par: 
\begin{equation*}  
  \phi_{L/K}(u)= \begin{cases}
      \int_0^u [G_0 : G_u]^{-1}\mathrm{d} t\; \text{si}\;u\geq 0 \\ \\[0.3 mm]
      u\;\text{si}\; -1\leq u<0
     \end{cases}.
  \end{equation*}
 C'est une bijection de $[-1,+\infty [$ dans lui-même (cf \cite[Chapitre IV, Proposition 12]{Corpslocaux}).
 On notera $\psi_{L/K}$ sa bijection réciproque. On peut maintenant expliciter les quotients $(G/H)_i$.

  \begin{thm}[Herbrand, Lemme 5, \cite{Corpslocaux}] \label{Théorème de Herbrand}
  Soit $K\subset L\subset M$ une tour d'extensions de corps telles que $L/K$ et $M/K$ soient galoisiennes. Notons $G$ le groupe de Galois de $M/K$ et $H$ celui de $M/L$. Alors, pour tout $v\geq 0$, on a $\Gal(L/K)_u\simeq (G/H)_v=G_uH/H$ où $u=\psi_{M/L}(v)$.
 \end{thm}  
   
   Ce théorème va nous permettre de calculer les sauts d'une tour d'extensions en fonction des sauts de chaque étage, dès lors que chaque étage possède un unique saut et que ces sauts forment une suite strictement croissante. 
   
\begin{lmm} \label{Calcul des sauts pour une extension de trois corps, toutes galoisiennes}
Soit $K\subset L\subset M$ une tour d'extensions finies. Supposons que $M/K$ et $L/K$ soient galoisiennes. Notons $t_1<\dots<t_n$ les sauts de $M/L$. Supposons que $L/K$ possède un unique saut égal à $t$ et que $t<t_1$. Alors, les sauts de $M/K$ sont $t,t_1,\dots,t_n$. De plus, pour tout $j\geq t+1$, on a $\Gal(M/K)_j=\Gal(M/L)_j$.
\end{lmm}

\begin{proof}
Posons $G=\Gal(M/K)$ et $H=\Gal(M/L)$.
Comme $t<t_1$, il s'ensuit que $\psi_{M/L}(x)=x$ pour tout $x\leq t+1$ (car $\phi_{M/L}(x)=x$ si $x\leq t_1$).
Ainsi, par le théorème \ref{Théorème de Herbrand}, on en déduit que $G/H=(G/H)_t=G_tH/H$. 
D'après le second théorème d'isomorphisme, il s'ensuit que $G/H\simeq G_t/(G_t\cap H)$. 
Or, $H\cap G_t=H_t$ et $H_t=H$ puisque $t<t_1$.
Ainsi, $G/H\simeq G_t/H$ et donc $G=G_t$ puisque $\#G=\#G_t$ et $G_t\subset G$.

  Comme $L/K$ possède un unique saut, on a, toujours d'après le théorème \ref{Théorème de Herbrand} :
\begin{equation*}
\{1\}=\Gal(L/K)_{t+1}\simeq G_{t+1}H/H.
\end{equation*}
  Ceci montre que $G_{t+1}\subset H\neq G$. 
  On en conclut que $t$ est le premier saut de $G$.
  De plus, pour tout $i\geq t+1$, on a $H_i=G_i \cap H=G_i$, ce qui permet de montrer que les autres sauts de $G$ sont $t_1,\dots,t_n$ et la seconde partie du lemme. 
\end{proof}

\begin{cor} \label{utile pour déterminer la suite des groupes de rami}
Soient $n\geq 2$ et $K_1\subset \dots\subset K_n$ une tour d'extensions finies. On suppose que pour tout $i$, $K_i/K_1$ est galoisienne et $K_{i+1}/K_i$ admet un unique saut égal à $t_i$. 
Si $t_1< \dots < t_{n-1}$ alors, les sauts de $K_n/K_1$ sont $t_1,\dots,t_{n-1}$ et $\Gal(K_n/K_1)_{t_i}=\Gal(K_n/K_i)$ pour tout $i$ .
\end{cor}

\begin{proof}
Supposons que pour un entier $l\leq n-1$, les sauts de $K_n/K_l$ soient $t_l,\dots,t_{n-1}$ et que pour tout $j\geq l$, on ait $\Gal(K_n/K_l)_{t_j}=\Gal(K_n/K_j)$ (c'est notre hypothèse si $l=n-1$). 
Alors, en appliquant le Lemme \ref{Calcul des sauts pour une extension de trois corps, toutes galoisiennes} à $M=K_n$, à $L=K_l$ et à $K=K_{l-1}$, on en déduit, car  $t_{l-1}<t_l$, que les sauts de $K_n/K_{l-1}$ sont les $t_{l-1},\dots,t_{n-1}$ et que 
\begin{equation*}
\Gal(K_n/K_{l-1})_{t_j} =\Gal(K_n/K_l)_{t_j}=\Gal(K_n/K_j)
\end{equation*}
 pour tout $j\geq l$.
De plus, il est clair que $\Gal(K_n/K_{l-1})_{t_{l-1}}=\Gal(K_n/K_{l-1})$.
Par une récurrence descendante, on en déduit le corollaire.
\end{proof} 

\`A partir de maintenant, on utilisera les notations suivantes. 
Pour un corps $K$, on notera $v_K$ la valuation normalisée sur $K$, $\mathfrak{p}_K$ son idéal premier et $K\{d\}$ l'unique extension non ramifiée de $K$ de degré $d$.
Pour une extension finie $L/K$ (non nécessairement galoisienne), on notera $\mathcal{D}_{L/K}$ la différente de cette extension. 
Si $L/K$ est galoisienne, il existe un lien entre sa différente et le cardinal de ses groupes de ramification. 
Ce lien est résumé dans la proposition ci-dessous. 

\begin{prop} (\cite[Chapitre IV,Proposition 4]{Corpslocaux}) \label{Formule de la valuation}
Pour toute extension galoisienne finie $L/K$ de groupe de Galois $G$, on a: 
\begin{equation*}
v_L(\mathcal{D}_{L/K})=\sum_{i=0}^{\infty} (\# G_i -1).
\end{equation*}
\end{prop}

En particulier, si $L/K$ est une extension galoisienne, totalement ramifiée de degré $d$ et qu'elle possède un unique saut $t$, alors
\begin{equation} \label{calcul de la différente 2}
v_L(\mathcal{D}_{L/K})=(d-1)(t+1).
\end{equation}

On va maintenant rappeler quelques résultats bien connus sur les extensions radicales et principalement sur les extensions de Kummer d'exposant $p$.
Commençons par étudier le saut d'une extension cyclique.

\begin{defn}
Soient $K$ un corps local et $\alpha\in K\backslash K^p$. Notons

\begin{enumerate}[i)]
 \item Pour tout entier $i\geq 0$, $U_K^{(i)}=\{x\in \mathcal{O}_K\vert\; x\equiv 1 \mod \mathfrak{p}_K^i\}$;
\item $c_{K}(\alpha)=\sup\{i\in\N^*\vert\; \exists x\in K^*, \alpha x^{-p}\in U_K^{(i)}\}$.
\end{enumerate}
\end{defn}
Avec cette définition, il est immédiat que $c_K(\alpha x^{-p})=c_K(\alpha)$, pour tout $x\in K^*$.

\begin{thm} [Hecke, Theorem 10.2.9,                  \cite{Advancedtopicsincomputationalnumbertheory}] \label{Hecke}
Soient $K$ un corps local tel que $\zeta_p\in K$ et $\alpha\in K\backslash K^p$. Alors,
\begin{equation*}
v_{K(\sqrt[p]{\alpha})}(\mathcal{D}_{K(\sqrt[p]{\alpha})/K})=(p-1)\left(\frac{pe(K\vert \Q_p)}{p-1}-c_K(\alpha)+1\right).
\end{equation*}

\end{thm}

En combinant \eqref{calcul de la différente 2} avec le théorème \ref{Hecke}, on en déduit que

\begin{cor} \label{Ma version du Théorème de Hecke}
Soient $K$ un corps local tel que $\zeta_p\in K$ et $\alpha\in K\backslash K^p$. Alors, le (unique) saut de $K(\sqrt[p]{\alpha})/K$ vaut 
\begin{equation*}
\frac{pe(K\vert \Q_p)}{p-1}-c_K(\alpha).
\end{equation*}
\end{cor}

\begin{prop} \label{croissance des suites}
Soient $K/\Q_p$ une extension finie telle que $\zeta_p\in K$ et $\alpha\in K\backslash K^p$.
Pour tout entier $s\geq 1$, fixons une racine $p^s$-ième de $\alpha$ que l'on note $\alpha^{1/p^s}$.
Notons également $K_s=K(\alpha^{1/p^s})$ et $t_s$ le saut de $K_s/K_{s-1}$. 
Si $K_s/K_{s-1}$ est ramifiée pour tout $s\geq 1$ alors, la suite $(t_s)_{s\geq 1}$ est strictement croissante.
\end{prop}

\begin{proof}
Notons $\phi_s=\phi_{K_s/K_{s-1}}$ et $c_s=c_{K_s}(\alpha^{1/p^s})$.
Soit $y\in K_s$ tel que 
\begin{equation*}
\alpha^{1/p^s}y^{-p}\in U_{K_s}^{(c_s)}.
\end{equation*}
 En passant à la norme $N_s:=N_{K_s/K_{s-1}}$, il s'ensuit que 
\begin{equation*} 
 \alpha^{1/p^{s-1}}N_s(y)^{-p}\in N_s(U_{K_s}^{c_s}).
 \end{equation*}
D'après \cite[Proposition 4, Chapitre V]{Corpslocaux}, on a 
\begin{equation*}
\alpha^{1/p^{s-1}}N_s(y)^{-p}\in U_{K_{s-1}}^{(\lfloor\phi_s(c_s)\rfloor)}
\end{equation*}
 et donc $\lfloor\phi_s(c_s)\rfloor\leq c_{s-1}$.
Par le corollaire \ref{Ma version du Théorème de Hecke}, 
\begin{align} \label{première version du théorème de Hecke}
 t_s & =\frac{p}{p-1}e(K_{s-1}\vert\Q_p)-c_{s-1} \\ \label{seconde version du théorème de Hecke}
 t_{s+1}& =\frac{p}{p-1}e(K_s\vert\Q_p)-c_s. 
\end{align} 

 Si $c_s\leq t_s$ alors, $\phi_s(c_s)=c_s$ et donc $c_s\leq c_{s-1}$. 
 Comme $e(K_s\vert\Q_p)\geq e(K_{s-1}\vert\Q_p)$, il s'ensuit de \eqref{première version du théorème de Hecke} et \eqref{seconde version du théorème de Hecke} que $t_{s+1}> t_s$. 
 
 Si $c_s> t_s$ alors, $\phi_s(c_s)=t_s+(c_s-t_s)/p$.
 D'où $t_s+(c_s-t_s)/p< c_{s-1}+1$ ou encore 
\begin{equation*}
c_s< t_s+p(c_{s-1}+1-t_s).
\end{equation*}

En utilisant \eqref{première version du théorème de Hecke} et \eqref{seconde version du théorème de Hecke}, on a, après quelques calculs, que 
\begin{equation*}
t_{s+1}+p\left( \frac{pe(K_{s-1}\vert\Q_p)}{p-1}-\frac{e(K_s\vert\Q_p)}{p-1}\right) > (2p-1)t_s-p .
\end{equation*}
Or, $e(K_s\vert\Q_p)=e(K_s\vert K_{s-1})e(K_{s-1}\vert\Q_p)$ et $e(K_s\vert K_{s-1})=p$ par hypothèse.
 Comme $t_s\geq 1$, on en déduit que $(2p-1)t_s-p>t_s$, ce qui prouve la proposition.

\end{proof}

Si $p\neq 2$, il n'est pas nécessaire de supposer que $K_{s+1}/K_s$ soit ramifiée pour tout $s\geq 0$. 
Grâce au lemme ci-dessous, il suffit de supposer que $K_1/K_0$ soit ramifiée.

\begin{lmm} \label{Si c'est ramifié au 1er étage alors c'est tout le temps ramifié}
Soient $K$ un corps local et $\alpha\in K\backslash K^p$. Si  $p\neq 2$ et si $K(\alpha^{1/p})/K$ est ramifiée alors, l'extension $K(\alpha^{1/p^s})/K(\alpha^{1/p^{s-1}})$ est ramifiée pour tout $s\geq 1$.
\end{lmm}

\begin{proof}
Supposons qu'il existe $s\geq 1$ tel que $K(\alpha^{1/p^s})/K(\alpha^{1/p^{s-1}})$ ne soit pas ramifiée. 
L'extension $K(\alpha^{1/p^s})/K$ n'est donc pas totalement ramifiée.
Comme $p\neq 2$, les sous-corps de cette extension sont les $K(\alpha^{1/p^f})$ avec $f=0,\dots,s$ (cf \cite[Théorème 2.1]{Thelatticeofsubfieldsofaradicalextension}).
Il s'ensuit que $K(\alpha^{1/p})/K$ est non-ramifiée, ce qui est absurde.
\end{proof}

Nous allons maintenant étudier les sauts d'une extension galoisienne de groupe de Galois $(\Z/p\Z)^n$ avec $n\geq 2$. 

\begin{prop} \label{calcul de la différente.}
Soient $K_1$ et $K_2$ deux corps locaux et $K=K_1\cap K_2$. On suppose que $K_1/K$ et $K_2/K$ soient galoisiennes, que $K_1K_2/K$ soit totalement ramifiée et que $K_1/K,\; K_2/K,\; K_1K_2/K_2$ et $K_1K_2/K_1$ possèdent toutes un unique saut que l'on notera respectivement $t_1,\;t_2,\; t'_1$ et $t'_2$. Notons $d_i=[K_i : K]$. Alors, 

\begin{enumerate} [i)]
\item $d_2 (d_1-1)t_1+(d_2-1)t'_2 =d_1(d_2-1)t_2+(d_1-1)t'_1$. 
\item De plus, si $d_1=d_2=p$ et $t'_1< t'_2$ alors, 
\begin{equation*}
\begin{cases}
t'_1=t_1 \\
t'_2=p t_2+(1-p)t_1.
\end{cases}
\end{equation*}

\end{enumerate}

\end{prop}

\begin{proof}

$i)$ Comme $K_1K_2/K$ est totalement ramifiée, on déduit de la formule sur la composée de deux différentes (cf \cite[Chapitre III, Proposition 8]{Corpslocaux}) que
\begin{align*} 
v:= v_{K_1K_2}(\mathcal{D}_{K_1K_2/K}) & = & v_{K_1K_2}(\mathcal{D}_{K_1K_2/K_1})+[K_1K_2 : K_1]v_{K_1}(\mathcal{D}_{K_1/K})\\ 
& = & v_{K_1K_2}(\mathcal{D}_{K_1K_2/K_2})+[K_1K_2 : K_2]v_{K_2}(\mathcal{D}_{K_2/K}).
\end{align*} 
Comme $[K_1K_2: K_1]=d_2$ et $[K_1K_2 : K_2]=d_1$, on a, en utilisant quatre fois \eqref{calcul de la différente 2}, que: 
\begin{equation} \label{je sais pas quoi mettre} 
\begin{aligned}
v & = (d_2-1)(t'_2+1)+d_2(d_1-1)(t_1+1) \\
 & =(d_1-1)(t'_1+1)+d_1(d_2-1)(t_2+1).
 \end{aligned}
\end{equation}
Le $i)$ s'ensuit en comparant ces égalités.  \\

$ii)$   Supposons maintenant que $d_1=d_2=p$ et que $t'_1<t'_2$.
   Comme $\Gal(K_1K_2/K)$ est d'ordre $p^2$, il ne peut avoir que deux sauts au plus. 
   De plus, 
\begin{equation*}   
   G=\Gal(K_1K_2/K)=\Gal(K_1K_2/K_1)\Gal(K_1K_2/K_2).
   \end{equation*}
    Ainsi, on obtient que $G_{t'_1}=G$, $G_{t'_1+1}=\Gal(K_1K_2/K_1)$ et que $G_{t'_2+1}=\{id\}$. 
    Les sauts de $K_1K_2/K$ sont donc $t'_1$ et $t'_2$ puisqu'ils sont distincts par hypothèse. 
   En utilisant la proposition \ref{Formule de la valuation} avec $L=K_1K_2$, on en déduit que 
\begin{equation*}
\begin{aligned}
v & =(p^2-1)(t'_1+1)+(p-1)(t'_2-t'_1)\\
& = (p-1)(p t'_1+p+1+t'_2).
\end{aligned}
\end{equation*}
Par ailleurs, \eqref{je sais pas quoi mettre} s'écrit (en factorisant par $d_1-1=d_2-1=p-1$) 
\begin{align*}
v & = (p-1)(t'_2+1+p(t_1+1)) \\
 & = (p-1)(t'_1+1)+p(t_2+1)).
\end{align*} 
 
   Il en résulte, en comparant ces trois relations, que $p t_1+t'_2 = p t_2+t'_1 = p t'_1+t'_2$. 
   Il s'ensuit que $t_1=t'_1$.
  De plus, $t'_2=pt_2+t'_1-p t'_1=pt_2+(1-p)t_1$, ce qui montre le $ii)$ et donc la proposition.
\end{proof}

Le résultat ci-dessous n'est qu'un cas particulier de la proposition \ref{calcul de la différente.}, mais qui se révèle être le point clé pour prouver le théorème \ref{calcul des sauts, cas où la valuation est divisible par p} et le théorème \ref{calcul des sauts, cas où la valuation n'est pas divisible par p}.
 Pour la proposition suivante, on considère $F$ et $a_1,\dots,a_n$ vérifiant les conditions de l'hypothèse \ref{hypothèse 3}. 
  
  \begin{prop} \label{calcul de la différente version appliquée.}
 Soient $r,s_1,\dots,s_n$ des entiers positifs avec $r\geq 1$. 
 Soient $l,k\in\{1,\dots,n+1\}$ distincts.
 Supposons $r\geq 2$ si $l=1$ ou $k=1$.
 Supposons également que $s_{l-1},s_{k-1}\geq 1$ si $l,k\geq 2$. Alors,
  \begin{enumerate} [i)]
  \item on a   
  \begin{multline*}
  t_{n,l}(r,s_1,\dots,s_n)-t_{n,k}(r,s_1,\dots,s_n) \\
  = p \big(t_{n,l}(r(k),s_1(k),\dots,s_n(k)) -t_{n,k}(r(l),s_1(l),\dots,s_n(l))\big).
  \end{multline*}
  
 \item  
  De plus, si $ t_{n,l}(r,s_1,\dots,s_n)< t_{n,k}(r,s_1,\dots,s_n)$ alors, 
 \begin{equation*}  
  t_{n,l}(r,s_1,\dots,s_n)=t_{n,l}(r(k),s_1(k),\dots,s_n(k)). 
  \end{equation*}
  
  \end{enumerate}
  \end{prop}

\begin{proof}
$i)$ Appliquons la proposition \ref{calcul de la différente.} à $K_1=L_{r(k),s_1(k),\dots,s_n(k)}$ et à $K_2=L_{r(l),s_1(l),\dots,s_n(l)}$.
Il est clair que $K_1K_2/K_1 \cap K_2$ est totalement ramifiée d'après le $iv)$ de l'hypothèse \ref{hypothèse 3}.
Comme $l\neq k$ et $r\geq 2$ si $l=1$ ou $k=1$, il en résulte que $[K_i : K_1\cap K_2]=p$ pour $i=1,2$.
Avec les notations de la proposition \ref{calcul de la différente.}, on a $d_1=d_2=p$.
On a également 
\begin{equation*}
\begin{cases}
t_1=t_{n,l}(r(k),s_1(k),\dots,s_n(k)),\; t_2=t_{n,k}(r(l),s_1(l),\dots,s_n(l)) \\
t'_1=t_{n,l}(r,s_1,\dots,s_n), \; t'_2=t_{n,k}(r,s_1,\dots,s_n).
\end{cases}
\end{equation*} 
D'après le $i)$ de la proposition \ref{calcul de la différente.}, $pt_1+t'_2=pt_2+t'_1$ ou encore $t'_2-t'_1=p(t_2-t_1)$, qui correspond à la première formule souhaitée. \\

$ii)$ Supposons maintenant que $t_{n,l}(r,s_1,\dots,s_n)< t_{n,k}(r,s_1,\dots,s_n)$, i.e. $t'_1<t'_2$.
Comme $d_1=d_2=p$, on déduit du $ii)$ de la proposition \ref{calcul de la différente.} que $t'_1=t_1$, qui correspond à la seconde formule souhaitée. 
\end{proof}

Remarquons que le $i)$ de la proposition \ref{calcul de la différente.} appliqué à $d_1=d_2=p$ montre que $t'_1=t'_2$ si et seulement si $t_1=t_2$. 
Le $ii)$ de la proposition \ref{calcul de la différente.} permet donc de calculer les sauts d'une extension galoisienne $L/K$ de groupe de Galois $(\Z/p\Z)^2$ dès lors que l'on puisse trouver deux sous-extensions de degré $p$ sur $K$ ayant chacune un saut distinct.

Le lemme suivant permet de calculer les sauts d'une extension galoisienne $L/K$ de groupe de Galois $(\Z/p\Z)^n$ dès lors que toutes les sous-extensions de degré $p$ sur $K$ ont le même saut. 
On a 

\begin{lmm} \label{Unicité du saut, version théorique}
Soient $n\geq 2$ et $L/K$ une extension abélienne totalement ramifiée de groupe de Galois $\left(\Z/p\Z\right)^n$. 
On suppose que toute sous-extension de degré $p$ sur $K$ ait le même saut $t_0$.
Alors, $L/K$ admet un unique saut égal à $t_0$.
\end{lmm}

\begin{proof}
Dans cette preuve, $K_i$ et $K'_i$ désigneront des sous-corps de $L$ de degré $p^i$ sur $K$. 
Montrons par récurrence sur $l\geq 1$ l'hypothèse $(H)$ suivante : pour tout corps $M$ inclus dans $L$ et de degré $p^l$ sur $K$, l'extension $M/K$ possède un unique saut égal à $t_0$. 

Par hypothèse, $(H)$ est vraie pour $l=1$.
Supposons maintenant que $(H)$ soit vraie jusqu'à un rang $l-1$ (avec $l\geq 2$) et montrons qu'elle l'est encore au rang $l$. 

Soient $K_l$ un corps et $K_{l-1}\subset K_l$. 
Notons $t$ le saut de $K_l/K_{l-1}$.
On souhaite montrer que $t=t_0$.
Soit $K_{l-2}\subset K_{l-1}$.
Comme  
\begin{equation*}
G:=\Gal(K_l/K_{l-2})\simeq \Z/p\Z\times\Z/p\Z,
\end{equation*}
 on en déduit qu'il existe $K'_{l-1}$ tel que $K_l=K_{l-1}K'_{l-1}$ et $K_{l-2}=K_{l-1}\cap K'_{l-1}$. 
Notons $t'$ le saut de $K_l/K'_{l-1}$.
Par hypothèse de récurrence, $t_0$ est à la fois le saut de $K_{l-1}/K_{l-2}$ et de $K'_{l-1}/K_{l-2}$. 
D'après le $i)$ de la proposition \ref{calcul de la différente.} appliqué à $K_1=K_{l-1}$ et à $K_2=K'_{l-1}$ (on a donc $d_1=d_2=p$ et $t_1=t_2=t_0$), on en déduit que $t=t'$.
Ainsi, $t$ est le saut de toute extension de la forme $K_l/K'_{l-1}$ avec $K'_{l-1}\supsetneq K_{l-2}$.
Cela signifie que $t$ est le plus grand saut de $G$.
Enfin, comme $\Gal(K_l/K_{l-2})=\Gal(K_l/K_{l-1})\Gal(K_l/K'_{l-1})$, on en déduit que $t$ est le plus petit saut de $G$.
Par conséquent, $t$ est l'unique saut de $G$.

Si $t_0< t$ alors, les sauts de $G$ sont $t_0$ et $t$ d'après le Lemme \ref{Calcul des sauts pour une extension de trois corps, toutes galoisiennes}, ce qui contredit l'unicité du saut.
  Ainsi, $t\leq t_0$. 
Le théorème \ref{Théorème de Herbrand} appliqué à $M=K_l$, à $L=K_{l-1}$ et à $K=K_{l-2}$ montre que $t=\psi_{K_l/K_{l-1}}(t_0)$.
 Ainsi, $\psi_{K_l/K_{l-1}}(t_0)\leq t_0$.
  Or, $\psi_{K_l/K_{l-1}}(x)\geq x$ pour tout $x\geq 0$. 
  Par conséquent, $\psi_{K_l/K_{l-1}}(t_0)= t_0$ et donc $t=t_0$.
  On a ainsi montré que pour tout $K_{l-1}\subset K_l$, le saut de $K_l/K_{l-1}$ est $t_0$.  
   
  Comme le dernier groupe de ramification non trivial de $\Gal(K_l/K)$ contient un groupe d'ordre $p$, il s'ensuit que $t_0$ est le dernier saut de $K_l/K$.
   Le théorème \ref{Théorème de Herbrand} appliqué à $M=K_l$, à $L=K_{l-1}$ et à $K=K$ montre que le plus petit saut de $M/K$ est au moins égal à $\psi_{M/L}(t_0)\geq t_0$ (par hypothèse de récurrence, $t_0$ est l'unique saut de $K_{l-1}/K$). 
   Ainsi, $t_0$ est l'unique saut de $K_l/K$, ce qui prouve le lemme.
    
    \end{proof}

Le lemme suivant donne la liste des sous-extensions de degré $p$ d'une extension de Kummer. Pour un corps $K$, la classe d'un élément $x\in K$ dans $K^\times/(K^\times)^p$ sera notée $[x]_K$.

\begin{lmm} \label{sous-corps d'une extension kumérienne.}
Soient $r\geq s\geq 1$ des entiers, $K$ un corps local contenant $\zeta_{p^r}$ et $a_1,\dots,a_n\in\mathcal{O}_K\backslash\mathcal{O}_K^p$ tels que 
\begin{equation} \label{indépendance des classes}
\forall i\geq 2, \; [a_i]_K \notin \langle [a_1]_K,\dots,[a_{i-1}]_K\rangle.
\end{equation}

Alors, les sous-corps intermédiaires de degré $p$ sur $K$ de l'extension abélienne $L:=K(a_1^{1/p^s},\dots,a_n^{1/p^s})/K$ sont de la forme $K(\gamma^{1/p})$ où $\gamma=c^p\prod_{j=1}^n a_j^{x_j}$ avec $c\in K$ et $x_j\in\{0,\dots,p-1\}$ des entiers non tous nuls.
\end{lmm}

\begin{proof}
Notons $\tilde{x}$ la classe de $x\in K^\times$ dans $K^\times/(K^\times)^{p^s}$.
D'après \cite[Theorem 10.2.5]{Advancedtopicsincomputationalnumbertheory}, on a 
\begin{equation*}
\Gal(L/K)\simeq \langle \widetilde{a_1},\dots,\widetilde{a_n}\rangle.
\end{equation*}
Soit $M$ une sous-extension de $L/K$ de degré $p$ sur $K$.
Alors, $M/K$ est une extension cyclique.
Toujours d'après \cite[Theorem 10.2.5]{Advancedtopicsincomputationalnumbertheory}, $M=K(\gamma^{1/p})$ pour un certain $\gamma\in K$.
En particulier, $\gamma\in L^p$ ou encore $d:=\gamma^{p^{s-1}}\in L^{p^s}\cap K$.
Ceci montre que $L=K(a_1^{1/p^s},\dots,a_n^{1/p^s}, d^{1/p^s})$.
En utilisant de nouveau \cite[Theorem 10.2.5]{Advancedtopicsincomputationalnumbertheory}, on a 
\begin{equation*}
\Gal(L/K)\simeq \langle \widetilde{a_1},\dots,\widetilde{a_n}, \widetilde{d}\rangle.
\end{equation*}
Par conséquent, $\langle \widetilde{a_1},\dots,\widetilde{a_n}\rangle=\langle \widetilde{a_1},\dots,\widetilde{a_n}, \widetilde{d}\rangle$.
Il existe donc des entiers $z_1,\dots,z_n\in\{0,\dots, p^s-1\}$ non tous nuls tels que $\widetilde{d}=\prod_{i=1}^n \widetilde{a_i}^{z_i}$.
Il existe donc $c\in K$ tel que $\gamma^{p^{s-1}}=c^{p^s} \prod_{i=1}^n a_i^{z_i}$.
On remarque que le cas $s=1$ donne ce que l'on souhaite.
Supposons donc $s>1$.
Pour tout entier $i$, notons $z_i=x_i p^{s-1}+r_i$ la division euclidienne de $z_i$ par $p^{s-1}$.
Remarquons que $x_i<p$.
Ainsi, $\gamma=  \zeta_{p^{s-1}}c^p \prod_{i=1}^n a_i^{x_i} \prod_{i=1}^n a_i^{r_i/p^{s-1}}$ pour une certaine racine $p^{s-1}$-ième de l'unité $\zeta_{p^{s-1}}$.
Comme $\zeta_{p^{s-1}},\gamma\in K$, on en déduit que $\prod_{i=1}^n a_i^{r_i/p^{s-1}}\in K$ et donc $\prod_{i=1}^n a_i^{r_i}\in K^p$ (car $s\geq 2$).
Il s'ensuit que $\prod_{i=1}^n [a_i]_K^{r_i}=[1]_K$. 
De \eqref{indépendance des classes} et du fait que $a_i\notin \mathcal{O}_K^p$, on en déduit que $r_i=0$ pour tout $i$, ce qui montre le lemme puisque $\zeta_{p^{s-1}}c^p=(\zeta_{p^s}c)^p$.
\end{proof}
 
 Reprenons les notations du Lemme \ref{sous-corps d'une extension kumérienne.}. 
 Il est clair que l'on peut supposer $c=1$. 
 Ainsi, les sous-extensions de $L/K$ de degré $p$ sur $K$ sont de la forme $K(\gamma^{1/p})$ avec $\gamma=\prod_{j=1}^n a_j^{x_j}$ où $x_j\in\{0,\dots,p-1\}$ sont des entiers non tous nuls.
 
Le dernier lemme de cette section donne une condition suffisante afin que la condition \eqref{indépendance des classes} soit vérifiée. 

\begin{lmm} \label{lmm que j'ai rajouter sur la condition suffisante}
 Soient $s_1,\dots,s_n\geq 1$ des entiers et $K$ un corps contenant $\zeta_{p^s}$ avec $s=\max\{s_1,\dots,s_n\}$. 
 Soient $\alpha_1,\dots,\alpha_n \in K$ tels que \[[K(\alpha_1^{1/p^{s_1}}, \dots, \alpha_n^{1/p^{s_n}}) : K]=p^{\sum_{i=1}^n s_i}.\]
 Alors, $[\alpha_i]_K\notin \langle [\alpha_1]_K,\dots,[\alpha_{i-1}]_K\rangle$ pour tout $i\geq 2$, 
\end{lmm}
 
\begin{proof}
Supposons par l'absurde que $[\alpha_i]_K\in \langle [\alpha_1]_K,\dots,[\alpha_{i-1}]_K\rangle$ pour un certain entier $i$. 
Alors, $\alpha_i=c^p \prod_{j=1}^{i-1} \alpha_j^{x_j}$ pour un certain $c\in K$ et certains $x_1,\dots,x_{i-1}\in\N$.
Notons $k=\min\{s_1,\dots,s_i\}$.
Alors, $\alpha_k^{x_k}=c^{-p}\alpha_i\prod_{\underset{j\neq k}{j=1}}^{i-1} \alpha_j^{-x_j}$.
De plus, il est clair que \[K\left(\alpha_1^{1/p^{s_1}}, \dots, \alpha_n^{1/p^{s_n}}\right)=K\left(\alpha_1^{1/p^{s_1}},\dots,\alpha_{k-1}^{1/p^{s_{k-1}}},c^{1/p^{s_k-1}}, \alpha_{k+1}^{1/p^{s_{k+1}}},\dots,\alpha_n^{1/p^{s_n}}\right).\]
Par conséquent, $[K(\alpha_1^{1/p^{s_1}}, \dots, \alpha_n^{1/p^{s_n}}) : K]\leq  p^{-1+\sum_{i=1}^n s_i}$, ce qui est absurde par hypothèse. 
\end{proof}
 
\section{Réduction.} \label{section avec tous les calculs chiants}

On souhaite calculer les groupes de ramification d'une extension radicale, galoisienne et finie $L/F$ de groupe de Galois d'exposant une puissance de $p$. 
Cela signifie que $L=L_{r,s_1,\dots,s_n}$ pour certains entiers positifs $r,s_1,\dots,s_n$ non nuls tels que $r\geq\max\{s_1,\dots,s_n\}$ et pour certains $a_1,\dots,a_n\in F$. 
 Si $p=2$, et uniquement dans ce cas, on supposera que 
\begin{equation} \label{hypothèse de réduction} 
 [L_{r,r,\dots,r} : F(\zeta_{2^r})]=2^{nr}.
 \end{equation}

Dans cette section, on se propose de montrer que l'on peut se réduire au cas où $F$ et les $a_1,\dots,a_n$ vérifient les conditions de l'hypothèse \ref{hypothèse 3}, que l'on rappelle pour la commodité du lecteur.

\begin{hyp*} 
  \begin{enumerate} [i)]
  \item $a_1,\dots,a_n\in\mathcal{O}_F\backslash\mathcal{O}_F^p$;
  \item $[L_{r,s_1,\dots,s_n} : F(\zeta_{p^r})]=p^{\sum_{i=1}^n s_i}$;
  \item $p\mid v_p(a_i)$ pour $i=1,\dots,n-1$;
  \item $L_{r,s_1,\dots,s_n}/F$ est totalement ramifiée.
  \end{enumerate}
 \end{hyp*}

Tout d'abord, il est clair que l'on peut se réduire au cas où $a_i\in \mathcal{O}_F\backslash\mathcal{O}_F^p$ pour tout $i\geq 1$, ce qui correspond au $i)$ de l'hypothèse \ref{hypothèse 3}.
Par la théorie de Kummer, on peut également se réduire au cas où $[a_i]_F\notin \langle [a_1]_F, \dots,[a_{i-1}]_F\rangle$ pour tout $i\geq 2$.

Donnons une nouvelle formulation de l'hypothèse $[a_i]_F\notin \langle [a_1]_F, \dots,[a_{i-1}]_F\rangle$.
Pour cela, on aura besoin du lemme ci-dessous : 

\begin{lmm}  [Capelli, Chapitre 6, Theorem 9.1 \cite{LangAlgebra}] \label{Capelli}
Soient $K$ un corps et $n\geq 2$. Soit $a\in K$.
Supposons que $a\notin K^p$ pour tout nombre premier $p$ divisant $n$.
De plus, supposons que $a\notin -4 K^4$ si $4\mid n$. 
Alors, $x^n-a$ est irréductible dans $K[x]$. 
\end{lmm}

Ce lemme va nous permettre de montrer le fait suivant : 

\begin{claim} \label{fait}
Si $p\neq 2$ alors, $[F(\zeta_{p^r}, a_1^{1/p^{s_1}},\dots,a_n^{1/p^{s_n}}) : F(\zeta_{p^r})]=p^{\sum_{i=1}^n s_i}$ si et seulement si $[a_i]_F \notin \langle [a_1]_F, \dots, [a_{i-1}]_F\rangle$.
\end{claim}

\begin{proof}
$ \underline{\Rightarrow}$ : C'est précisément le Lemme \ref{lmm que j'ai rajouter sur la condition suffisante} appliqué à $K=F$.
 
 $\underline{\Leftarrow}$ :
Supposons par l'absurde qu'il existe $c\in F(\zeta_{p^r})$ tel que $a_i=c^p\prod_{j=1}^{i-1} a_j^{x_j}$.
Alors, $c^p\in F\cap F(\zeta_{p^r})^p$.
 Par un théorème de Schinzel (cf \cite[Theorem 2]{abelianbinomialspowerresiduesandexponentialcongruences}), si $\alpha\in F$ et $p\neq 2$, alors $\alpha\in F^p$ si et seulement si $\alpha\in F(\zeta_{p^r})^p$.
Ainsi, $c^p\in F^p$.
Donc, $[a_i]_F\in \langle [a_1]_F,\dots,[a_{i-1}]_F\rangle$, ce qui est absurde par hypothèse. 
En conclusion, $[a_i]_{F(\zeta_{p^r})}\notin \langle [a_1]_{F(\zeta_{p^r})},\dots,[a_{i-1}]_{F(\zeta_{p^r})}\rangle$ pour tout $i\geq 2$.
    
De plus, comme $a_i\in F\backslash F^p$, le théorème de Schinzel cité juste avant montre que $a_i\notin F(\zeta_{p^r})^p$ pour tout $i$.
Du Lemme \ref{Capelli}, on a $[F(\zeta_{p^r}, a_i^{1/p^{s_i}}) : F(\zeta_{p^r})]=p^{s_i}$ pour tout $i$.   
 Montrons maintenant que 
\begin{equation} \label{le label du fait}
F(\zeta_{p^r}, a_1^{1/p^{s_1}},\dots,a_{i-1}^{1/p^{s_{i-1}}})\cap F(\zeta_{p^r}, a_i^{1/p^{s_i}})=F(\zeta_{p^r}).
\end{equation} 
 
Comme $\Gal(F(\zeta_{p^r}, a_i^{1/p^{s_i}})/F(\zeta_{p^r}))\simeq \Z/p^{s_i}\Z$, les sous corps de $F(\zeta_{p^r}, a_i^{1/p^{s_i}})$ contenant $F(\zeta_{p^r})$ sont les $F(\zeta_{p^r}, a_i^{1/p^g})$ avec $g\in\{0,\dots,s_i\}$.
Ainsi, l'intersection du \eqref{le label du fait} est égale à $F(\zeta_{p^r}, a_i^{1/p^g})$, pour un certain $g$.
Supposons par l'absurde que $g\neq 0$. 
Notons $s=\max\{s_1,\dots,s_n\}$.
Alors, \[a_i\in F(\zeta_{p^r}, a_1^{1/p^{s_1}},\dots,a_{i-1}^{1/p^{s_{i-1}}})^p\subset F\left(\zeta_{p^r}, a_1^{1/p^s}, \dots,a_{i-1}^{1/p^s}\right)^p.\] 
D'après le Lemme \ref{sous-corps d'une extension kumérienne.}, il existe des entiers $x_1,\dots,x_{i-1}$ non tous nuls et $c\in F(\zeta_{p^r})$ tels que $a_i=c^p\prod_{j=1}^{i-1} a_j^{x_j}$, ce qui est absurde d'après le premier paragraphe de cette preuve.
On a donc $g=0$, ce qui montre \eqref{le label du fait}.

De \eqref{le label du fait}, on en déduit que
\begin{equation*}
[F(\zeta_{p^r}, a_1^{1/p^{s_1}},\dots,a_i^{1/p^{s_n}}) : F(\zeta_{p^r})] = p^{s_n}[F(\zeta_{p^r}, a_1^{1/p^{s_1}},\dots,a_{n-1}^{1/p^{s_{n-1}}}) : F(\zeta_{p^r})].
\end{equation*}
Par une récurrence immédiate, $[F(\zeta_{p^r}, a_1^{1/p^{s_1}},\dots,a_n^{1/p^{s_n}}) : F(\zeta_{p^r})]= p^{\sum_{i=1}^n s_i}$.
On a ainsi montré le fait.
\end{proof}
Si $p\neq 2$, on s'est ainsi réduit au $ii)$ de l'hypothèse \ref{hypothèse 3}.
Si $p=2$, alors, d'après \eqref{hypothèse de réduction}, $[L_{r,r,\dots,r} : F(\zeta_{2^r})]=2^{nr}$.
Par la multiplicativité des degrés, on en déduit que 

\begin{align*}
2^{nr}=[L_{r,r,\dots,r} : F(\zeta_{2^r})] & = [L_{r,r,\dots,r} : L_{r,s_1,\dots,s_n}][L_{r,s_1,\dots,s_n} : F(\zeta_{2^r})] \\
& \leq 2^{\sum_{i=1}^n (r-s_i)} 2^{\sum_{i=1}^n s_i}=2^{nr}.
\end{align*}
 Par conséquent, $[L_{r,s_1,\dots,s_n} : F(\zeta_{2^r})]=2^{\sum_{i=1}^n s_i}$, ce qui correspond au $ii)$ de l'hypothèse \ref{hypothèse 3}.

Montrons maintenant que l'on peut se réduire au $iv)$ de l'hypothèse \ref{hypothèse 3}.
Pour cela, on va calculer l'indice de ramification de $L_{r,s_1,\dots,s_n}/F$ en distinguant le cas $p\neq 2$ du cas $p=2$.
Bien qu'il suffise de le calculer pour $r\geq \max\{s_1,\dots,s_n\}$, on le calculera même pour $r<\max\{s_1,\dots,s_n\}$. 
Cela aura son importance dans la suite (pour montrer le Lemme \ref{calcul de $t(1,1)$.} par exemple).
Commençons tout d'abord par calculer le degré de cette extension.
 
\begin{lmm} \label{degré de l'extension}
Pour tous $r,s_1,\dots,s_n$, on a $[L_{r,s_1,\dots,s_n} : F]= (p-1)p^{r-1+\sum_{k=1}^n s_k}$.
\end{lmm}

\begin{proof}

Comme $F/\Q_p$ est non ramifiée et que $\Q_p(\zeta_p)/\Q_p$ l'est totalement, 
\begin{equation*}
[F(\zeta_{p^r}) : F]=[\Q_p(\zeta_{p^r}): \Q_p]=(p-1)p^{r-1}.
\end{equation*}

Si $r\geq s:=\max\{s_1,\dots,s_n\}$ alors, la conclusion du lemme est claire d'après le $ii)$ de l'hypothèse \ref{hypothèse 3}.
Si $r<s$ alors, d'après la multiplicativité des degrés, 
\begin{equation*}
[L_{s,s_1,\dots,s_n} : F]=[L_{s,s_1,\dots,s_n} : L_{r,s_1,\dots,s_n}][L_{r,s_1,\dots,s_n} : F].
\end{equation*}

De plus, 
\begin{equation*}
 \begin{cases}
 [L_{s,s_1,\dots,s_n} : L_{r,s_1,\dots,s_n}]\leq p^{s-r} \\
 [L_{r,s_1,\dots,s_n} : F]\leq [F(\zeta_p) : F]p^{r-1+\sum_{k=1}^n s_k} \\
 [L_{s,s_1,\dots,s_n} : F]= [F(\zeta_p) : F]p^{s-1+\sum_{k=1}^n s_k}
 \end{cases}.
\end{equation*} 
Ainsi, 
\begin{align*}
(p-1)p^{r-1+\sum_{k=1}^n s_k} & \geq [L_{r,s_1,\dots,s_n} : F]  = \frac{[L_{s,s_1,\dots,s_n} : F]}{[L_{s,s_1,\dots,s_n} : L_{r,s_1,\dots,s_n}]} \\
& \geq  [F(\zeta_p) : F]p^{r-1+\sum_{k=1}^n s_k} = (p-1)p^{r-1+\sum_{k=1}^n s_k},
\end{align*}
 ce qui prouve le lemme.

\end{proof}

Nous allons maintenant étudier l'indice de ramification de $L_{r,s_1,\dots,s_n}/F$. 
Commençons par le cas où $p\neq 2$.
La proposition suivante est une généralisation naturelle de \cite[Theorem 5.5]{viviani}; elle montre que dans ce cas, l'extension $L_{r,s_1,\dots,s_n}/F$ est totalement ramifiée.

\begin{prop} \label{extension totalement ramifiée}
Supposons $p\neq 2$. Alors, pour tous entiers $r,s_1,\dots,s_n$, on a $L_{r,s_1,\dots,s_n}/F$ est totalement ramifiée, i.e. $e(L_{r,s_1,\dots,s_n}\vert F)=(p-1) p^{r-1+\sum_{i=1}^n s_i}$ . 
\end{prop}

\begin{proof}

 Il suffit de le montrer pour $r\geq\max\{s_1,\dots,s_n\}$. 
 Comme $F(\zeta_{p^r})/F$ est totalement ramifiée, il résulte du Lemme \ref{degré de l'extension} que $e(F(\zeta_{p^r})\vert F)=p^{r-1} (p-1)$.

Il suffit donc de montrer que $L_{r,s_1,\dots,s_k}/L_{r,s_1,\dots,s_{k-1}}$ est totalement ramifiée pour tout $k\geq 1$.
 Afin d'alléger les notations, notons $L=L_{r,s_1,\dots,s_{k-1}}$.
 Remarquons que $L=L_{0,s_1,\dots,s_{k-1}}(\zeta_{p^r})$.
 
 Supposons par l'absurde que $L(a_k^{1/p^{s_k}})/L$ ne soit pas totalement ramifiée et notons $p^h$ son degré d'inertie. 
 Comme $L\{p^h\}$ est obtenu en adjoignant à $L$ certaines racines de l'unité (cf \cite[Chapitre 1, section 7]{Algebraicnumbertheory}), il s'ensuit que $L\{p^h\}/L_{0,s_1,\dots,s_{k-1}}$ est une extension abélienne. 
 D'après la contraposée du Lemme \ref{Si c'est ramifié au 1er étage alors c'est tout le temps ramifié} (car $p\neq 2$), on en déduit que $L(a_k^{1/p})\subset L\{p^h\}$ et donc que $a_k\in L\{p^h\}^p$.
 Par conséquent, le corps de décomposition du polynôme $x^p-a_k\in L_{0,s_1,\dots,s_{k-1}}[x]$ est abélien sur $L_{0,s_1,\dots,s_{k-1}}$.
Par un théorème de Schinzel (cf \cite[Theorem 2]{abelianbinomialspowerresiduesandexponentialcongruences}), il s'ensuit, car $p\neq 2$, que $a_k\in L_{0,s_1,\dots,s_{k-1}}^p$, ce qui est absurde. 
  Ainsi, l'extension $L(a_k^{1/p^{s_k}})/L$ est totalement ramifiée pour tout $k\geq 1$, ce qui montre le lemme.
\end{proof}

Traitons maintenant le cas $p=2$. 
Nous allons écrire deux lemmes qui généralisent de manière naturelle \cite[Lemme 20-Lemme 23]{DelCorsoRoosiNormalIntegralbasesandtamenessconditions}.

\begin{lmm} \label{extension non normale}
Soient $b\in F\backslash \pm F^2$ et $\beta\in\overline{F}$ tels que $\beta^4=b$. Alors, $F(\beta)/F$ n'est pas galoisienne.
\end{lmm}

\begin{proof}
Supposons par l'absurde que $F(\beta)/F$ soit galoisienne. 
La condition $b\in F\backslash \pm F^2$ permet d'assurer l'irréductibilité de $x^4-b$ d'après le Lemme \ref{Capelli}. 
Son corps de décomposition est donc $F(\beta,i)$.
Comme $F(\beta)/F$ est galoisienne, cela signifie que $i\in F(\beta)$.
Comme $i\in F(\beta)\backslash F$, on déduit de \cite[Theorem 2.1]{Thelatticeofsubfieldsofaradicalextension} que $F(\beta^2)$ est l'unique corps quadratique sur $F$ contenu dans $F(\beta)$. 
Comme $i$ l'est également, il s'ensuit que $F(\beta^2)=F(i)$ ou encore que $b\in F(i)^2$.

  Posons $b=(\alpha+i\gamma)^2\in F(i)^2$. 
 Alors, $b=\alpha^2-\gamma^2+2i\alpha\gamma$. 
 Comme $i\notin F$, cela signifie que $\alpha\gamma=0$ et donc que $b\in \pm F^2$, ce qui est absurde.
 Ceci montre le lemme.
 
\end{proof}

\`A partir de maintenant, pour une extension $L/K$, on notera $f(L\vert K)$ son degré d'inertie. 

\begin{lmm}
Pour tous $r,s_1,\dots,s_n$, on a $f(L_{r,s_1,\dots,s_n} \vert F) \leq 2$.
\end{lmm}

\begin{proof}
Il suffit de le prouver pour $r\geq s:=\max\{s_1,\dots,s_n\}$. 
Remarquons que le lemme est vérifié si $s=1$ car d'après \cite[Theorem 12]{CapuanoLauraDelCorsoAnoteonupperramificationjumpsinbelianextensionsofexponent}, le degré d'inertie de $L_{r,1,\dots,1}/F(\zeta_{2^r})$ vaut au plus $2$ et que $f(F(\zeta_{2^r})\vert F)=1$.

Supposons donc $s\geq 2$ et supposons par l'absurde que $f(L_{r,s_1,\dots,s_n}\vert F)>2$.
Comme $[L_{r,s_1,\dots,s_n} : F]=2^m$ (pour un entier $m\geq 2$), on en déduit qu'il existe un entier $t\geq 2$ tel que $f(L_{r,s_1,\dots,s_n} \vert F)=2^t$.
Comme $F\{2^t\}/F$ est une extension abélienne, on en conclut que $F\{4\}\subset F\{2^t\}$.
Soit $\alpha\in\overline{F}$ tel que $F\{4\}=F(\alpha)$. 
Intéressons nous maintenant au corps $L_\alpha:=F(\zeta_{2^r}, \alpha)$.
Comme $F(\zeta_{2^r})/F$ est une extension totalement ramifiée et que $F(\alpha)/F$ est non-ramifiée, on en déduit que $L_\alpha/F(\zeta_{2^r})$ est une extension non ramifiée de groupe de Galois $\Z/4\Z$.

Soit $L$ l'unique sous-corps de $L_\alpha$ de degré $2$ sur $F(\zeta_{2^r})$. 
Comme $L_\alpha\subset L_{r,s,\dots,s}$, on déduit du Lemme \ref{sous-corps d'une extension kumérienne.} que $L=F(\zeta_{2^r}, \delta)$ où $\delta=\prod_{j=1}^n a_j^{x_j/2}$ avec $x_j\in\{0,1\}$ des entiers non tous nuls. 
Quitte à permuter les $a_i^{1/p^{s_i}}$, on peut supposer $x_1=1$.

Remarquons que $\delta,a_i\notin L^2$ pour tout $i\geq 1$ (c'est clair pour $\delta$).
En effet, si $a_i\in L^2$, on en déduirait, grâce au Lemme \ref{sous-corps d'une extension kumérienne.}, qu'il existe $c\in F(\zeta_{2^r})$ tel que $a_i=c^2 \delta^2$.
Ainsi, $[a_i]_{F(\zeta_{2^r})}=\prod_{j=1}^n [a_j]_{F(\zeta_{2^r})}^{x_j}$.
Comme $x_1\neq 0$, le Lemme \ref{lmm que j'ai rajouter sur la condition suffisante} appliqué à $K=F(\zeta_{2^r})$ contredit le $ii)$ de l'hypothèse \ref{hypothèse 3}.

Comme $L_{r,s,\dots,s}=F(\delta^{1/2^{s-1}}, a_2^{1/2^s},\dots,a_n^{1/2^s})$, il s'ensuit que $[L_{r,s,\dots,s} : L]=p^{2s-1}$.
D'après le Lemme \ref{lmm que j'ai rajouter sur la condition suffisante}, cela signifie que $[a_i]_L\notin \langle [\delta]_L,[a_2]_L,\dots, [a_{i-1}]_L\rangle$ pour tout $i\geq 2$.
Ainsi, on peut de nouveau appliquer le Lemme \ref{sous-corps d'une extension kumérienne.} à $K=L$.
Comme $L_\alpha/L$ est une extension quadratique et que 
\begin{equation*}
L_\alpha\subset L(\delta^{1/2^{s-1}}, a_2^{1/2^s},\dots,a_n^{1/2^s})\subset L(\delta^{1/2^s}, a_2^{1/2^s},\dots,a_n^{1/2^s}),
\end{equation*}
on en déduit que $L_\alpha=L(\delta^{y/2}\prod_{i=2}^n a_i^{y_i/2})$ pour certains entiers $y,y_i\in\{0,1\}$. 
Remarquons que $y\neq 0$ (sinon $\Gal(L_\alpha/F(\zeta_{2^r}))\simeq \Z/2\Z\times \Z/2\Z$).
Par conséquent, $y=1$ et donc, $L_\alpha=F(\zeta_{2^r}, \beta)$ où $\beta=\delta^{1/2}\prod_{i=2}^n a_i^{y_i/2}$. 
Comme $L_\alpha/F$ est abélienne, on en déduit que $F(\beta)/F$ est abélienne. 
De plus, $\beta^4=\prod_{i=1}^n a_i^{x_i+2y_i}\in F$.
D'après le Lemme \ref{extension non normale}, il suffit de montrer que $\beta^4\notin \pm F^2$ pour obtenir une contradiction.

Clairement, $\beta^4\in \pm F^2$ si et seulement si $\delta^2 \in \pm F^2$.
De plus, si $\delta^2\in\pm F^2$ alors, $\delta\in F(i)\subset F(\zeta_{2^r})$ puisque $r\geq 2$.
Cela contredit le fait que $F(\zeta_{2^r}, \delta)/F(\zeta_{2^r})$ est une extension quadratique.
Ceci prouve le lemme.

\end{proof}

Ainsi, si $f(L_{r,s_1,\dots,s_n}\vert F)=1$ alors, $L_{r,s_1,\dots,s_n}/F$ est totalement ramifié.
Supposons maintenant que $f(L_{r,s_1,\dots,s_n}\vert F)=2$. 
Alors, $f(L_{r,s_1,\dots,s_n}\vert F(\zeta_{2^r}))=2$. 
Le Lemme \ref{sous-corps d'une extension kumérienne.} montre que $F(\zeta_{2^r})\{2\}=F(\zeta_{2^r}, \alpha)$ où $\alpha=\prod_{i=1}^n a_i^{x_i/2}$ pour certains entiers $x_i\in\{0,1\}$ non tous nuls. 
Comme $[F(\alpha) : F]=2$, il s'ensuit que $F\{2\}=F(\alpha)$. 
Quitte à permuter les $a_i^{1/p^{s_i}}$, on peut supposer que $s_1$ est le plus petit $s_i$ tel que $x_i\neq 0$. 
Ainsi, il devient clair que

\begin{equation*}
L_{r,s_1,\dots,s_n}=F\{2\}(\zeta_{2^r},\alpha^{1/2^{s_1-1}},a_2^{1/2^{s_2}},\dots,a_n^{1/2^{s_n}}).
\end{equation*}
De plus, $F\{2\}(\zeta_{2^r},\alpha^{1/2^{s_1-1}},a_2^{1/2^{s_2}},\dots,a_n^{1/2^{s_n}})/F\{2\}$ est totalement ramifiée.
Il est clair que $\alpha\in \mathcal{O}_{F\{2\}}\backslash \mathcal{O}_{F\{2\}}^2$ et que \[F\{2\}(\zeta_{2^r},\alpha^{1/2^{s_1-1}},a_2^{1/2^{s_2}},\dots,a_n^{1/2^{s_n}}) : F\{2\}(\zeta_{2^r})]=2^{s_1-1+\sum_{i=2}^n s_i}.\]
En résumé, quitte à remplacer $F$ par $F\{2\}$ (qui est encore une extension non-ramifiée) et $a_1^{1/2^{s_1}}$ par $\alpha^{1/2^{s_1-1}}$, on a montré que dans tous les cas, on pouvait se réduire au fait que $L_{r,s_1,\dots,s_n}/F$ soit totalement ramifiée, ce qui correspond à la condition $iv)$ de l'hypothèse \ref{hypothèse 3}.

Nous terminons cette section en montrant que l'on peut aussi se réduire à la condition $iii)$ de l'hypothèse \ref{hypothèse 3}. 
Supposons qu'il existe $i$ et $j$ tels que $p\nmid v_F(a_i)$ et $p\nmid v_F(a_j)$.
  Alors, pour tous $s_i,s_j \in\N^*$, 
\begin{equation*}
\begin{cases}
F(a_i^{1/p^{s_i}}, a_j^{1/p^{s_j}})= F(\alpha^{1/p^{s_i}}, a_j^{1/p^{s_j}}) \; \text{si} \; s_i\leq s_j \\
F(a_i^{1/p^{s_i}}, a_j^{1/p^{s_j}})= F(a_i^{1/p^{s_i}}, \alpha^{1/p^{s_j}}) \; \text{si} \; s_j\leq s_i 
\end{cases}
\end{equation*}  
  où $\alpha=a_i^{v_F(a_j)}a_j^{-v_F(a_i)}$.
  En remarquant que $v_F(\alpha)=0$, cela montre que l'on peut se ramener au cas où soit $p\mid v_F(a_i)$, soit $p\mid v_F(a_j)$. 
  Ainsi, sans perte de généralité, on peut supposer que $p\mid v_F(a_1),\dots,v_F(a_{n-1})$, ce qui correspond à la condition $iii)$ de l'hypothèse \ref{hypothèse 3}.

\section{Calcul des $t_{1,k}(r,s)$ avec $k\in\{1,2\}$.} \label{cas n=1}

 Pour un corps local $K$, on notera $\pi_K$ une uniformisante de $K$. 
 Notons également $f$ le degré $[F : \Q_p]$.
 Pour tous $r\geq 1$ et $s\geq 1$, on se propose de calculer $t_{1,1}(r,s)$ et $t_{1,2}(r,s)$, respectivement l'unique saut de $F(\zeta_{p^r},a^{1/p^s})/F(\zeta_{p^{r-1}}, a^{1/p^s})$ et de $F(\zeta_{p^r},a^{1/p^s})/F(\zeta_{p^r}, a^{1/p^{s-1}})$.
Les formules de cette section serviront d'initialisation lorsque l'on fera des récurrences sur le nombre de $a_i$ et permettrons donc de traiter le cas général. 
Rappelons qu'avec nos réductions, $L_{1,s}/F$ est une extension totalement ramifiée de degré $p^s(p-1)$.
Ainsi, d'après la proposition \ref{Ma version du Théorème de Hecke}, on en déduit que pour tout $s\geq 0$,
\begin{equation} \label{première utilisation du thm de Hecke}
t_{1,2}(1,s+1)=p^{s+1}-c_{L_{1,s}}(a^{1/p^s}).
\end{equation}
Posons $c_s=c_{L_{1,s}}(a^{1/p^s})$.
Comme $(a^{1/p^s})^{1-p^f}=a^{1/p^s} (a^{1/p^{s-f+1}})^{-p}$, on a (par définition de $c_K(\alpha)$), 
\begin{equation*}
c_s=c_{L_{1,s}}((a^{1/p^s})^{1-p^f}).
\end{equation*}

Enfin, le cas $F=\Q_p$ a été traité par Viviani dans \cite{viviani}.
Néanmoins, son approche est basée sur le fait que l'on connaisse une uniformisante de $L_{1,s+1}/L_{1,s}$ (cf \cite[Lemme 5.7, Lemme 6.4]{viviani}). 
Cette méthode ne semble pas possible dès lors que l'on remplace $\Q_p$ par $F$, ce qui nous a conduit à utiliser le théorème de Hecke.

\begin{lmm} \label{calcul de $t(1,1)$.}
Soit $s\geq 0$, on a :

\begin{enumerate} [i)]
\item Si $p\mid v_F(a)$ alors, $t_{1,2}(1,s+1)= p^{s+1}-p+1$; 
\item Si $p\nmid v_F(a)$ alors, $t_{1,2}(1,s+1)=p^{s+1}$.
\end{enumerate}

\end{lmm}

\begin{proof}
 $i)$ Supposons d'abord $v_F(a)=0$. 
 Dans ce cas, $a^{1/p^s}\in\mathcal{O}_{L_{1,s}}^\times$.
 Comme $\mathcal{O}_{F(a^{1/p^s})}/\mathfrak{p}_{F(a^{1/p^s})}\simeq \F_{p^f}$, on en déduit que $(a^{1/p^s})^{1-p^f}\equiv 1 \mod \mathfrak{p}_{F(a^{1/p^s})}$.
 Or, $\mathfrak{p}_{F(a^{1/p^s})}=\mathfrak{p}_{L_{1,s}}^{p-1}$ car $L_{1,s}/F(a^{1/p^s})$ est totalement ramifiée de degré $p-1$.
 Par définition de $c_K(\alpha)$, on en déduit que $c_s\geq p-1$ pour tout $s\geq 0$.
 
   Montrons par récurrence sur $s$ que $c_s\leq p-1$.  
Comme $c_0\geq p-1$, on déduit de la proposition \ref{Ma version du Théorème de Hecke} que $t_{1,2}(1,1)\leq 1$.
Or, $t_{1,2}(1,1)\neq 0$ car $L_{1,1}/L_{1,0}$ est totalement ramifiée de degré $p$.
 Il s'ensuit que $t_{1,2}(1,1)=1$ et donc $c_0=p-1$.
 Ceci initialise la récurrence.
 
  Supposons maintenant que $c_s\leq p-1$ et montrons que $c_{s+1}\leq p-1$.
  Notons 
\begin{equation*} 
  x=a^{1/p^{s+1}}.
  \end{equation*}
  Remarquons que $x^{1-p^f} \equiv 1 \mod \mathfrak{p}_{L_{1,s+1}}^{p-1}$.
  Soit $y\in L_{1,s+1}^*$ tel que 
\begin{equation} \label{conditions pour obtenir l'absurdité}  
  z=x^{1-p^f}y^{-p}\in\mathcal{O}_{L_{1,s+1}} \; \text{et} \; z \equiv 1 \mod \mathfrak{p}_{L_{1,s+1}}^{c_{s+1}}.
  \end{equation} 
   Remarquons que $y\in \mathcal{O}_{L_{1,s+1}}^\times$ du fait que $x,\;z\in \mathcal{O}_{L_{1,s+1}}^\times$.
   
   Supposons par l'absurde que $c_{s+1}\geq p$. 
Comme $\mathcal{O}_{L_{1,s+1}}=\Z_p[\pi_{L_{1,s+1}}]$, on peut alors écrire $y=\sum_{i\geq 0} a_i \pi_{L_{1,s+1}}^i$ avec $a_i\in\Z_p$ et $a_0\in\Z_p^\times$ car $y\in\mathcal{O}_{L_{1,s+1}}^\times$.
 Il s'ensuit donc que $y^p\equiv a_0^p \mod \mathfrak{p}_{L_{1,s+1}}^p$.
   Comme $a_0\in\Z_p$, il s'ensuit que $a_0^p\equiv a_0 \mod p$.
   Par conséquent, $y^p\equiv a_0 \mod \mathfrak{p}_{L_{1,s+1}}^p$.
   Comme $z\equiv 1 \mod \mathfrak{p}_{L_{1,s+1}}^p$, on déduit de \eqref{conditions pour obtenir l'absurdité} que $x^{1-p^f}\equiv a_0 \mod \mathfrak{p}_{L_{1,s+1}}^p$. 
   En élevant à la puissance $p$, il s'ensuit que 
   \begin{equation*}
   (x^p)^{1-p^f}a_0^{-p}\equiv 1 \mod \mathfrak{p}_{L_{1,s}}^p.
   \end{equation*}   
   Comme $x^p=a^{1/p^s}$ et $a_0\in\Z_p^\times$, il s'ensuit que $c_s\geq p$, ce qui contredit l'hypothèse de récurrence.    
   Ainsi, $c_{s+1}\leq p-1$.
   On a donc montré que $c_s=p-1$ pour tout $s\geq 0$.
   De \eqref{première utilisation du thm de Hecke}, il en résulte que $t_{1,2}(1,s+1)=p^{s+1}-p+1$.
   Cela termine la preuve de la première affirmation du lemme dans le cas où $v_F(a)=0$.

 Supposons maintenant que $p\mid v_F(a)\neq 0$. 
 Posons $a=p^{v_F(a)}\gamma=\pi_{L_{1,s}}^{(p-1)p^s v_F(a)}\gamma$ avec $v_F(\gamma)=0$ et $y=\pi_{L_{1,s}}^{(p-1) v_F(a)p^{-1}} \beta$. 
 Il est clair que $a^{1/p^s}y^{-p}=\gamma^{1/p^s}\beta^{-p}$ et donc $c_{L_{1,s}}(a^{1/p^s})=c_{L_{1,s}}(\gamma^{1/p^s})=p-1$ puisque $v_F(\gamma)=0$. 
On conclut grâce à \eqref{première utilisation du thm de Hecke}. \\

 $ii)$  Supposons par l'absurde qu'il existe $y\in L_{1,s}^*$ tel que $z=a^{1/p^s}y^{-p}\in\mathcal{O}_{L_{1,s}}$ et $z\equiv 1 \mod \mathfrak{p}_{L_{1,s}}$.
 En particulier, $v_{L_{1,s}}(z)=0$. 
  Comme $z=a^{1/p^s}y^{-p}$, on en déduit que
\begin{equation*}  
  v_{L_{1,s}}(y)=v_{L_{1,s}}(a^{1/p^s}) p^{-1}=(p-1)v_F(a)p^{-1}\notin\Z
\end{equation*}  
 car $p\nmid v_F(a)$, ce qui est absurde. 
  Ceci prouve donc que $c_s=0$. 
  De \eqref{première utilisation du thm de Hecke}, on en déduit que $t_{1,2}(1,s+1)=p^{s+1}$.  
\end{proof}

On peut maintenant prouver les résultats principaux de cette section, à savoir les théorème \ref{cas où n=1 et v p a vaut 0} et \ref{Cas n=1 et v p a vaut 1}. 

\begin{thm} \label{cas où n=1 et v p a vaut 0}
Si $p \mid v_F(a)$ alors, pour tous $r\geq s\geq 0$ et $r\neq 0$, on a: 

\begin{enumerate} [i)]
\item $t_{1,2}(s,r)=p^{r+s-1}-\frac{p-1}{p+1}(p^{2s-1}+1)$ si $s\geq 1$;
\item $t_{1,1}(s,r)=t_{1,1}(s,s-1)$ si $s\geq 1$;
\item $t_{1,2}(r,s)=t_{1,2}(s,s)$ si $s\geq 1$;
\item Si de plus $r>s$ alors, $t_{1,1}(r,s)=p^{r+s-1}-p^{2s}+\frac{p-1}{p+1}(p^{2s}-1)$.
\end{enumerate}

\end{thm}

\begin{proof}
Montrons le théorème pour $s=0$.
 Dans ce cas, $i)$, $ii)$ et $iii)$ sont vides.
 Comme $L_{r,0}=F(\zeta_{p^r})$, alors $t_{1,1}(r,0)$ est le saut de l'extension $F(\zeta_{p^r})/F(\zeta_{p^{r-1}})$ qui vaut $p^{r-1}-1$ (cf \cite[Chapitre 4, §4]{Corpslocaux}), ce qui montre $iv)$.

  Montrons également le théorème pour $s=1$. 
   Tout d'abord, $i)$ découle du Lemme \ref{calcul de $t(1,1)$.} et $ii)$ du fait que $L_{1,r}/L_{0,r}$ soit une extension totalement ramifiée de degré $p-1$.
 Montrons $iii)$ et $iv)$ par récurrence sur $r$. 
 Le cas $r=1$ est trivialement vérifié.
 Supposons que $r$ vérifie $iii)$ et $iv)$ et montrons que $r+1$ aussi.
Par hypothèse de récurrence, on déduit du $iii)$ que $t_{1,2}(r,1)=t_{1,2}(1,1)$. 
 D'après le Lemme \ref{calcul de $t(1,1)$.}, $t_{1,2}(1,1)=1$.
 Ainsi, $t_{1,2}(r,1)=1$.
 De plus, $t_{1,1}(r+1,0)=p^r-1$. 
 D'après le $i)$ de la proposition \ref{calcul de la différente version appliquée.}, 
\begin{align*}
t_{1,2}(r+1,1)-t_{1,1}(r+1,1) & =p(t_{1,2}(r,1)-t_{1,1}(r+1,0)) \\
& = p(2-p^r)<0.
\end{align*} 
 
Le $ii)$ de la proposition \ref{calcul de la différente version appliquée.} montre que $t_{1,2}(r+1,1)=t_{1,2}(r,1)$.
Il s'ensuit donc que $t_{1,2}(r+1,1)=1$. 
Ainsi,
\begin{align*}
t_{1,1}(r+1,1) & =p(p^r-2)+t_{1,2}(r+1,1) \\
& =p^{r+1}-2p+1,
\end{align*}
ce qui montre $iii)$ et $iv)$. 
Cela prouve que le théorème est vrai pour $s=1$.

Notons $\Lambda$ l'ensemble des couples $(r,s)\neq (0,0)$ de $\N^2$ tels que $r\geq s\geq 0$ et vérifiant les affirmations $i)$ à $iv)$ du théorème.
Supposons par l'absurde qu'il existe $(r,s)\notin\Lambda$ avec $r\geq s\geq 0$ et choisissons le minimal pour l'ordre lexicographique.
Par ce qui précède, $s\geq 2$. 
 
Montrons dans un premier temps que $(r,s)$ vérifie $i)$ et $ii)$.
Par minimalité de $(r,s)$, on a que $(r,s-1)\in\Lambda$.
Ainsi, en utilisant $i)$,
\begin{equation*}
t_{1,2}(s-1,r)=p^{r+s-2}-\frac{p-1}{p+1}(p^{2s-3}+1).
\end{equation*}
  Comme $(s,s-1)\in\Lambda$, on a, d'après $iv)$, 
\begin{equation*}  
  t_{1,1}(s,s-1)=\frac{p-1}{p+1}(p^{2s-2}-1).
  \end{equation*}
  De plus, on a $t_{1,1}(s,r-1)=t_{1,1}(s,s-1)$.
En effet, cette égalité est trivialement vérifiée si $r=s$.
 Si $r\geq s+1$, alors $(r-1,s)\in\Lambda$. 
 Ainsi, $t_{1,1}(s,r-1)=t_{1,1}(s,s-1)$ d'après le $ii)$.

   Il s'ensuit donc, d'après le $i)$ de la proposition \ref{calcul de la différente version appliquée.}, que  
\begin{align*}
t_{1,1}(s,r)-t_{1,2}(s,r)& =p(t_{1,1}(s,r-1)-t_{1,2}(s-1,r)) \\
& = p\left(\frac{p-1}{p+1}p^{2s-2}-p^{r+s-2}+\frac{p-1}{p+1}p^{2s-3}\right) \\
& = p\left((p-1)p^{2s-3}-p^{r+s-2}\right)<0
\end{align*}
Le $ii)$ de la proposition \ref{calcul de la différente version appliquée.} montre que $t_{1,1}(s,r)= t_{1,2}(s,r-1)$.
D'où $t_{1,1}(s,r)=\frac{p-1}{p+1}(p^{2s-2}-1)$ et
\begin{align*}
t_{1,2}(s,r) & = p( p^{r+s-2}-(p-1)p^{2s-3})+\frac{p-1}{p+1}(p^{2s-2}-1) \\
& = p^{r+s-1}-\frac{p-1}{p+1}(p^{2s-1}+1),
\end{align*}
ce qui montre que $(r,s)$ vérifie $i)$ et $ii)$. 
En particulier, $(s,s)$ vérifie $i)$ et $ii)$.
Par ailleurs, $(s,s)$ vérifie trivialement $iii)$ et $iv)$.
Ainsi, $(s,s)\in\Lambda$.
Comme $(r,s)\notin\Lambda$, il s'ensuit que $r\geq s+1$.

Montrons maintenant que $(r,s)$ vérifie $iii)$ et $iv)$ si $r\geq s+1$. 
Par minimalité, $(r-1,s),(r,s-1)\in\Lambda$.
 En utilisant respectivement $iii)$ et $iv)$, 
\begin{equation*} 
 t_{1,2}(r-1,s)=t_{1,2}(s,s)
\end{equation*}
   et    
   \begin{equation*}   
   t_{1,1}(r,s-1)=p^{r+s-2}-p^{2(s-1)}+\frac{p-1}{p+1}(p^{2(s-1)}-1).
   \end{equation*}
  Comme $(s,s)\in\Lambda$, on a, d'après le $i)$,  
\begin{equation*}  
  t_{1,2}(s,s)=p^{2s-1}-\frac{p-1}{p+1}(p^{2s-1}-1).
  \end{equation*}
   Il s'ensuit donc, d'après le $i)$ de la proposition \ref{calcul de la différente version appliquée.}, que   
\begin{align*}
t_{1,2}(r,s)-t_{1,1}(r,s) & =p(t_{1,2}(r-1,s)-t_{1,2}(r,s-1)) \\
& = p\left(p^{2s-1}-\frac{p-1}{p+1}(p^{2s-2}+p^{2s-1})+p^{2s-2}-p^{r+s-2}\right) \\
& = p(2p^{2s-2}-p^{r+s-2})<0.
\end{align*}

Le $ii)$ de la proposition \ref{calcul de la différente version appliquée.} montre que $t_{1,2}(r,s)= t_{1,2}(r-1,s)$.
 D'où  $t_{1,2}(r,s)= p^{2s-1}-\frac{p-1}{p+1}(p^{2s-1}-1)$ et
\begin{align*}
t_{1,1}(r,s) & = p(p^{r+s-2}-2p^{2s-2})+p^{2s-1}-\frac{p-1}{p+1}(p^{2s-1}+1)\\
& = p^{r+s-1}-\frac{p-1}{p+1}(p^{2s-1}+1)-p^{2s-1}\\
& = p^{r+s-1}-p^{2s}+\frac{p-1}{p+1}(p^{2s}-1).
\end{align*}

On vient de montrer que $(r,s)$ vérifie $iii)$ et $iv)$, ce qui contredit le fait que $(r,s)\notin\Lambda$.
Le théorème s'ensuit.

\end{proof}

\begin{thm} \label{Cas n=1 et v p a vaut 1}
Si $p\nmid v_F(a)$ alors, pour tous $r>s\geq 0$, on a: 

\begin{enumerate} [i)]
\item $t_{1,2}(s,s)= \frac{p^{2s}+p^{2s-2}+p-1}{p+1}$ si $s\geq 1$; 
\item $t_{1,2}(s,r)=p^{r+s-1}-\frac{p^{2s-1}-p^{2s-2}-p+1}{p+1}$ si $s\geq 1$;
\item $t_{1,2}(r,s)=t_{1,2}(s+1,s)=\frac{2p^{2s}+p-1}{p+1}$ si $s\geq 1$;
\item $t_{1,1}(s,r)= t_{1,1}(s,s)= t_{1,1}(s,s-1)$ si $s\geq 1$;
\item $t_{1,1}(r,s)=p^{r+s-1}-\frac{2p^{2s+1}-p+1}{p+1}$ si $r>s+1$;
\item $t_{1,1}(s+1,s)=t_{1,1}(s+1,s-1)$ si $s\geq 1$.
\end{enumerate}
\end{thm}

\begin{proof}
Montrons le théorème pour $s=0$. 
Remarquons que toutes les affirmations, sauf le $v)$, sont vides. 
L'extension cyclotomique $\Q_p(\zeta_{p^r})/\Q_p(\zeta_{p^{r-1}})$ ayant un saut égal à $p^{r-1}-1$, on en déduit que $t_{1,1}(r,0)=p^{r-1}-1$ et $v)$ s'ensuit.

 Montrons le théorème pour $s=1$.
  Tout d'abord, $i)$ et $ii)$ découlent du $ii)$ du Lemme \ref{calcul de $t(1,1)$.} et $iv)$ du fait que l'extension $L_{1,r}/L_{0,r}$ soit totalement ramifiée de degré $p-1$.
  D'après le $i)$ de la proposition \ref{calcul de la différente version appliquée.},   
  \begin{align*}
  t_{1,1}(2,1)-t_{1,2}(2,1) & =p(t_{1,1}(2,0)-t_{1,2}(1,1)) \\
   & = p( p-1-p)=-p< 0 \; (\text{par le Lemme \ref{calcul de $t(1,1)$.}}).
  \end{align*} 
 
Le $ii)$ de la proposition \ref{calcul de la différente version appliquée.} montre que $t_{1,1}(2,1)=t_{1,1}(2,0)=p-1$.
D'où $vi)$.
On en déduit donc que 
\begin{align*}
t_{1,2}(2,1) & = p(t_{1,2}(1,1)-t_{1,1}(2,0))+t_{1,1}(2,0) \\
& = p+p-1=2p-1,
\end{align*}
 ce qui montre $iii)$ dans le cas $r=2$.
 De plus, $v)$ est vide pour $r=2$.
 On a ainsi montré le théorème pour $r=2$ et $s=1$.

Montrons $iii)$ et $v)$ (toujours pour $s=1$) par récurrence sur $r\geq 2$. 
On a déjà traité le cas $r=2$.
Supposons que $r$ vérifie $iii)$ et $v)$ et montrons que $r+1$ aussi.
On sait que $t_{1,2}(r+1,0)=p^r-1$.
De plus, d'après le $iii)$, $t_{1,2}(r,1)= 2p-1$.
 D'après le $i)$ de la proposition \ref{calcul de la différente version appliquée.},  
\begin{align*}
t_{1,2}(r+1,1)-t_{1,1}(r+1,1) & =p(t_{1,2}(r,1)-t_{1,1}(r+1,0)) \\
& = p(2p-1-p^r+1)=p(2p-p^r)<0.
\end{align*} 
 
 Le $ii)$ de la proposition \ref{calcul de la différente version appliquée.} montre que $t_{1,2}(r+1,1)=t_{1,2}(r,1)=2p-1$. 
 Ainsi, 
\begin{align*} 
 t_{1,1}(r+1,1) & = p(p^r-2p)+t_{1,2}(2,1) \\
 & =p^{r+1}-2p^2+2p-1,
 \end{align*}
  ce qui montre l'hérédité du $iii)$ et du $v)$.
  Le théorème est donc vrai pour $s=1$.
  
  Notons $\Lambda$ l'ensemble des couples $(r,s)$ de $\N^2$ tels que $r> s\geq 0$ et vérifiant les affirmations $i)$ à $vi)$ du théorème. 
Supposons par l'absurde qu'il existe $(r,s)\notin\Lambda$ avec $r> s\geq 0$ et choisissons le minimal pour l'ordre lexicographique. 
Par ce qui précède, $s\geq 2$. 
 
Par minimalité de $(r,s)$, on en déduit que $(s,s-1)\in\Lambda$. 
Ainsi, en utilisant respectivement le $vi)$ et le $v)$, on en déduit que 
\begin{equation} \label{égalité du t 1 1 s s-1}
t_{1,1}(s,s-1)=t_{1,1}(s,s-2)
\end{equation}
 et 
\begin{equation*} 
 t_{1,1}(s,s-2)=p^{2s-3}-\frac{2p^{2s-3}-p+1}{p+1}.
 \end{equation*}
En utilisant le $ii)$,
\begin{equation*} 
 t_{1,2}(s-1,s)=p^{2s-2}-\frac{p^{2s-3}-p^{2s-4}-p+1}{p+1}.
 \end{equation*}
D'après le $i)$ de la proposition \ref{calcul de la différente version appliquée.}, 
\begin{align*}
t_{1,1}(s,s)-t_{1,2}(s,s)& =p(t_{1,1}(s,s-1)-t_{1,2}(s-1,s)) \\
& = p\left(p^{2s-3}-p^{2s-2}-\frac{p^{2s-3}+p^{2s-4}}{p+1}\right) \\
& = p^{2s-3}(-p^2+p-1)<0.
\end{align*}
Le $ii)$ de la proposition \ref{calcul de la différente version appliquée.} montre que 
\begin{equation} \label{égalité du t 1 1 s s}
t_{1,1}(s,s)=t_{1,1}(s,s-1).
\end{equation}
 D'où  
\begin{equation*}
t_{1,1}(s,s)=p^{2s-3}-\frac{2p^{2s-3}-p+1}{p+1}
\end{equation*}
 et 
\begin{align*}
t_{1,2}(s,s) & = p^{2s-3}(p^2-p+1)+p^{2s-3}-\frac{2p^{2s-3}-p+1}{p+1} \\
& = \frac{p^{2s-3}(p^2-p+2)(p+1)-2p^{2s-3}+p-1}{p+1}=\frac{p^{2s}+p^{2s-2}+p-1}{p+1},
\end{align*}
ce qui montre $i)$ et la seconde égalité du $iv)$. 

Par minimalité, $(s+1,s-1)\in\Lambda$.
Ainsi, en utilisant le $v)$, 
\begin{equation*}
t_{1,1}(s+1,s-1)=p^{2s-1}-\frac{2p^{2s-1}-p+1}{p+1}.
\end{equation*}
Par conséquent, en utilisant le $i)$ de la proposition \ref{calcul de la différente version appliquée.},  
\begin{align*}
t_{1,1}(s+1,s)-t_{1,2}(s+1,s) & = p(t_{1,1}(s+1,s-1)-t_{1,2}(s,s)) \\
& = p\left(p^{2s-1}-\frac{2p^{2s-1}+p^{2s}+p^{2s-2}}{p+1}\right) \\
& = p(p^{2s-1}-p^{2s-2}(p+1))=-p^{2s-1}<0.
\end{align*}
Le $ii)$ de la proposition \ref{calcul de la différente version appliquée.} montre que $t_{1,1}(s+1,s)=t_{1,1}(s+1,s-1)$.
D'où 
\begin{equation*}
t_{1,1}(s+1,s)=p^{2s-1}-\frac{p^{2s-1}-p+1}{p+1}
\end{equation*}
 et 
\begin{align*}
t_{1,2}(s+1,s)& =2p^{2s-1}-\frac{2p^{2s-1}-p+1}{p+1}= \frac{2p^{2s}+p-1}{p+1},
  \end{align*}
 ce qui prouve la seconde égalité du $iii)$ et $vi)$. 
 
 Nous allons maintenant prouver que $(r,s)$ vérifie $ii)$ et la seconde égalité du $iv)$.
Si $r\geq s+2$, alors $(r-1,s)\in \Lambda$.
 En utilisant $iv)$, on en déduit que 
\begin{equation*} 
t_{1,1}(s,r-1)=t_{1,1}(s,s-1)=t_{1,1}(s,s-2) \; \text{d'après \eqref{égalité du t 1 1 s s-1}}
\end{equation*}
De \eqref{égalité du t 1 1 s s}, il s'ensuit que $t_{1,1}(s,r-1)=t_{1,1}(s,s-2)$ pour tout $r\geq s+1$. 
 Or, d'après le $v)$, 
\begin{equation*}
t_{1,1}(s,s-2)=p^{2s-3}-\frac{2p^{2s-3}-p+1}{p+1}. 
\end{equation*} 
 et d'après le $ii)$,
\begin{equation*} 
 t_{1,2}(s-1,r)=p^{r+s-2}-\frac{p^{2s-3}-p^{2s-4}-p+1}{p+1}.
 \end{equation*}
 En utilisant le $i)$ de la proposition \ref{calcul de la différente version appliquée.}, 
\begin{align*}
t_{1,1}(s,r)-t_{1,2}(s,r)& =p(t_{1,1}(s,r-1)-t_{1,2}(s-1,r)) \\
& = p\left(p^{2s-3}-\frac{p^{2s-3}+p^{2s-4}}{p+1}-p^{r+s-2}\right) \\
& = p(p^{2s-3}-p^{2s-4}-p^{r+s-2})<0.
\end{align*}
Le $ii)$ de la proposition \ref{calcul de la différente version appliquée.} montre que  $t_{1,1}(s,r)=t_{1,1}(s,r-1)=t_{1,1}(s,s-1)$. 
D'où 
\begin{equation*}
t_{1,1}(s,r)=p^{2s-3}-\frac{2p^{2s-3}-p+1}{p+1}
\end{equation*}
 et
\begin{align*}
t_{1,2}(s,r) & = p(p^{r+s-2}+p^{2s-4}-p^{2s-3})+p^{2s-3}-\frac{2p^{2s-3}-p+1}{p+1} \\
& = p^{r+s-1}-\frac{p^{2s-1}-p^{2s-2}-p+1}{p+1},
\end{align*}
ce qui montre $ii)$ et la première égalité du $iv)$.  
 
 Il reste à montrer que $(r,s)$ vérifie $iii)$ et $v)$.
 Remarquons que dans le cas $r=s+1$, le $iii)$ est trivialement vérifié et $v)$ est vide.
 Ainsi, $(s+1,s)\in\Lambda$.
 Comme $(r,s)\notin\Lambda$, on en déduit que $r\geq s+2$.
 Par minimalité, $(r-1,s),(r,s-1)\in\Lambda$. 
En utilisant respectivement le $iii)$ et le $v)$,
\begin{equation*}
t_{1,2}(r-1,s)=\frac{2p^{2s}+p-1}{p+1}
\end{equation*}
 et 
\begin{equation*} 
 t_{1,1}(r,s-1)=p^{r+s-2}-\frac{2p^{2s-1}-p+1}{p+1}.
 \end{equation*}
En utilisant le $i)$ de la proposition \ref{calcul de la différente version appliquée.},
\begin{align*}
t_{1,2}(r,s)-t_{1,1}(r,s) & = p(t_{1,2}(r-1,s)-t_{1,1}(r,s-1)) \\
& = p\left(\frac{2p^{2s-1}+2p^{2s}}{p+1}-p^{r+s-2}\right) \\
& = -p(p^{r+s-2}-2p^{2s-1}) <0.
\end{align*}
Le $ii)$ de la proposition \ref{calcul de la différente version appliquée.} montre que $t_{1,2}(r,s) = t_{1,2}(r-1,s)$.
D'où 
\begin{equation*}
t_{1,2}(r-1,s)=\frac{2p^{2s}+p-1}{p+1}
\end{equation*}
 et
\begin{align*}
t_{1,1}(r,s) & = p(p^{r+s-2}-2p^{2s-1})+\frac{2p^{2s}+p-1}{p+1} \\
& = p^{r+s-1}+\frac{-2p^{2s}(p+1)+2p^{2s}+p-1}{p+1}\\
& = p^{r+s-1}-\frac{2p^{2s+1}-p+1}{p+1},
\end{align*}
ce qui prouve  $iii)$ et $v)$ et contredit le fait que $(r,s)\notin\Lambda$.
 Le théorème s'ensuit.

\end{proof}

\section{Relations de récurrence entre les $t_{n,k}(r,s_1,\dots,s_n)$.} \label{Relations de récurrence entre les}

Dans cette section, fixons un entier $r\geq 1$.
Soient $F$ et $a_1,\dots,a_n$ vérifiant les conditions de l'hypothèse \ref{hypothèse 3}.
On dira que $t_{n,k}(r,s_1,\dots,s_n)$ a la propriété $(P)$ si la (éventuelle) dépendance en $a_1,\dots,a_n$ est une (éventuelle) dépendance en $v_F(a_n)$. 
Par exemple, le théorème \ref{cas où n=1 et v p a vaut 0} et le théorème \ref{Cas n=1 et v p a vaut 1} montrent que $t_{1,k}(r,s)$ a la propriété $(P)$ pour tous $k\in\{1,2\}$, $r\geq1$ et $s\geq 0$.

Les résultats principaux de cette section se feront par récurrence sur $n$. 
On va donc commencer par le cas $n=2$. 
Cela aura pour but d'alléger les preuves des résultats principaux de cette section. 
De plus, ce cas particulier permettra de mieux comprendre notre approche pour traiter le cas général.

\subsection{Un premier résultat dans le cas $n=2$.}

\begin{lmm} \label{2unicité du saut ?}
Soit $s\geq 1$ un entier. 
Supposons que $p\nmid v_F(a_2)$.
Supposons également que $t:=t_{2,2}(r,s+1,s-1)=t_{2,3}(r,s,s)$ et que ce nombre ait la propriété $(P)$.
Alors, 
\begin{enumerate}[i)]
\item $L_{r,s+1,s}/L_{r,s,s-1}$ possède un unique saut égal à $t$;
\item pour tout $k\in \{1,2\}$, on a $t_{n,k+1}(r,s+1,s)=t$.
\end{enumerate}
\end{lmm}

\begin{proof}
Notons $K=K_{r,s,s-1}$ et $K'=K_{r,s+1,s}$. 
Alors, $[K' : K]=p^2$.  
En utilisant le Lemme \ref{lmm que j'ai rajouter sur la condition suffisante} (appliqué à $\alpha_1=a_1^{1/p^s}$ et à $\alpha_2=a_2^{1/p^{s-1}}$) et le Lemme \ref{sous-corps d'une extension kumérienne.}, un sous-corps intermédiaires de $K'/K$ de degré $p$ sur $K$ est de la forme $K(\gamma^{1/p})$ où $\gamma= \alpha_1^{x_1}\alpha_2^{x_2}$ avec $x_1,x_2\in\{0,\dots,p-1\}$ non tous nuls.

Si  $x_1=0$ alors, $x_2\neq 0$ et $\gamma=\alpha_2^{x_2}$.
Dans ce cas, le saut de $K(\gamma^{1/p})/L$ vaut $t_{2,3}(r,s,s)$, i.e. $t$.

  Supposons maintenant que $x_1\neq 0$. 
  Notons $\gamma_1=a_1^{x_1}a_2^{px_2}$.
Alors, $p\mid v_F(\gamma_1)$ puisque $p\mid v_F(a_1)$.
De plus, il est clair que 
\begin{equation*}
K=F(\zeta_{p^r}, \gamma_1^{1/p^s},a_2^{1/p^{s-1}}).
\end{equation*}
Comme $t_{2,2}(r,s,s-1)$ a la propriété $(P)$, on en déduit, en remplaçant $a_1$ par $\gamma_1$, que le saut de $K(\gamma^{1/p})/L$ vaut $t_{2,2}(r,s+1,s-1)$, i.e. $t$.

Ainsi, le saut de toute sous-extension de $K'/K$ de degré $p$ sur $K$ vaut $t$.
Le Lemme \ref{Unicité du saut, version théorique} permet donc de montrer $i)$ et $ii)$ s'ensuit car les $t_{2,k+1}(r,s+1,s)$ avec $k\in\{0,1\}$ sont des sauts de $K'/K$.
\end{proof}

\begin{prop} \label{2Théorème fondamental 2.}
Supposons que $p\nmid v_F(a_2)$. 
Alors, pour tous $s_1,s_2$ tels que $1\leq s_1\leq s_2$, on a  

\begin{enumerate} [i)]
\item \begin{enumerate} [a)]
\item $t_{2,2}(r,s_1,s_2)=t_{2,2}(r,s_1,s_1-1)$;
\item $t_{2,3}(r,s_2,s_1)=t_{2,3}(r,s_1,s_1)$ ;
\end{enumerate} 

\item \begin{enumerate} [a)]
\item $t_{2,3}(r,s_1,s_2)=p t_{2,3}(r,s_1-1,s_2) + (1-p)t_{2,2}(r,s_1,s_2)$;
\item $t_{2,2}(r,s_2,s_1)=p t_{2,2}(r,s_2,s_1-1) +(1-p)t_{2,3}(r,s_2,s_1)$ sauf si $s_1=s_2$;
\end{enumerate} 
\item $t_{2,3}(r,s_1,s_2)=t_{2,2}(r,s_2+1,s_1-1)$;
\item Les sauts $t_{2,2}(r,s_1,s_2)$ et $t_{2,3}(r,s_1,s_2)$ ont la propriété $(P)$.
\end{enumerate}

\end{prop}

\begin{proof}
Notons $\Lambda$ l'ensemble des $(s_1,s_2)$ tel que $s_1\leq s_2$ et vérifiant les affirmations $i)$ à $iv)$ de la proposition.
Supposons par l'absurde qu'il existe $s_1\leq s_2$ tel que $(s_1,s_2)\notin\Lambda$.
On le choisit minimal pour l'ordre lexicographique.

La preuve se fera en quatre étapes (décrites ci-dessous) très largement indépendantes (seul le lemme ci-dessous sera utilisé dans la première et la troisième étape). 
 
\textit{Première étape} : $ib)$ et $ii b)$ sont vérifiés.

\textit{Deuxième étape} : $ia),\; ii a)$ et $iv)$ sont vérifiés.

\textit{Troisième étape} :  $iii)$ est vérifié si $s_1>1$.

\textit{Quatrième étape} : $iii)$ est vérifié si $s_1=1$. \\

\textit{Première étape} :
 
Si $s_2=s_1$, alors $ib)$ est trivialement vérifié et $ii b)$ est vide.
Supposons $s_2=s_1+1$.
Par minimalité, $(s_1,s_1)\in\Lambda$ et on déduit du $iii)$ que 
\begin{equation*}
t_{2,2}(r,s_1+1,s_1-1)=t_{2,3}(r,s_1,s_1).
\end{equation*}
Les conditions du Lemme \ref{2unicité du saut ?} appliquées à $s=s_1$ et à $E=\{1\}$ sont ainsi vérifiées.
Par conséquent, le $ii)$ du Lemme \ref{2unicité du saut ?} permet d'en déduire que
\begin{equation} \label{2blabla}
t_{2,2}(r,s_1+1,s_1)=t_{2,3}(r,s_1+1,s_1)=t_{2,3}(r,s_1,s_1),
\end{equation}
ce qui montre $i b)$ et $ii b)$.
 
 Supposons maintenant $s_2\geq s_1+2$.  
 Ainsi, $(s_1,s_2-1)\in\Lambda$ par minimalité. 
 Du $ib)$, il s'ensuit que $t_{2,3}(r,s_2-1,s_1)=t_{2,3}(r,s_1,s_1)$.
 Par conséquent, le lemme ci-dessous, appliqué à $m=s_1$ et à $l=s_2$, montre $i b)$ et $ii b)$.

\begin{lmm*} \label{2Lemme dont je dois retirer le numéro.}
Soient $m\in\{1,\dots,s_1\}$ et $ l\in\{m+2,\dots ,s_2+1-\delta_{m,s_1}\}$. 
Alors,

\begin{enumerate} [i)]
\item $t_{2,3}(r,l,m)=t_{2,3}(r,l-1,m)$;
\item $t_{2,2}(r,l,m)=p t_{2,2}(r,l,m-1)+(1-p)t_{2,3}(r,l,m)$.
\end{enumerate}

\end{lmm*}

\begin{proof}
La proposition  \ref{croissance des suites} appliquée à $K=L_{r,0,m-1}$ et à $\alpha=a_1$ montre (car $m+1 < l$) que 
\begin{equation*}
t_{2,2}(r,l,m-1)>t_{2,2}(r,m+1,m-1).
\end{equation*} 
Par minimalité, $(m,m)\in\Lambda$ (car $m\leq s_1< s_2$).
 On déduit alors du $iii)$ que 
\begin{equation*}
t_{2,2}(r,m+1,m-1)=t_{2,3}(r,m,m).
\end{equation*}
Par minimalité, $(m,l-1)\in\Lambda$ puisque $l-1>m$ par hypothèse. 
Ainsi, d'après le $i b)$, 
\begin{equation*}
t_{2,3}(r,l-1,m)=t_{2,3}(r,m,m).
\end{equation*}
On en déduit donc que $t_{2,3}(r,l-1,m)< t_{2,2}(r,l,m-1)$.
D'après le $i)$ de la proposition \ref{calcul de la différente version appliquée.},
\begin{equation} \label{2le multline du lemme}
t_{2,3}(r,l,m)-t_{2,2}(r,l,m)= p(t_{2,3}(r,l-1,m)-t_{2,2}(r,l,m-1))<0.
\end{equation}
 Le $ii)$ de la proposition \ref{calcul de la différente version appliquée.} montre que $t_{2,3}(r,l,m)=t_{2,3}(r,l-1,m)$ (i.e.  le $i)$ de ce lemme) et le $ii)$ se déduit du \eqref{2le multline du lemme} car 
\begin{equation*} 
t_{2,2}(r,l,m)= t_{2,3}(r,l,m)-p(t_{2,3}(r,l-1,m)-t_{2,2}(r,l,m-1)). 
\end{equation*}

\end{proof} 

\vspace{0.1cm}

\textit{Deuxième étape} :

D'après le $i)$ de la proposition \ref{calcul de la différente version appliquée.}, 
\begin{multline} \label{2equation utile une fois}
t_{2,2}(r,s_1,s_2)-t_{2,3}(r,s_1,s_2) = p(t_{2,2}(r,s_1,s_2-1)-t_{2,3}(r,s_1-1,s_2)).
\end{multline}
Si $s_1<s_2$ alors, $(s_1,s_2-1)\in\Lambda$ et on déduit du $i a)$ que 
\begin{equation} \label{2égalité servant pour l'inégalité}
t_{2,2}(r,s_1,s_2-1)=t_{2,2}(r,s_1,s_1-1).
\end{equation}
Ainsi, \eqref{2égalité servant pour l'inégalité} est valable pour tout $s_1\leq s_2$.
 Montrons maintenant que
\begin{equation} \label{2D}
t_{2,2}(r,s_1,s_2-1)<t_{2,3}(r,s_1-1,s_2)
\end{equation}
en distinguant le cas $s_1>1$ du cas $s_1=1$.

Si $s_1>1$ alors, $( s_1-1,s_1-1)\in\Lambda$ par minimalité. En utilisant le $iii)$,
\begin{equation} \label{2E}
t_{2,3}(r,s_1-1,s_1-1)=t_{2,2}(r,s_1,s_1-2).
\end{equation}
Les hypothèses du Lemme \ref{2unicité du saut ?} appliquées à $s=s_1-1$ et à $E=\{1\}$ sont donc vérifiées d'après \eqref{2E} et le $iv)$.
Le $ii)$ du Lemme \ref{2unicité du saut ?} montre donc que
\begin{equation} \label{2AA}
t_{2,2}(r,s_1,s_1-2)=t_{2,2}(r,s_1, s_1-1).
\end{equation}
En combinant \eqref{2E} et \eqref{2AA}, on en déduit que
\begin{equation} \label{2c'est le * de mon brouillon}
t_{2,2}(r,s_1,s_1-1)=t_{2,3}(r,s_1-1,s_1-1).
\end{equation}
 La proposition  \ref{croissance des suites} appliquée à $K=L_{r,s-1,0}$ et à $\alpha=a_n$ prouve (car $s_1-1 <s_2$) que $t_{2,3}(r,s_1-1,s_1-1)< t_{2,3}(r,s_1-1,s_2)$.
 En utilisant \eqref{2égalité servant pour l'inégalité}, \eqref{2c'est le * de mon brouillon} et cette inégalité, on obtient 
\begin{equation*} 
 t_{2,2}(r,s_1,s_2-1)<t_{2,3}(r,s_1-1,s_2),
 \end{equation*}
qui est bien l'inégalité \eqref{2D} dans le cas où $s_1>1$. 
 
 Supposons maintenant que $s_1=1$.
 D'après le théorème \ref{cas où n=1 et v p a vaut 0},
\begin{equation*}
t_{2,2}(r,s_1,s_1-1)=t_{2,2}(r,1,0)=1.
\end{equation*}
De \eqref{2égalité servant pour l'inégalité}, il suffit de montrer que $t_{2,3}(r,0,s_2)>1$ pour en déduire \eqref{2D}.
La proposition \ref{croissance des suites} appliquée à $K=L_{r,0,0}$ et à $\alpha=a_n$ donne (car $s_2 \geq 1$) $t_{2,3}(r,0,s_2)\geq t_{2,3}(r,0,1)$. 
De plus, $t_{2,3}(r,0,1)>1$ d'après le théorème \ref{Cas n=1 et v p a vaut 1}.
Ceci montre \eqref{2D} dans le cas $s_1=1$.

De \eqref{2equation utile une fois} et \eqref{2D}, on obtient que $t_{2,2}(r,s_1,s_2)<t_{2,3}(r,s_1,s_2)$.
 Le $ii)$ de la proposition \ref{calcul de la différente version appliquée.} montre que $t_{2,2}(r,s_1,s_2)=t_{2,2}(r,s_1,s_2-1)$.
 En utilisant \eqref{2égalité servant pour l'inégalité}, on en déduit $ia)$.
En utilisant de nouveau \eqref{2equation utile une fois}, on obtient que
\begin{equation} \label{2G} 
 t_{2,3}(r,s_1,s_2)=p t_{2,3}(r,s_1-1,s_2)+(1-p)t_{2,2}(r,s_1,s_2),
\end{equation}
ce qui montre $ii a)$.

Montrons maintenant que $( s_1,s_2)$ vérifie $iv)$.
Tout d'abord, le $ia)$ montre que $t_{2,2}(r,s_1,s_2)=t_{2,2}(r,s_1,s_1-1)$.

Si $s_1=1$ alors, $t_{2,2}(r,s_1,s_2)=t_{2,2}(r,1,0)=1$ et vérifie donc $iv)$.

Si $s_1>1$ alors, par minimalité, $(s_1,s_1-1)\in\Lambda$. 
Ainsi, $t_{2,2}(r,s_1,s_1-1)$, et donc $t_{2,2}(r,s_1,s_2)$, vérifie $iv)$. 

Il ne nous reste plus qu'à montrer que $t_{2,3}(r,s_1,s_2)$ vérifie $iv)$. 
De \eqref{2G}, il suffit de montrer que  $t_{2,3}(r,s_1-1,s_2)$ vérifie $iv)$ puisque $t_{2,2}(r,s_1,s_2)$ vérifie $iv)$.

Si $s_1>1$, cela découle du fait que $(s_1-1,s_2)\in\Lambda$, et donc que $t_{2,3}(r,s_1-1,s_2)$ vérifie $iv)$.
Enfin, si $s_1=1$, cela découle du théorème \ref{Cas n=1 et v p a vaut 1}.
Ainsi, $( s_1,s_2)$ vérifie $iv)$. \\

\textit{Troisième étape} : 
Supposons $s_1>1$.
Alors, d'après le $ii a)$, 
\begin{equation} \label{2H} 
 t_{2,3}(r,s_1,s_2) =p t_{2,3}(r,s_1-1,s_2)+(1-p)t_{2,2}(r,s_1,s_2).
 \end{equation}
Par minimalité, $(s_1-1,s_2)\in\Lambda$ et on déduit respectivement du $iii)$ et du $i a)$ que 
\begin{equation} \label{2I}
t_{2,3}(r,s_1-1,s_2)=t_{2,2}(r,s_2+1,s_1-2) \; \text{et} \; t_{2,2}(r,s_1,s_2)=t_{2,2}(r,s_1,s_1-1).
 \end{equation}
Du \eqref{2c'est le * de mon brouillon}, on en déduit que 
\begin{equation} \label{2II} 
t_{2,2}(r,s_1,s_2)=t_{2,3}(r,s_1-1,s_1-1).
\end{equation}
De plus, le $i b)$ donne 
\begin{equation} \label{2J}
t_{2,3}(r,s_2,s_1-1)=t_{2,3}(r,s_1-1,s_1-1).
\end{equation}
Ainsi, de \eqref{2H}, \eqref{2I}, \eqref{2II} et \eqref{2J}, il s'ensuit que 
\begin{multline} \label{2première formule de l'égalité du iii)} 
 t_{2,3}(r,s_1,s_2) =p t_{2,2}(r,s_2+1,s_1-2)+(1-p)t_{2,3}(r,s_1-1,s_1-1).
 \end{multline}
 
Le $ii)$ du lemme dans la preuve appliqué à $m=s_1-1$ et à $l=s_2+1$ donne :
\begin{multline} \label{2K}
t_{2,2}(r,s_2+1,s_1-1)= pt_{2,2}(r,s_2+1,s_1-2)+(1-p)t_{2,3}(r,s_2+1,s_1-1).
\end{multline}
 
Comme $(s_1-1, s_2+1)\in\Lambda$ par minimalité, on a, en utilisant le $i b)$,
\begin{equation} \label{2L}
t_{2,3}(r,s_2+1,s_1-1)=t_{2,3}(r,s_1-1,s_1-1).
\end{equation}
 En injectant \eqref{2L} dans \eqref{2L}, que l'on compare ensuite à \eqref{2première formule de l'égalité du iii)}, on en déduit $iii)$, ce qui conclut la troisième étape. \\

Supposons maintenant $s_1=1$. 
Posons $s=s_2$.
Alors, d'après le $iia)$, on en déduit que $t_{2,3}(r,1,s_2)=p t_{2,3}(r,0,s_2)+(1-p) t_{2,2}(r,1,s_2)$.
Ainsi, pour montrer $iii)$, il suffit de montrer que 
\begin{equation} \label{2equation pour initialiser le iii)}
t_{2,2}(r,s+1,0)=p t_{2,3}(r,0,s)+(1-p)t_{2,2}(r,1,s).
\end{equation} 
 
 \textit{Quatrième étape} :
 En utilisant respectivement le $i a)$ et le théorème \ref{cas où n=1 et v p a vaut 0},  
\begin{equation*} 
 t_{2,2}(r,1,s)=t_{2,2}(r,1,0)=1.
 \end{equation*}

  Si $r<s+1$ alors, en utilisant respectivement le $i)$ du théorème \ref{cas où n=1 et v p a vaut 0} et le $ii)$ du théorème \ref{Cas n=1 et v p a vaut 1}, on obtient : 
\begin{equation*}  
  t_{2,2}(r,s+1,0)=p^{r+s}-\frac{p-1}{p+1}(p^{2r-1}+1)
  \end{equation*}
 et
\begin{equation*} 
 t_{2,3}(r,0,s)=p^{r+s-1}-\frac{p^{2r-1}-p^{2r-2}-p+1}{p+1}.
  \end{equation*} 
 L'égalité \eqref{2equation pour initialiser le iii)} est ainsi vérifiée.
 
 Si $r\geq s+1$ alors, en utilisant respectivement le $iii)$ et le $i)$ du théorème \ref{cas où n=1 et v p a vaut 0} ainsi que le $iii)$ du théorème \ref{Cas n=1 et v p a vaut 1}, 
\begin{equation*} 
 t_{2,2}(r,s+1,0)=p^{2s-1}-\frac{p-1}{p+1}(p^{2s-1}+1) \; \text{et} \; t_{2,3}(r,0,s)=\frac{2p^{2s}+p-1}{p+1}.
 \end{equation*} 
 L'égalité \eqref{2equation pour initialiser le iii)} est ainsi vérifiée, ce qui conclut la quatrième étape. 
 Ainsi, $(s_1,s_2)\in\Lambda$, ce qui est absurde. 
 Ceci montre la proposition. 

\end{proof}

\subsection{Le cas général.}

\begin{lmm} \label{unicité du saut 2}
Soient $n\geq 2$, $s_1,\dots,s_n\geq 1$ et $E\subset \{1,\dots,n\}$ un ensemble non vide tels  que pour tout $i\in E$, on ait $p\mid v_F(a_i)$ et $s:=s_i\geq 0$. 
Supposons aussi que les $(t_{n,k+1}(r,s_1+\delta_{1,k},\dots,s_n+\delta_{n,k}))_{k\in E}$ aient la propriété $(P)$ et soient tous égaux à un certain $t$. 
Posons 
\begin{equation*}
s_i'=\begin{cases} 
s_i+1 \; \text{si} \; i\in E \\
s_i \; \text{sinon}
\end{cases}.
\end{equation*}
Alors,
\begin{enumerate}[i)]
\item $L_{r,s'_1,\dots,s_n'}/L_{r,s_1,\dots,s_n}$ possède un unique saut égal à $t$;
\item pour tout $i\in E$, on a $t_{n,i+1}(r,s'_1,\dots,s_n')=t$.
\end{enumerate}

\end{lmm}

\begin{proof}
Notons $K=L_{r,s_1,\dots,s_n}$ et $K'=L_{r,s'_1,\dots,s_n'}$.
On sait que $[K' : K]=p^{\sum_{i=1}^n s'_i-s_i}$.
D'après le Lemme \ref{lmm que j'ai rajouter sur la condition suffisante} (appliqué à $\alpha_i=a_i^{1/p^s}$ où $i\in E$) et le Lemme  \ref{sous-corps d'une extension kumérienne.}, un sous-corps intermédiaire de degré $p$ de l'extension $K'/K$ est de la forme $K(\gamma^{1/p})$ avec $\gamma=\prod_{i\in E} \alpha_i^{x_i}$ où les $x_i\in\{0,\dots,p-1\}$ sont des entiers non tous nuls.
Remarquons que $p \mid v_F(\gamma^{p^s})$ puisque $\gamma^{p^s}=\prod_{i\in E} a_i^{x_i}$ et que $p\mid v_F(a_i)$ pour tout $i\in E $.
Fixons $i\in E$ tel que $x_i\neq 0$.
Il est clair que 
\begin{equation*}
K=F(\zeta_{p^r}, a_1^{1/p^{s_1}},\dots,a_{i-1}^{1/p^{s_{i-1}}}, (\gamma^{p^s})^{1/p^s},a_{i+1}^{1/p^{s_{i+1}}},\dots,a_n^{1/p^{s_n}}).
\end{equation*}
Comme $t_{n,i+1}(r,s_1+\delta_{1,i},\dots,s_n+\delta_{n,i})$ a la propriété $(P)$, on en déduit, en remplaçant $a_i$ par $\gamma^{p^s}$, que le saut de $K(\gamma^{1/p})/K$ vaut $t_{n,i+1}(r,s_1+\delta_{1,i},\dots,s_n+\delta_{n,i})$, i.e. $t$.
Ainsi, le saut de toute sous-extension intermédiaire de $K'/K$ de degré $p$ sur $K$ vaut $t$.
Le Lemme \ref{Unicité du saut, version théorique} prouve $i)$.
 Le $ii)$ se déduit du $i)$ car les $t_{n,k+1}(r,s'_1,\dots,s_n')$, avec $k\in E$, sont des sauts de $K'/K$. 

\end{proof}

Les exemples suivants serviront pour l'initialisation de la prochaine proposition.
De plus, le premier exemple ci-dessous est une première application du Lemme \ref{unicité du saut 2}.

\begin{ex} \label{exemple sur le calcul de t 1 1 avec n=2}
\rm{ Si $n=2$ et $p\mid v_F(a_2)$, alors $t_{2,2}(r,1,0)=t_{2,3}(r,0,1)=1$ d'après le théorème \ref{cas où n=1 et v p a vaut 0}.
Ainsi, le Lemme \ref{unicité du saut 2} appliqué à $s=0$ et à $E=\{1,2\}$ montre que $1$ est l'unique saut de $L_{r,1,1}/L_{r,0,0}$ et donc que $t_{1,2}(r,1,1)=t_{2,3}(r,1,1)=1$.}
\end{ex}

\begin{ex} \label{calcul de t 1,1 qui sert qu'une fois}

\rm{Supposons $n=2$ et $p\nmid v_F(a_2)$.
D'après le $i)$ de la proposition \ref{calcul de la différente version appliquée.},
\begin{equation} \label{exemple 2 18}
t_{2,2}(r,1,1)-t_{2,3}(r,1,1) =p (t_{2,2}(r,1,0)-t_{2,3}(r,0,1)).
\end{equation}
Or, $t_{2,2}(r,1,0)=1$ d'après le théorème \ref{cas où n=1 et v p a vaut 0}. 
De plus, $t_{2,3}(1,0,1)=p$ d'après le Lemme \ref{calcul de $t(1,1)$.} et $t_{2,3}(r,0,1)=2p-1$ si $r\geq 2$ d'après le théorème \ref{Cas n=1 et v p a vaut 1}.
De \eqref{exemple 2 18}, il s'ensuit que $t_{2,2}(r,1,1)< t_{2,3}(r,1,1)$.
 D'après le $ii)$ de la proposition \ref{calcul de la différente version appliquée.}, on en conclut que $t_{2,2}(r,1,1)=t_{2,2}(r,1,0)=1$.
}
\end{ex}

\begin{prop} \label{Théorème fondamental}
 Soient $n\geq 2$ et $s_1,\dots,s_n\geq 0$. Si $n=2$, on supposera que $p\mid v_F(a_n)$.
 Soient  $l,k\in\{1,\dots,n\}$ tels que $l<k$ et $p\mid v_F(a_l),v_F(a_k)$. Notons $\sigma\in\mathfrak{S}_n$ la transposition $(l,k)$ ou l'identité. 
 Si $1\leq s_l\leq s_k$, alors

\begin{enumerate} [i)]
\item $t_{n,k+1}(r,s_1,\dots,s_n)=t_{n,\sigma(k)+1}(r,s_{\sigma(1)},\dots,s_{\sigma(n)})$;
\item $t_{n,\sigma(l)+1}(r,s_{\sigma(1)},\dots,s_{\sigma(n)})=t_{n,\sigma(l)+1}(r,s'_1,\dots,s_n')$ où  $s_i'=s_i$ si $i\neq k,l$ et $s_i'=s_l$ sinon ;
\item $t_{n,\sigma(k)+1}(r,s_{\sigma(1)},\dots,s_{\sigma(n)})=(1-p)t_{n,\sigma(l)+1}(r,s_{\sigma(1)},\dots,s_{\sigma(n)}) \\
+ p t_{n,\sigma(k)+1}(r,s_{\sigma(1)}-\delta_{\sigma(1),l},\dots,s_{\sigma(n)}-\delta_{\sigma(n),l})$;
\item Pour $i\in\{l,k\}$, le saut $t_{n,i+1}(r,s_1,\dots,s_n)$ a la propriété $(P)$.
\end{enumerate}

\end{prop}

\begin{proof}
 
  Quitte à permuter $a_1^{1/p^{s_1}}$ avec $a_l^{1/p^{s_l}}$ et $a_2^{1/p^{s_2}}$ avec $a_k^{1/p^{s_k}}$, on peut supposer, sans perte de généralité, que $l=1$ et $k=2$. 
Notons $\Lambda$ l'ensemble des $(n;s_1,\dots,s_n)$ tel que $n\geq 2$, que $p\mid v_F(a_n)$ si $n=2$ et tel que le $n$-uplet $(s_1,\dots,s_n)$ vérifie $s_1\leq s_2$ et les affirmations $i)$ à $iv)$ de la proposition.

  Montrons que $(2;1,1)\in\Lambda$ (on a alors $p\mid v_F(a_2)$ puisque $n=2$).
Remarquons que $ii)$ est trivialement vérifié. 
Le $i)$, $iii)$ et $iv)$ découlent des égalités ci-dessous (cf Exemple \ref{exemple sur le calcul de t 1 1 avec n=2})
\begin{equation*}
t_{2,2}(r,1,1)=t_{r,2,3}(r,1,1)=t_{2,2}(r,1,0)=t_{2,3}(r,0,1)=1.
\end{equation*}
  
  Montrons également que $(3;1,1,1)\in\Lambda$ si $p\nmid v_F(a_3)$.
Dans ce cas, $ii)$ est trivialement vérifié.
De plus, 
\begin{equation*}
t_{3,2}(r,1,0,1)=t_{3,3}(r,0,1,1)=t_{2,2}(r,1,1)=1
\end{equation*}
 d'après l'Exemple \ref{calcul de t 1,1 qui sert qu'une fois}.
En appliquant le Lemme \ref{unicité du saut 2} à $s=0$ et à $E=\{1,2\}$, 
\begin{equation*}
t_{3,2}(r,1,1,1)=t_{3,3}(r,1,1,1)=1,
\end{equation*}
 ce qui montre $i)$, $iii)$ et $iv)$. \\

Supposons par l'absurde qu'il existe $(n; s_1,\dots,s_n)\notin\Lambda$ tel que $n\geq 2$, que $p\mid v_F(a_n)$ si $n=2$ et tel que $s_1\leq s_2$. 
 Choisissons le minimal pour l'ordre lexicographique.
Remarquons que le $i)$ est trivialement vérifié si $\sigma=id$.
Afin d'alléger les notations, notons $\overline{s}=(s_3,\dots,s_n)$.

 Tout d'abord, supposons $s_1=s_2=s$. 
  Alors, $ii)$ est trivialement vérifié.
  Dans un premier temps, on souhaite montrer que 
  \begin{equation} \label{égalité dans le cas s 1 égal s 2}   
   t_{n,2}(r,s,s-1,\overline{s})=t_{n,3}(r,s-1,s,\overline{s}) \; \text{et vérifie} \; iv).
   \end{equation}
 
   Si $s>1$ alors, $(n;s-1,s,\overline{s})\in\Lambda$ par minimalité. 
On en déduit alors \eqref{égalité dans le cas s 1 égal s 2} en utilisant respectivement le $i)$ et le $iv)$.
  Supposons maintenant $s=1$. 
  Alors, $n\geq 3$ si $p\mid v_F(a_n)$ puisque $(2;1,1)\in\Lambda$ et $n\geq 4$ si $p\nmid v_F(a_n)$ puisque $(3;1,1,1)\in\Lambda$.
  Ainsi, $(n-1;1,\overline{s})\in\Lambda$.
Il s'ensuit donc que, $t_{n-1,2}(r,1,\overline{s})$ vérifie $iv)$. 
Par conséquent,   
\begin{equation*}  
 \begin{cases}
  t_{n,2}(r,1,0,\overline{s})=t_{n-1,2}(r,1,\overline{s}) \\
  t_{n,3}(r,0,1,\overline{s})=t_{n-1,2}(r,1,\overline{s})
  \end{cases},
  \end{equation*}
ce qui montre \eqref{égalité dans le cas s 1 égal s 2}.
   
Les hypothèses du Lemme \ref{unicité du saut 2} où l'on remplace $s$ par $s-1$ et où $E=\{1,2\}$ sont donc vérifiées.
Le $ii)$ du Lemme \ref{unicité du saut 2} montre donc les égalités suivantes :
\begin{equation} \label{unicité du saut dans la preuve du thm fondamental}
t_{n,2}(r,s,s,\overline{s})=t_{n,3}(r,s,s,\overline{s})=t_{n,2}(r,s,s-1,\overline{s})=t_{n,3}(r,s-1,s,\overline{s}).
\end{equation}   
Ceci prouve $i)$ et $iv)$.    
Remarquons que ces égalités montrent également $iii)$. 
 En effet, si $\sigma=id$, l'égalité du $iii)$ devient 
\begin{equation*} 
 t_{n,3}(r,s,s,\overline{s})=(1-p)t_{n,2}(r,s,s,\overline{s})+p t_{n,3}(r,s-1,s,\overline{s}),
\end{equation*} 
 qui est clairement vérifiée d'après \eqref{unicité du saut dans la preuve du thm fondamental}. 
 Dans le cas où $\sigma=(1,2)$, 
\begin{equation*} 
 t_{n,2}(r,s,s,\overline{s})=(1-p)t_{n,3}(r,s,s,\overline{s})+p t_{n,2}(r,s,s,\overline{s}),
 \end{equation*}
 qui est également clairement vérifiée d'après \eqref{unicité du saut dans la preuve du thm fondamental}.
 Ainsi, $(n;s,s,\overline{s})\in\Lambda$, ce qui est absurde. \\
 
Supposons maintenant $1\leq s_1<s_2$.
 Comme $(n;s_1,s_2-1,\overline{s})\in\Lambda$ par minimalité, on déduit du $ii)$ que
\begin{equation} \label{1)} 
 t_{n,\sigma(1)+1}(r,s_{\sigma(1)}-\delta_{\sigma(1),2},s_{\sigma(2)}-\delta_{\sigma(2),2},\overline{s})=t_{n,\sigma(1)+1}(r,s_1,s_1,\overline{s}).
\end{equation} 
Comme $(n;s_1,s_1,\overline{s})\in\Lambda$ par minimalité, on a, d'après le $i)$,
\begin{equation} \label{je peux permuter 1 et 2}   
   t_{n,\sigma(1)+1}(r,s_1,s_1,\overline{s})=t_{n,\sigma(2)+1}(r,s_1,s_1,\overline{s}).
\end{equation}   
Du \eqref{unicité du saut dans la preuve du thm fondamental}, on en déduit que 
\begin{equation} \label{je peux retirer 1 au plus grand}   
   t_{n,\sigma(2)+1}(r,s_1,s_1,\overline{s})=t_{n,\sigma(2)+1}(r,s_1-\delta_{\sigma(1),1},s_1-\delta_{\sigma(2),1},\overline{s}).
\end{equation}
En combinant les formules \eqref{1)}, \eqref{je peux permuter 1 et 2} et \eqref{je peux retirer 1 au plus grand}, il s'ensuit que   
\begin{multline} \label{2)}  
   t_{n,\sigma(1)+1}(r,s_{\sigma(1)}-\delta_{\sigma(1),2},s_{\sigma(2)}-\delta_{\sigma(2),2},\overline{s}) \\
   =t_{n,\sigma(2)+1}(r,s_1-\delta_{\sigma(1),1},s_1-\delta_{\sigma(2),1},\overline{s}).
\end{multline}   
Enfin, en utilisant la proposition \ref{croissance des suites} à $\alpha=a_2$ et à $K=L_{r,s_1(1-\delta_{\sigma(1),2}),s_2(1-\delta_{\sigma(2),2}),\overline{s}}$, on en déduit (car $s_1<s_2$) que  
\begin{multline} \label{3)} 
 t_{n,\sigma(2)+1}(r,s_1-\delta_{\sigma(1),1},s_1-\delta_{\sigma(2),1},\overline{s})<t_{n,\sigma(2)+1}(r,s_{\sigma(1)}-\delta_{\sigma(1),1},s_{\sigma(2)}-\delta_{\sigma(2),1},\overline{s}).
\end{multline} 
 En utilisant le $i)$ de la proposition \ref{calcul de la différente version appliquée.}, on a 
\begin{multline*} 
 t_{n,\sigma(1)+1}(r,s_{\sigma(1)},s_{\sigma(2)},\overline{s})-t_{n,\sigma(2)+1}(r,s_{\sigma(1)},s_{\sigma(2)},\overline{s})= \\
p t_{n,\sigma(1)+1}(r,s_{\sigma(1)}- \delta_{\sigma(1),2},s_{\sigma(2)}-\delta_{\sigma(2),2},\overline{s}) \\ 
-pt_{n,\sigma(2)+1}(r,s_{\sigma(1)}-\delta_{\sigma(1),1},s_{\sigma(2)}-\delta_{\sigma(2),1},\overline{s})).
\end{multline*}
De \eqref{2)} et \eqref{3)}, on en déduit que    
\begin{equation*}   
   t_{n,\sigma(1)+1}(r,s_{\sigma(1)},s_{\sigma(2)},\overline{s})<t_{n,\sigma(2)+1}(r,s_{\sigma(1)},s_{\sigma(2)},\overline{s}).
\end{equation*}   
   En utilisant le $ii)$ de la proposition \ref{calcul de la différente version appliquée.}, il s'ensuit, d'après \eqref{2)},  
  \begin{equation} \label{c'est le 20}
\begin{aligned}
t_{n,\sigma(1)+1}(r,s_{\sigma(1)},s_{\sigma(2)},\overline{s}) & =t_{n,\sigma(1)+1}(r,s_{\sigma(1)}-\delta_{\sigma(1),2},s_{\sigma(2)}-\delta_{\sigma(2),2},\overline{s}) \\
& = t_{n,\sigma(1)+1}(r,s_1,s_1,\overline{s}) \; (\text{d'après \eqref{1)}}), 
\end{aligned}
\end{equation}
ce qui prouve $ii)$ et 
\begin{equation} \label{preuve du iii)}
\begin{aligned}
t_{n,\sigma(2)+1}(r,s_{\sigma(1)},s_{\sigma(2)},\overline{s}) & =p t_{n,\sigma(2)+1}(r,s_{\sigma(1)}-\delta_{\sigma(1),1},s_{\sigma(2)}-\delta_{\sigma(2),1},\overline{s}) \\
& +(1-p)t_{n,\sigma(1)+1}(r,s_{\sigma(1)},s_{\sigma(2)},\overline{s}),
\end{aligned}
\end{equation}  
 ce qui montre $iii)$. 
 
Par minimalité, $(n; s_1,s_1,\overline{s})\in\Lambda$. 
Ainsi, $t_{n,2}(r,s_1,s_1,\overline{s})$ vérifie $iv)$. 
En appliquant \eqref{c'est le 20} à $\sigma=id$, il en résulte que $t_{n,2}(r,s_1,s_2,\overline{s})$ vérifie $iv)$.  
 
Il ne nous reste plus qu'à montrer que $(n;s_1,\dots,s_n)$ vérifie $i)$ et $iv)$ pour $i=2$. 
 En appliquant \eqref{preuve du iii)} respectivement à $\sigma=id$ et à $\sigma=(1,2)$, on en déduit que 
\begin{equation} \label{A} 
 t_{n,3}(r,s_1,s_2,\overline{s})= p t_{n,3}(r,s_1-1,s_2,\overline{s})+(1-p)t_{n,2}(r,s_1,s_2,\overline{s})
 \end{equation}
 et 
 \begin{equation} \label{B}
t_{n,2}(r,s_2,s_1,\overline{s}) =p t_{n,2}(r,s_2,s_1-1,\overline{s})+(1-p)t_{n,3}(r,s_2,s_1,\overline{s}).
\end{equation}  
Par minimalité, $(n; s_1,s_1,\overline{s})\in\Lambda$. 
Ainsi, $t_{n,2}(r,s_1,s_1,\overline{s})=t_{n,3}(r,s_1,s_1,\overline{s})$ d'après le $i)$.
De plus, $t_{n,2}(r,s_1,s_2,\overline{s})=t_{n,2} (r, s_1,s_1,\overline{s})$ et $t_{n,3}(r,s_2,s_1,\overline{s})=t_{n,3}(r,s_1,s_1,\overline{s})$ d'après le $ii)$.
Par conséquent, $t_{n,2}(r,s_1,s_2,\overline{s})=t_{n,3}(r,s_2,s_1,\overline{s})$.
En injectant cette égalité dans \eqref{A}, que l'on compare ensuite à \eqref{B}, on en déduit que pour montrer $i)$ et $iv)$, il suffit de prouver que 
\begin{equation} \label{C}
t_{n,3}(r,s_1-1,s_2,\overline{s})=t_{n,2}(r,s_2,s_1-1,\overline{s}) \; \text{et vérifie} \; iv).
\end{equation}

Supposons $s_1>1$. 
Alors, $(n; s_1-1,s_2,\overline{s})\in\Lambda$ par minimalité. 
On s'aperçoit que dans ce cas, $\eqref{C}$ découle du $i)$ et du $iv)$. 

Supposons maintenant $s_1=1$.
Dans ce cas, \eqref{C} équivaut à 
\begin{equation} \label{le 28}
t_{n,3}(r,0,s_2,\overline{s})=t_{n,2}(r,s_2,0,\overline{s}) \; \text{et vérifie} \; iv).
\end{equation}

Le cas $n=2$ découle immédiatement du théorème \ref{cas où n=1 et v p a vaut 0}.
Le cas $n=3$ et $p\nmid v_F(a_n)$ correspond au $iv)$ du théorème \ref{2Théorème fondamental 2.}. 
Enfin, supposons $n\geq 4$ ou $n=3$ si $p\mid v_F(a_n)$.
Alors, $(n-1; s_2,\overline{s})\in\Lambda$ et le même argument que pour montrer \eqref{égalité dans le cas s 1 égal s 2} dans le cas où $s=1$ s'applique pour en déduire \eqref{le 28}.
  
  Ainsi, $(n,s_1,s_2,\overline{s})\in\Lambda$, ce qui est absurde. La proposition s'ensuit.
 
 \end{proof}
 
 Soient $n\geq 2$, $s_1,\dots,s_n\geq 1$ et $E\subset \{1,\dots,n\}$ un ensemble non vide tel que pour tout $i\in E$, on ait $p\mid v_F(a_i)$ et $s:=s_i\geq 0$.
 Fixons $k\in E$.
 En appliquant le $i)$ et le $iv)$ de la proposition \ref{Théorème fondamental} à tout couple $(l,k)\in E^2$, on en déduit que les $(t_{n,l+1}(r,s_1+\delta_{1,l},\dots,s_n+\delta_{n,l}))_{l\in E}$ sont tous égaux à $t_{n,k+1}(r,s_1+\delta_{1,k},\dots,s_n+\delta_{n,k})$ et ont la propriété $(P)$.
 Ceci montre donc que les hypothèses nécessaires à l'application du Lemme \ref{unicité du saut 2} sont en réalité vérifiées.
La proposition \ref{Théorème fondamental} permet donc de simplifier l'énoncé du Lemme \ref{unicité du saut 2} qui devient le corollaire suivant.

\begin{cor} \label{unicité du saut 3}
Soient $n\geq 2$, $s_1,\dots,s_n\geq 1$ et $E\subset\{1,\dots,n\}$ un ensemble non vide tel que pour tout $i\in E,$ on ait $p\mid v_F(a_i)$ et $s:=s_i\geq 0$. Fixons $k\in E$.
Alors, $L_{r,s'_1,\dots,s_n'}/L_{r,s_1,\dots,s_n}$ où 
\begin{equation*}
s_i'=\begin{cases}
s_i+1 \; \text{si} \; i\in E \\
s_i \; \text{sinon}
\end{cases}
\end{equation*}
 possède un unique saut égal à $t_{n,k+1}(r,s'_1,\dots,s_n')$. 
 De plus, pour tout $i\in E$, on a  
\begin{equation*} 
 t_{n,k+1}(r,s'_1,\dots,s_n')=t_{n,i+1}(r,s_1+\delta_{1,i},\dots,s_n+\delta_{n,i}).
 \end{equation*}
\end{cor}

En particulier, si $E=\{1,2\}$, $k=1$ et $p\mid v_F(a_1),v_F(a_2)$, alors pour tout $s\geq 0$,
\begin{equation} \label{application de l'unicité du saut dans le cas où p divise a_n}
t_{n,2}(r,s+1,s+1,s_3,\dots,s_n)=t_{n,2}(r,s+1,s,s_3,\dots,s_n) .
\end{equation}

\begin{rqu} \label{Remarque sur le calcul de r,1,1}
 Soit $n\geq 2$. Si $p\mid v_F(a_n)$, le corollaire \ref{unicité du saut 3} appliqué à $k=1$, $s=0$ et $E=\{1,\dots,n\}$ montre que $L_{r,1,\dots,1}/L_{r,0,\dots,0}$ possède un unique saut égal à $t_{n,2}(r,1,0,\dots,0)$.
 D'après le $iii)$ du théorème \ref{cas où n=1 et v p a vaut 0}, ce saut vaut $1$.
\end{rqu}

\begin{cor} \label{Corollaire où s_1 est le minimum des s_i}
Soit $n\geq 2$ et $s_1,\dots,s_n\geq 1$. Si $s_1=\min\{s_1,\dots,s_{n-1}\}$ alors,
\begin{equation*}
t_{n,2}(r,s_1,\dots,s_n)=t_{n,2}(r,s_1,\dots,s_1,s_n).
\end{equation*}
De plus, si $p\mid v_F(a_n)$ et si $s_1\leq s_n$ alors, $t_{n,2}(r,s_1,\dots,s_n)=t_{n,2}(r,s_1,\dots,s_1)$.
\end{cor}

\begin{proof}

En appliquant le $ii)$ de la proposition \ref{Théorème fondamental} à $l=1$, $k=2$ et à $\sigma=id$, on en déduit que 
\begin{equation*}
t_{n,2}(r,s_1,s_2,\dots,s_n)=t_{n,2}(r,s_1,s_1,s_3,\dots,s_n).
\end{equation*} 
En appliquant de nouveau le $ii)$ de la proposition \ref{Théorème fondamental} à $l=1$, $k=3$ et à $\sigma=id$, on en déduit que  
\begin{equation*}
t_{n,2}(r,s_1,s_1,s_3,\dots,s_n)=t_{n,2}(r,s_1,s_1,s_1, s_4,\dots,s_n).
\end{equation*}
 Par une récurrence immédiate, on en déduit que 
\begin{equation*} 
 t_{n,2}(r,s_1,s_2,\dots,s_n)=t_{n,2}(r,s_1,\dots,s_1,s_n).
 \end{equation*}
 La seconde partie du corollaire devient claire puisque dans ce cas, on peut de nouveau appliquer la récurrence.
\end{proof}

Supposons $p\mid v_F(a_n)$. 
En unissant la Remarque \ref{Remarque sur le calcul de r,1,1} et le corollaire \ref{Corollaire où s_1 est le minimum des s_i}, on en déduit que pour tous $s_2,\dots,s_n\geq 0$,
\begin{equation} \label{unification de la rq et du cor}
t_{n,2}(r,1,s_2,\dots,s_n)=t_{n,2}(r,1,\dots,1)=1.
\end{equation} 

Remarquons que la proposition \ref{Théorème fondamental} ne traite pas le cas où $k=n$ et $p\nmid v_F(a_n)$. 
On conclura donc cette sous-section en donnant un analogue de la proposition \ref{Théorème fondamental} dans ce cas.

\begin{lmm} \label{unicité du saut ?}
Soient $n\geq 2$, $s_1,\dots,s_n\geq 0$ et $E\subset \{1,\dots,n-1\}$ un ensemble non vide tels que pour tout $i\in E$, on ait $s:=s_i=s_n+1$. 
Supposons que les nombres $(t_{n,k+1}(r,s_1+\delta_{1,k},\dots,s_n+\delta_{n,k}))_{k\in E\cup\{n\}}$ aient la propriété $(P)$ et soient tous égaux à un certain $t$. 
Posons 
\begin{equation*}
s_i'=\begin{cases} 
s_i+1 \; \text{si} \; i\in E\cup \{n\} \\
s_i \; \text{sinon}.
\end{cases}. 
\end{equation*}
Alors, 
\begin{enumerate}[i)]
\item $L_{r,s'_1,\dots,s_n'}/L_{r,s_1,\dots,s_n}$ possède un unique saut égal à $t$;
\item pour tout $k\in E\cup \{n\}$, on a $t_{n,k+1}(r,s'_1,\dots,s_n')=t$;
\end{enumerate}
\end{lmm}

\begin{proof}
Notons $K=L_{r,s_1,...,s_n}$ et $K'=L_{r,s'_1,\dots,s_n'}$.
Alors, $[K' : K]=p^{\sum_{i=1}^n s'_i-s_i}$. 
D'après le Lemme \ref{lmm que j'ai rajouter sur la condition suffisante} (appliqué à $\alpha_i=a_i^{1/p^s}$ (pour $i\in E$) et $\alpha_n= a_n^{1/p^{s-1}}$) et le Lemme \ref{sous-corps d'une extension kumérienne.}, un sous-corps intermédiaires de $K'/K$ de degré $p$ sur $K$ est de la forme $K(\gamma^{1/p})$ où $\gamma= \alpha_n^{x_n} \prod_{i\in E} \alpha_i^{x_i}$ avec les $x_i,x_n\in\{0,\dots,p-1\}$ non tous nuls.

Si  $x_i=0$ pour tout $i\in E$, alors $x_n\neq 0$ et $\gamma=\alpha_n^{x_n}$.
Dans ce cas, le saut de $K(\gamma^{1/p})/L$ vaut $t_{n,n+1}(r,s_1,\dots,s_{n-1},s_n+1)$, i.e. $t$.

  Soit $k\in E$ tel que $x_k\neq 0$. 
  Notons $\gamma_1=a_n^{p x_n} \prod_{i\in E} a_i^{x_i}$.
On constate que $p\mid v_F(\gamma_1)$ puisque $p\mid v_F(a_i)$ pour tout $i\in E $.
De plus, il est clair que 
\begin{equation*}
K=F(\zeta_{p^r}, a_1^{1/p^{s_1}},\dots,a_{k-1}^{1/p^{s_{k-1}}}, \gamma_1^{1/p^s},a_{k+1}^{1/p^{s_{k+1}}},\dots,a_n^{1/p^{s_n}}).
\end{equation*}
Remarquons que $K(\gamma^{1/p})=K(\gamma_1^{1/p^{s+1}})$. 
Comme $t_{n,k+1}(r,s_1+\delta_{1,k},\dots,s_n+\delta_{n,k})$ a la propriété $(P)$, on en déduit, en remplaçant $a_k$ par $\gamma_1$, que le saut de $K(\gamma^{1/p})/L$ est égal à $t_{n,k+1}(r,s_1+\delta_{1,k},\dots,s_n+\delta_{n,k})$, i.e. $t$.

Ainsi, le saut de toute sous-extension intermédiaire de $L'/L$ de degré $p$ sur $K$ vaut $t$.
Le Lemme \ref{Unicité du saut, version théorique} permet de montrer $i)$ et $ii)$ s'ensuit car les $t_{n,i+1}(r,s'_1,\dots,s_n')$ avec $i\in E$ sont des sauts de $K'/K$.
\end{proof}

\begin{prop} \label{Théorème fondamental 2.}
Soient $n\geq 2$ et $s_1,\dots,s_n \geq 0$. 
Soit $k\in\{1,\dots,n-1\}$ tel que $1\leq s_k\leq s_n$.
Alors,  

\begin{enumerate} [i)]
\item \begin{enumerate} [a)]
\item $t_{n,k+1}(r,s_1,\dots,s_n)=t_{n,k+1}(r,s_1,\dots,s_{n-1},s_k-1)$;
\item $t_{n,n+1}(r,s_1,\dots,s_n,\dots,s_{n-1},s_k)=t_{n,n+1}(r,s_1,\dots,s_{n-1},s_k)$ (où l'on a noté, ici et après, $s_1,\dots,s_n,\dots,s_{n-1},s_k$ à la place de 
\begin{equation*}
s_1,\dots,s_{k-1},s_n,s_{k+1},\dots,s_{n-1},s_k;
\end{equation*}

\end{enumerate} 

\item \begin{enumerate} [a)]
\item $t_{n,n+1}(r,s_1,\dots,s_n)=p t_{n,n+1}(r,s_1,\dots,s_k-1,\dots,s_n) + \\
(1-p)t_{n,k+1}(r,s_1,\dots,s_n)$;
\item $t_{n,k+1}(r,s_1,\dots,s_n,\dots,s_k)=p t_{n,k+1}(r,s_1,\dots,s_n,\dots,s_k-1) \\
+ (1-p)t_{n,n+1}(r,s_1,\dots,s_n,\dots,s_k)$ sauf si $s_n=s_k$;
\end{enumerate} 
\item $t_{n,n+1}(r,s_1,\dots,s_n)=t_{n,k+1}(r,s_1,\dots,s_n+1,\dots, s_k-1)$;
\item pour $i\in\{k,n\}$, le saut $t_{n,i+1}(r,s_1,\dots,s_n)$ a la propriété $(P)$.
\end{enumerate}

\end{prop}

\begin{proof}
 Remarquons que le cas $n=2$ est précisément la proposition \ref{2Théorème fondamental 2.}.
 Ainsi, on supposera $n\geq 3$.
 Quitte à permuter $a_1^{1/p^{s_1}}$ avec $a_k^{1/p^{s_k}}$, on peut supposer, sans perte de généralité, $k=1$. 
Notons $\Lambda$ l'ensemble des $(n;s_1,\dots,s_n)$ vérifiant les hypothèses de la proposition et les affirmations $i)$ à $iv)$ de la proposition.
Supposons par l'absurde qu'il existe $(n; s_1,\dots,s_n)\notin\Lambda$ vérifiant les hypothèses de la proposition. 
On le choisit minimal pour l'ordre lexicographique.
Afin d'alléger les notations, notons 
\begin{equation*}
\overline{s}=(s_2,\dots,s_{n-1}) \; \text{et} \;  \overline{S}=(s_3,\dots,s_{n-1}).
\end{equation*}

 La preuve se fera en cinq étapes (décrites ci-dessous) très largement indépendantes (seul le lemme ci-dessous sera utilisé dans la première et la troisième étape). 
 
\textit{Première étape} : $ib)$ et $ii b)$ sont vérifiés.

\textit{Deuxième étape} : $ia),\; ii a)$ et $iv)$ sont vérifiés.

\textit{Troisième étape} :  $iii)$ est vérifié si $s_1>1$.

\textit{Quatrième étape} : $iii)$ est vérifié si $\min\{s_2,s_n\}=1$.

\textit{Cinquième étape} : $iii)$ est vérifié si $\min\{s_2,s_n\}>1$. \\

\textit{Première étape} :
 
Si $s_n=s_1$, alors $ib)$ est trivialement vérifié et $ii b)$ est vide.
Supposons $s_n=s_1+1$.
Par minimalité, $(n;s_1,\overline{s},s_1)\in\Lambda$ et on déduit du $iii)$ que 
\begin{equation*}
t_{n,2}(r,s_1+1,\overline{s},s_1-1)=t_{n,n+1}(r,s_1,\overline{s},s_1).
\end{equation*}
Les conditions du Lemme \ref{unicité du saut ?} appliquées à $s=s_1$ et à $E=\{1\}$ sont ainsi vérifiées.
Par conséquent, le $ii)$ du Lemme \ref{unicité du saut ?} permet d'en déduire que
\begin{equation} \label{blabla}
t_{n,2}(r,s_1+1,\overline{s},s_1)=t_{n,n+1}(r,s_1+1,\overline{s},s_1)=t_{n,n+1}(r,s_1,\overline{s},s_1),
\end{equation}
ce qui montre $i b)$ et $ii b)$.
 
 Supposons maintenant $s_n\geq s_1+2$.  
  Ainsi, $(s_1,\overline{s},s_n-1)\in\Lambda$ par minimalité. 
 Du $ib)$, il s'ensuit que $t_{n,n+1}(r,s_n-1,\overline{s},s_1)=t_{n,n+1}(r,s_1,\overline{s},s_1)$.
 Par conséquent, le lemme ci-dessous, appliqué à $m=s_1$ et à $l=s_n$, montre $i b)$ et $ii b)$. 

\begin{lmm*} \label{Lemme dont je dois retirer le numéro.}
Soient $m\in\{1,\dots,s_1\}$ et $ l\in\{m+2, \dots,s_n+1-\delta_{m,s_1}\}$. 
Alors,

\begin{enumerate} [i)]
\item $t_{n,n+1}(r,l,\overline{s},m)=t_{n,n+1}(r,l-1,\overline{s},m)$;
\item $t_{n,2}(r,l,\overline{s},m)=p t_{n,2}(r,l,\overline{s},m-1)+(1-p)t_{n,n+1}(r,l,\overline{s},m)$.
\end{enumerate}

\end{lmm*}

\begin{proof}
La proposition  \ref{croissance des suites} appliquée à $K=L_{r,0,\overline{s},m-1}$ et à $\alpha=a_1$ montre (car $m+1 < l$) que 
\begin{equation*}
t_{n,2}(r,l,\overline{s},m-1)>t_{n,2}(r,m+1,\overline{s},m-1).
\end{equation*} 
Par minimalité, $(n;m,\overline{s},m)\in\Lambda$ (car $m\leq s_1< s_n$).
 On déduit alors du $iii)$ que 
\begin{equation} \label{le label du lemme dans le théorème fondamental}
t_{n,2}(r,m+1,\overline{s},m-1)=t_{n,n+1}(r,m,\overline{s},m).
\end{equation}
Si $m=s_1$ alors, $s_1+2\leq l\leq s_n$ et donc $(n;m,\overline{s},l-1)\in\Lambda$.
Si $m<s_1$, alors $(n;m,\overline{s},l-1)\in\Lambda$ par minimalité.
Dans les deux cas, $(n;m,\overline{s},l-1)\in\Lambda$.
Ainsi, d'après le $i b)$, 
\begin{equation*}
t_{n,n+1}(r,l-1,\overline{s},m)=t_{n,n+1}(r,m,\overline{s},m).
\end{equation*}
On en déduit donc que $t_{n,n+1}(r,l-1,\overline{s},m)< t_{n,2}(r,l,\overline{s},m-1)$.
D'après le $i)$ de la proposition \ref{calcul de la différente version appliquée.},
\begin{multline} \label{le multline du lemme}
t_{n,n+1}(r,l,\overline{s},m)-t_{n,2}(r,l,\overline{s},m)=  \\
p(t_{n,n+1}(r,l-1,\overline{s},m)-t_{n,2}(r,l,\overline{s},m-1))<0,
\end{multline}
la dernière inégalité venant de \eqref{le label du lemme dans le théorème fondamental}.
 Le $ii)$ de la proposition \ref{calcul de la différente version appliquée.} montre que $t_{n,n+1}(r,l,\overline{s},m)=t_{n,n+1}(r,l-1,\overline{s},m)$ (i.e.  le $i)$ de ce lemme) et le $ii)$ se déduit du \eqref{le multline du lemme} car 
\begin{equation*} 
t_{n,2}(r,l,\overline{s},m)= t_{n,n+1}(r,l,\overline{s},m)-p(t_{n,n+1}(r,l-1,\overline{s},m)-t_{n,2}(r,l,\overline{s},m-1)).
\end{equation*}

\end{proof}

\vspace{0.1 mm}

\textit{Deuxième étape} :

D'après le $i)$ de la proposition \ref{calcul de la différente version appliquée.}, 
\begin{multline} \label{equation utile une fois}
t_{n,2}(r,s_1,\overline{s},s_n)-t_{n,n+1}(r,s_1,\overline{s},s_n) = \\
p(t_{n,2}(r,s_1,\overline{s},s_n-1)-t_{n,n+1}(r,s_1-1,\overline{s},s_n)).
\end{multline}
Si $s_1<s_n$ alors, $(n;s_1,\overline{s},s_n-1)\in\Lambda$ et on déduit du $i a)$ que 
\begin{equation} \label{égalité servant pour l'inégalité}
t_{n,2}(r,s_1,\overline{s},s_n-1)=t_{n,2}(r,s_1,\overline{s},s_1-1).
\end{equation}
Ainsi, \eqref{égalité servant pour l'inégalité} est valable pour tout $s_1\leq s_n$.
 Montrons maintenant que
\begin{equation} \label{D}
t_{n,2}(r,s_1,\overline{s},s_n-1)<t_{n,n+1}(r,s_1-1,\overline{s},s_n)
\end{equation}
en distinguant le cas $s_1>1$ du cas $s_1=1$.

Si $s_1>1$ alors, $(n; s_1-1,\overline{s},s_1-1)\in\Lambda$ par minimalité. En utilisant le $iii)$,
\begin{equation} \label{E}
t_{n,n+1}(r,s_1-1,\overline{s},s_1-1)=t_{n,2}(r,s_1,\overline{s},s_1-2).
\end{equation}
Remarquons que \eqref{E} et le $iv)$ montre que les hypothèses du Lemme \ref{unicité du saut ?} appliquées à $s=s_1-1$ et à $E=\{1\}$ sont vérifiées.
Le $ii)$ du Lemme \ref{unicité du saut ?} montre que
\begin{equation} \label{AA}
t_{n,2}(r,s_1,\overline{s},s_1-2)=t_{n,2}(r,s_1,\overline{s}, s_1-1).
\end{equation}

En combinant \eqref{E} et \eqref{AA}, on en déduit que
\begin{equation} \label{c'est le * de mon brouillon}
t_{n,2}(r,s_1,\overline{s},s_1-1)=t_{n,n+1}(r,s_1-1,\overline{s},s_1-1).
\end{equation}
 La proposition  \ref{croissance des suites} appliquée à $K=L_{r,s-1,\overline{s},0}$ et à $\alpha=a_n$ prouve (car $s_1-1 <s_n$) que $t_{n,n+1}(r,s_1-1,\overline{s},s_1-1)< t_{n,n+1}(r,s_1-1,\overline{s},s_n)$.
 En utilisant \eqref{égalité servant pour l'inégalité}, \eqref{c'est le * de mon brouillon} et cette inégalité, on obtient $t_{n,2}(r,s_1,\overline{s},s_n-1)<t_{n,n+1}(r,s_1-1,\overline{s},s_n)$, qui est bien l'inégalité \eqref{D} dans le cas où $s_1>1$. 
 
 Supposons maintenant que $s_1=1$.
 D'après \eqref{unification de la rq et du cor},
\begin{equation*}
t_{n,2}(r,s_1,\overline{s},s_1-1)=t_{n,2}(r,1,\overline{s},0)=1.
\end{equation*}
De \eqref{égalité servant pour l'inégalité}, il suffit de prouver que $t_{n,n+1}(r,0,\overline{s},s_n)>1$ pour en déduire \eqref{D}.
La proposition \ref{croissance des suites} appliquée à $K=L_{r,0,\overline{s},0}$ et à $\alpha=a_n$ donne (car $s_n \geq 1$)
\begin{equation*}
t_{n,n+1}(r,0,\overline{s},s_n)\geq t_{n,n+1}(r,0,\overline{s},1).
\end{equation*}
Il suffit donc de montrer que 
\begin{equation*}
t_{n,n+1}(r,0,\overline{s},1)>1.
\end{equation*}
Comme $n\geq 3$, alors $(n-1;\overline{s},1)\in\Lambda$ par minimalité et on déduit du $ii a)$ que 
\begin{equation*}
t_{n,n+1}(r,0,\overline{s},1)=t_{n,n+1}(r,0,1,\overline{S},1).
\end{equation*}  
Du $iii)$, on obtient que $t_{n,n+1}(r,0,1,\overline{S},1)=t_{n,3}(r,0,2,\overline{S},0)$ et donc,
\begin{equation} \label{F}
t_{n,n+1}(r,0,\overline{s},1)=t_{n,3}(r,0,2,\overline{S},0).
\end{equation} 
La proposition  \ref{croissance des suites} appliquée à $K=L_{r,0,0,\overline{S},0}$ et à $\alpha=a_2$ montre que 
\begin{equation*}
t_{n,n+1}(r,0,2,\overline{S},0)>t_{n,3}(r,0,1,\overline{S},0)=1
\end{equation*}
(l'égalité découle du \eqref{unification de la rq et du cor}).
De cette inégalité et \eqref{F}, on en déduit \eqref{D}.

De \eqref{equation utile une fois} et \eqref{D}, on obtient que $t_{n,2}(r,s_1,\overline{s},s_n)<t_{n,n+1}(r,s_1,\overline{s},s_n)$.
 Le $ii)$ de la proposition \ref{calcul de la différente version appliquée.} appliqué à \eqref{equation utile une fois} montre que
\begin{equation*} 
 t_{n,2}(r,s_1,\overline{s},s_n)=t_{n,2}(r,s_1,\overline{s},s_n-1).
 \end{equation*}
En utilisant \eqref{égalité servant pour l'inégalité}, on en conclut que
\begin{equation*}  
  t_{n,2}(r,s_1,\overline{s},s_n)=t_{n,2}(r,s_1,\overline{s},s_1-1).
  \end{equation*}
On déduit ainsi du \eqref{equation utile une fois} que 
\begin{equation} \label{G} 
 t_{n,n+1}(r,s_1,\overline{s},s_n)=p t_{n,n+1}(r,s_1-1,\overline{s},s_n)+(1-p)t_{n,2}(r,s_1,\overline{s},s_n),
\end{equation}
ce qui montre $i a)$ et $ii a)$.

Montrons maintenant que $(n; s_1,\overline{s},s_n)$ vérifie $iv)$.
Tout d'abord, on déduit du $ia)$ que $t_{n,2}(r,s_1,\overline{s},s_n)=t_{n,2}(r,s_1,\overline{s},s_1-1)$.

Si $s_1=1$ alors, $t_{n,2}(r,s_1,\overline{s},s_n)=t_{n,2}(r,1,\overline{s},0)=1$ et vérifie donc $iv)$.

Si $s_1>1$ alors, par minimalité, $(n,s_1,\overline{s},s_1-1)\in\Lambda$. 
Ainsi, $t_{n,2}(r,s_1,\overline{s},s_1-1)$, et donc $t_{n,2}(r,s_1,\overline{s},s_n)$, vérifie $iv)$. 

Il ne nous reste plus qu'à montrer que $t_{n,n+1}(r,s_1,\overline{s},s_n)$ vérifie $iv)$. 
De \eqref{G}, il suffit de montrer que  $t_{n,n+1}(r,s_1-1,\overline{s},s_n)$ vérifie $iv)$.

Si $s_1>1$, alors $(n;s_1-1,\overline{s},s_n)\in\Lambda$, et donc $t_{n,n+1}(r,s_1-1,\overline{s},s_n)$ vérifie $iv)$.
Si $s_1=1$, alors $(n-1;\overline{s},s_n)\in\Lambda$, et donc $t_{n,n+1}(r,0,\overline{s},s_n)$ vérifie $iv)$.
Ainsi, $(n; s_1,\overline{s},s_n)$ vérifie $iv)$. \\

\textit{Troisième étape} : 
Supposons $s_1>1$.
Alors, d'après le $ii a)$,  
\begin{equation} \label{H} 
 t_{n,n+1}(r,s_1,\overline{s},s_n) =p t_{n,n+1}(r,s_1-1,\overline{s},s_n)+(1-p)t_{n,2}(r,s_1,\overline{s},s_n).
 \end{equation}
Par minimalité, $(n;s_1-1,\overline{s},s_n)\in\Lambda$. 
Du $iii)$, on obtient que
\begin{equation} \label{I}
t_{n,n+1}(r,s_1-1,\overline{s},s_n)=t_{n,2}(r,s_n+1,\overline{s},s_1-2).
\end{equation}
En utilisant également le $ia)$, on a $t_{n,2}(r,s_1,\overline{s},s_n)=t_{n,2}(r,s_1,\overline{s},s_1-1)$.
 De cette égalité et du \eqref{c'est le * de mon brouillon}, on en déduit que 
\begin{equation} \label{II} 
t_{n,2}(r,s_1,\overline{s},s_n)=t_{n,n+1}(r,s_1-1,\overline{s},s_1-1).
\end{equation}
De plus, le $i b)$ donne 
\begin{equation} \label{J}
t_{n,n+1}(r,s_n,\overline{s},s_1-1)=t_{n,n+1}(r,s_1-1,\overline{s},s_1-1).
\end{equation}
Ainsi, de \eqref{H}, \eqref{I}, \eqref{II} et \eqref{J}, il s'ensuit que 
\begin{multline} \label{première formule de l'égalité du iii)} 
 t_{n,n+1}(r,s_1,\overline{s},s_n) =p t_{n,2}(r,s_n+1,\overline{s},s_1-2)+(1-p)t_{n,n+1}(r,s_1-1,\overline{s},s_1-1).
 \end{multline}
Le $ii)$ du lemme dans la preuve appliqué à $m=s_1-1$ et à $l=s_n+1$ donne :
\begin{multline} \label{K}
t_{n,2}(r,s_n+1,\overline{s},s_1-1)= \\
pt_{n,2}(r,s_n+1,\overline{s},s_1-2)+(1-p)t_{n,n+1}(r,s_n+1,\overline{s},s_1-1).
\end{multline}
De plus, d'après le $i b)$,
\begin{equation} \label{L}
t_{n,n+1}(r,s_n+1,\overline{s},s_1-1)=t_{n,n+1}(r,s_1-1,\overline{s},s_1-1).
\end{equation}
 En injectant \eqref{L} dans \eqref{K}, que l'on compare ensuite à \eqref{première formule de l'égalité du iii)}, on en déduit $iii)$, ce qui conclut la troisième étape. \\

 \`A partir de maintenant, supposons $s_1=1$. 
Si $(n-1;s'_2,\dots,s_n')\leq (n-1;s_2,\dots,s_n)$ pour l'ordre lexicographique alors, d'après le $i a)$ et \eqref{unification de la rq et du cor}, on en déduit que 

\begin{equation} \label{tout vaut 1}
t_{n,2}(r,1,s_2',\dots,s_n')=t_{n,2}(r,1,s_1',\dots,s_{n-1}',0)=1.
\end{equation}
 Ainsi, d'après le $iia)$ et \eqref{tout vaut 1}, il s'ensuit que 
\begin{equation}  \label{fdsddh}
 t_{n,n+1}(r,1,\overline{s},s_n)=p t_{n,n+1}(r,0,\overline{s},s_n)+1-p.
 \end{equation}
 
\textit{Quatrième étape} : 
  Supposons que $s'=1$. 
  On souhaite montrer que   
\begin{equation} \label{objectif de la cinquième étape.}  
  t_{n,n+1}(r,1,\overline{s},s_n)=t_{n,2}(r,s_n+1,\overline{s},0).
  \end{equation}

Dans un premier temps, nous allons montrer que 
\begin{equation} \label{ce que je dois montrer}
t_{n,n+1}(r,0,\overline{s}, s_n):= t_{n,n+1}(r,0,s_2,\overline{S},s_n)=t_{n,3}(r,0,s_n+1,\overline{S},0).
\end{equation} 
  
Supposons $s_2=1$. 
Comme $(n-1;\overline{s},s_n)\in\Lambda$, on en déduit \eqref{ce que je dois montrer} en utilisant $iii)$.
Supposons $s_n=1$ et $s_2\geq 2$.
Comme $(n-1;\overline{s},1)\in\Lambda$, on en déduit, en utilisant $i a)$ que $t_{n,n+1}(r,0,s_2,\overline{S},1)=t_{n,n+1}(r,0,1,\overline{S},1)$.
En utilisant maintenant le $iii)$, il s'ensuit que $t_{n,n+1}(r,0,1,\overline{S},1)=t_{n,n+1}(r,0,2,\overline{S},0)$, i.e. \eqref{ce que je dois montrer}.

En utilisant \eqref{tout vaut 1}, \eqref{fdsddh} et \eqref{ce que je dois montrer}, on en déduit que 
\begin{equation*} 
t_{n,n+1}(r,1,\overline{s},s_n) =p t_{n,3}(r,0,s_n+1,\overline{S},0)+ 1-p = t_{n,3}(r,1,s_n+1,\overline{S},0).
\end{equation*}        
 La dernière égalité se déduit du \eqref{tout vaut 1} et du $iii)$ de la proposition \ref{Théorème fondamental} appliqué à $\sigma=id$, $l=1$ et $k=2$ (comme $n\geq 3$, alors $p\mid v_F(a_2)$).
 Le $i)$ de la proposition \ref{Théorème fondamental} appliqué à $k=n$ et à $\sigma=id$ montre que $t_{n,3}(r,1,s_n+1,\overline{S},0)=t_{n,2}(r,s_n+1,1,\overline{S},0)$.
 Ainsi,
\begin{equation} \label{version presque finale} 
 t_{n,n+1}(r,1,\overline{s},s_n)=t_{n,2}(r,s_n+1,1,\overline{S},0).
 \end{equation}

 Si $s_2=1$, alors $t_{n,2}(r,s_n+1,1,\overline{S},0)=t_{n,2}(r,s_n+1,\overline{s},0)$ et \eqref{version presque finale} montre \eqref{objectif de la cinquième étape.}.
 Si $s_n=1$ et $s_2\geq 2$ alors, $t_{n,n+1}(r,1,\overline{s},1)=t_{n,2}(r,2,1,\overline{S},0)$ d'après \eqref{version presque finale}.
 En appliquant \eqref{application de l'unicité du saut dans le cas où p divise a_n} à $s=1$, on a $t_{n,2}(r,2,1,\overline{S},0)=t_{n,2}(r,2,2,\overline{S},0)$.
 Le $ii)$ de la proposition \ref{Théorème fondamental} appliqué à $\sigma=id, l=1$ et $k=2$ montre que $t_{n,2}(r,2,2,\overline{S},0)=t_{n,2}(r,2,s_2,\overline{S},0)$.
 Ainsi, $t_{n,n+1}(r,1,\overline{s},1)=t_{n,2}(r,2,\overline{s},0)$, ce qui montre \eqref{objectif de la cinquième étape.}. \\

\textit{Cinquième étape} :
Supposons $s'>1$.
D'après le $ii)$ de la proposition \ref{Théorème fondamental} appliqué à $\sigma=id$ et à $(l,k)=(1,2)$, on a, en utilisant également \eqref{tout vaut 1},
\begin{equation} \label{pfff ...}
t_{n,3}(r,1,s',\overline{S},s_n)=p t_{n,3}(r,0,s',\overline{S},s_n)+1-p.
\end{equation}
Par minimalité, $(n-1;s',\overline{S},s_n)\in\Lambda$.
Ainsi, d'après le $1a)$,
\begin{equation} \label{M}
t_{n,3}(r,0,s',\overline{S},s_n)=t_{n,3}(r,0,s',\overline{S},s'-1)
\end{equation}
Toujours par minimalité, $(n-1, s', \overline{S}, s_n-1)\in\Lambda$. 
Donc, toujours d'après le $1a)$, 
\begin{equation} \label{N}
t_{n,3}(r,0,s',\overline{S}, s_n-1)=t_{n,3}(r,0,s',\overline{S},s'-1).
\end{equation}
De \eqref{M} et \eqref{N}, on en déduit que $t_{n,3}(r,0,s',\overline{S},s_n)=t_{n,3}(r,0,s',\overline{S},s_n-1)$.
Il en résulte, en utilisant \eqref{pfff ...}, que
\begin{equation} \label{troisième formule pour conclure le iii)}
t_{n,3}(r,1,s',\overline{S},s_n) =p t_{n,3}(r,0,s',\overline{S},s_n-1)+1-p= t_{n,3}(r,1,s',\overline{S},s_n-1),
\end{equation}
la dernière égalité découlant du \eqref{tout vaut 1} et du $ii)$ de la proposition \ref{Théorème fondamental} appliqué à $\sigma=id$ et à $(l,k)=(1,2)$.

Par minimalité, $(n-1;s'-1,\overline{S},s'-1)\in\Lambda$ et on déduit du $iii)$ que 
\begin{equation*}
t_{n,3}(r,0,s',\overline{S},s'-2)=t_{n,n+1}(r,0,s'-1,\overline{S},s'-1).
\end{equation*}
Ainsi, les conditions du Lemme \ref{unicité du saut ?} appliquées à $s=s'-1$ et à $E=\{2\}$ sont vérifiées. 
Donc,
\begin{equation*}
t_{n,3}(r,0,s',\overline{S},s'-2)=t_{n,3}(r,0,s',\overline{S},s'-1)=t_{n,n+1}(r,0,s'-1,\overline{S},s'-1).
\end{equation*}

De \eqref{M}, on a $t_{n,3}(r,0,s',\overline{S},s_n)=t_{n,n+1}(r,0,s'-1,\overline{S},s'-1)$.
En utilisant \eqref{pfff ...} pour la première égalité et le $iii)$ de la proposition \ref{Théorème fondamental} appliqué à $\sigma=id$ et à $(l,k)=(1,n)$ pour la seconde, on en déduit que
\begin{align*}
t_{n,3}(r,1,s',\overline{S},s_n) & =p t_{n,n+1}(r,0,s'-1,\overline{S},s'-1)+1-p\\
& = t_{n,n+1}(r,1,s'-1,\overline{S},s'-1).
\end{align*} 
Par minimalité, $(n;1,s'-1,\overline{S},s'-1)\in\Lambda$ et on déduit du $iii)$ que 
\begin{equation*}
t_{n,n+1}(r,1,s'-1,\overline{S},s'-1)=t_{n,2}(r,s',s'-1,\overline{S},0).
\end{equation*}
Par \eqref{application de l'unicité du saut dans le cas où p divise a_n}, $t_{n,2}(r,s',s'-1,\overline{S},0)=t_{n,2}(r,s',s',\overline{S},0)$.
Ainsi,
\begin{equation} \label{quatrième formule pour conclure iii)}
t_{n,3}(r,1,s',\overline{S},s_n)=t_{n,2}(r,s',s',\overline{S},0).
\end{equation}
Le $i)$ de la proposition \ref{calcul de la différente version appliquée.} appliqué à $l=n+1$ et à $k=3$ donne  
 \begin{multline} \label{O}
 t_{n,n+1}(r,1,s',\overline{S},s_n)-t_{n,3}(r,1,s',\overline{S},s_n) = \\ 
 p(t_{n,n+1}(r,1,s'-1,\overline{S},s_n)-t_{n,3}(r,1,s',\overline{S},s_n-1)).
 \end{multline}
Par minimalité, $(n;1,s'-1,\overline{S},s_n)\in\Lambda$ et on déduit du $iii)$ que 
\begin{equation*}
t_{n,n+1}(r,1,s'-1,\overline{S},s_n)=t_{n,2}(r,s_n+1,s'-1,\overline{S},0).
\end{equation*}
 Il s'ensuit, en utilisant \eqref{troisième formule pour conclure le iii)}, \eqref{quatrième formule pour conclure iii)} et \eqref{O}, que 
\begin{equation*} 
t_{n,n+1}(r,1,s',\overline{S},s_n)= pt_{n,2}(r,s_n+1,s'-1,\overline{S},0) + (1-p)t_{n,3}(r,s',s',\overline{S},0).
\end{equation*}
D'après le $iii)$ de la proposition \ref{Théorème fondamental} appliqué à $\sigma=(1 2)$ et à $(l,k)=(1,2)$,
\begin{multline*}
t_{n,2}(r,s_n+1,s',\overline{S},0)=p t_{n,2}(r,s_n+1,s'-1,\overline{S},0)+(1-p)t_{n,3}(r,s_n+1,s',\overline{S},0).
\end{multline*}
Le $ii)$ de la proposition \ref{Théorème fondamental} appliqué à $\sigma=(2,3)$ et à $(l,k)=(1,2)$ donne
\begin{equation*}
t_{n,3}(r,s_n+1,s',\overline{S},0)=t_{n,3}(r,s',s',\overline{S},0).
\end{equation*}
Des trois dernières égalités, on en déduit que
\begin{equation*}
t_{n,n+1}(r,1,\overline{s},s_n)=t_{n,2}(r,s_n+1,s',\overline{S},0).
\end{equation*}
Si $s'=s_2$, alors cette égalité correspond à $iii)$.
 Si $s'=s_n$ alors, d'après \eqref{application de l'unicité du saut dans le cas où p divise a_n}, 
\begin{equation*} 
  t_{n,2}(r,s_n+1,s_n,\overline{S},0)=t_{n,2}(r,s_n+1,s_n+1,\overline{S},0)
  \end{equation*}
   et, d'après le $ii)$ de la proposition \ref{Théorème fondamental} appliqué à $l=1$ et à $\sigma=id$,  
\begin{equation*}   
   t_{n,2}(r,s_n+1,s_n+1,\overline{S},0)=t_{n,2}(r,s_n+1,s_2,\overline{S},0).
   \end{equation*}
Des trois dernières égalités, on obtient que $t_{n,n+1}(r,1,\overline{s},s_n)=t_{n,2}(r,s_n+1,s_2,\overline{S},0)$, ce qui est le résultat souhaité. 
 Ainsi, $(n; s_1,\dots,s_n)\in\Lambda$, ce qui est absurde.
 Ceci prouve la proposition.
\end{proof}

Soient $n\geq 2$, $s_1,\dots,s_n\geq 1$ et $E\subset \{1,\dots,n-1\}$ un ensemble non vide tel que pour tout $i\in E$, on ait $s_i=s_n+1$.
 En utilisant le $iii)$ et le $iv)$ de la proposition \ref{Théorème fondamental 2.}, on a que $t_{n,n+1}(r,s_1,\dots,s_n+1)=t_{n,i+1}(r,s_1+\delta_{1,i},\dots,s_n+\delta_{n,i})$ et a la propriété $(P)$ pour tout $i\in E$.
 Ceci montre donc que les hypothèses nécessaires à l'application du Lemme \ref{unicité du saut ?} sont en réalité vérifiées. 
La proposition \ref{Théorème fondamental 2.} permet donc de simplifier l'énoncé du Lemme \ref{unicité du saut ?} qui devient le corollaire suivant.

\begin{cor} \label{unicité du saut}
Soient $n\geq 2$, $s_1,\dots,s_n\geq 0$ et $E\subset\{1,\dots,n-1\}$ un ensemble non vide tels que $s_i=s_n+1$ pour tout $i\in E$. Alors, $L_{r,s'_1,\dots,s_n'}/L_{r,s_1,\dots,s_n}$ possède un unique saut égal à $t_{n,n+1}(r,s_1,\dots,s_{n-1},s_n+1)$ où 
\begin{equation*}
s_i'=\begin{cases} 
s_i+1 \; \text{si} \; i\in E\cup\{n\} \\
s_i \; \text{sinon}
\end{cases} .
\end{equation*}
De plus, $t_{n,i+1}(r,s'_1,\dots,s_n')=t_{n,n+1}(r,s_1,\dots,s_{n-1},s_n+1)$ pour tout $i\in E\cup \{n\}$. 
\end{cor}

Le corollaire suivant est une conséquence directe du $ii)$ de la proposition \ref{Théorème fondamental} et du $ib)$ de la proposition \ref{Théorème fondamental 2.}.

\begin{cor} \label{Corollaire où je transforme les nombres plus grand que s n en sn n}
Soient $n\geq 2$ et $s_1,\dots,s_n$ des entiers non nuls. 
Fixons un entier $m\in\{2,\dots,n+1\}$.
Pour tout $i$, fixons $s_i'$ tel que $\min\{s_i,s_{m-1}\}\leq s_i'\leq s_i$.
Alors,
\begin{equation*}
t_{n,m}(r,s_1,\dots,s_n)=t_{n,m}(r,s'_1,\dots,s_n').
\end{equation*}
 
\end{cor}

\begin{proof}
Si $s_1\leq s_{m-1}$, alors 
\begin{equation*}
t_{n,m}(r,s_1,\dots,s_n)=t_{n,m}(r,s'_1,s_2,\dots,s_n)
\end{equation*}
 puisque le seul choix possible de $s'_1$ est $s_1$.
Supposons donc $s_{m-1}< s_1$.
En appliquant le $ii)$ de la proposition \ref{Théorème fondamental} à $l=1, k=m$ et à $\sigma=(1,m)$ ou le $i b)$ de la proposition \ref{Théorème fondamental 2.} appliqué à $k=1$, à $s_{m-1}\rightarrow s_1$ et à $s_1\rightarrow s_{m-1}$ suivant la valeur de $v_F(a_n)$, on en déduit que 
\begin{equation*}
t_{n,m}(r,s_1,\dots,s_n)=t_{n,m}(r,s_{m-1},s_2,\dots,s_n).
\end{equation*}
De même, en appliquant le $ii)$ de la proposition \ref{Théorème fondamental} à $l=1, k=m$ et à $\sigma=(1,m)$ ou le $i b)$ de la proposition \ref{Théorème fondamental 2.} appliqué à $k=1$, à $s_{m-1}\rightarrow s_1$ et à $s_1\rightarrow s_{m-1}$ suivant la valeur de $v_F(a_n)$, on en déduit que 
\begin{equation*}
t_{n,m}(r,s_1',\dots,s_n)=t_{n,m}(r,s_{m-1},s_2,\dots,s_n).
\end{equation*}
Il s'ensuit que $t_{n,m}(r,s_1,\dots,s_m)=t_{n,m}(s_1',s_2,\dots,s_n)$.

En répétant le même argument où l'on remplace la coordonnée $s_1$ par $s_2$ (et ainsi de suite), on obtient que (et en remarquant que $s_m=s_m'$)
\begin{equation*}
t_{n,m}(r,s_1,\dots,s_n)=t_{n,m}(r,s'_1,\dots,s_n'),
\end{equation*}
ce qui montre le corollaire.
\end{proof}

\section{Calcul des $t_{n,k}(r,s_1,\dots,s_n)$.} \label{calculs chiants}

Soient $n,r,s_1,\dots,s_n\geq 1$ et $F,a_1,\dots,a_n$ vérifiant l'hypothèse \ref{hypothèse 3}.
Dans la suite, supposons que $s_1\leq \dots \leq s_{n-1}$ si $p\nmid v_F(a_n)$ (resp. $s_1\leq \dots \leq s_n$ si $p\mid v_F(a_n)$), ce que l'on peut toujours supposer si on permute convenablement les $a_i^{1/p^{s_i}}$. 

Les résultats de la section précédente permettent de calculer tous les sauts $t_{n,k}(r,s_1,\dots,s_n)$ avec $k\in\{1,\dots,n+1\} $ (cf Remarque \ref{Remarque après le calcul des sauts dont j'ai besoin}).
Cependant, on se limitera au calcul des $t_{n,k}(r,s_1,\dots,s_n)$ qui nous seront utiles pour décrire la suite des groupes de ramification de $\Gal(L_{r,s_1,\dots,s_n}/F)$ (cf le théorème \ref{calcul des sauts, cas où la valuation est divisible par p} et le théorème \ref{calcul des sauts, cas où la valuation n'est pas divisible par p}).

La notation ci-dessous ne sera valable que pour cette section et la section \ref{calcul des sauts}.
Il n'y aura donc pas de conflit de notation avec \eqref{définition de r k} et \eqref{définition de s k}.
 Soit $k\in\{2,\dots,n+1\}$. 
Pour un entier $l$, notons 
\begin{equation} \label{definition de l'}
l(k) = \min\{l,s_{k-1}\} 
 \end{equation}
 Pour tout entier $m\geq 1$, notons également $\tau_m(l)=t_{m,2}(r,l,\dots,l,l(m))$. 

\begin{lmm} \label{lemme intermédiaire}
Pour tous $r, s_1,\dots,s_n \geq 1$ et $k\geq 3$, on a 

\begin{enumerate} [i)]
\item  $t_{n,2}(r,s_1,\dots,s_n)=\tau_n(s_1)$;
\item $t_{n,k}(r,s_1,\dots,s_n)=p^{s_1(k)}t_{n-1,k-1}(r,s_2(k),\dots,s_n(k))+(1-p)\sum_{l=1}^{s_1(k)} p^{s_1(k)-l} \tau_n(l)$.
\end{enumerate}
\end{lmm}

\begin{proof}

$i)$ 
Cela découle du corollaire \ref{Corollaire où je transforme les nombres plus grand que s n en sn n} appliqué à $k=2$.
 
$ii)$
 D'après le corollaire \ref{Corollaire où je transforme les nombres plus grand que s n en sn n}, on a $t_{n,k}(r,s_1,\dots,s_n)=t_{n,k}(r,s_1(k),\dots,s_n(k))$.

En utilisant le $iii)$ de la proposition \ref{Théorème fondamental} appliqué à $l=1$ et à $\sigma=id$ ou le $ii a)$ de la proposition \ref{Théorème fondamental 2.} suivant la valeur de $v_F(a_n)$ ainsi que le $i)$ de ce lemme, il s'ensuit que
\begin{equation*}
t_{n,k}(r,s_1(k),\dots,s_n(k))=p t_{n,k}(r,s_1(k)-1,s_2(k),\dots,s_n(k))+(1-p)\tau_n(s_1(k)).
\end{equation*}
En répétant le même argument en remplaçant $s_1(k)$ par $s_1(k)-1$ (et ainsi de suite), on déduit par récurrence que 
\begin{equation*} 
t_{n,k}(r,s_1(k),\dots,s_n(k))=p^{s_1(k)}t_{n-1,k-1}(r,s_2(k),\dots,s_n(k))+(1-p)\sum_{l=1}^{s_1(k)} p^{s_1(k)-l} \tau_n(l).
\end{equation*} 

\end{proof}
 
 Soient $m,l\geq 1$ des entiers.
Le $i)$ du théorème ci-dessous montre que le calcul explicite des $\tau_m(l)$ entraîne le calcul explicite des $t_{n,k}(r,s_1,\dots,s_n)$. 

Le $ii)$ du théorème ci-dessous montre que le calcul explicite des $t_{n,2}(r,l,\dots,l)$ avec $p\mid v_F(a_n)$ et des $t_{n,n+1}(r,l,\dots,l)$ avec $p\nmid v_F(a_n)$ entraînent le calcul explicite des $\tau_n(l)$. 

Le $iii)$ (resp. $iv)$) du théorème ci-dessous calcule explicitement les $t_{n,2}(r,l,\dots,l)$ avec $p\mid v_F(a_n)$ (resp. $t_{n,n+1}(r,l,\dots,l)$ avec $p\nmid v_F(a_n)$). 

Enfin, le $v)$ du théorème ci-dessous appliqué à $j=n$ montre que le calcul explicite des $t_{1,1}(r,s)$ et des $\tau_n(l)$ permettent le calcul explicite des $t_{n,1}(r,s_1,\dots,s_n)$.

\begin{thm} \label{thm ou je calcule TOUT}
Pour tous $r,s_1,\dots,s_n \geq 1$ et $k\geq 2$, on a

\begin{enumerate}[i)]
\item \begin{multline*}
t_{n,k}(r,s_1,\dots,s_n) =p^{\sum_{j=1}^{k-2} s_j(k)} \tau_{n-k+2}(s_{k-1}(k)) \\ 
 +(1-p)\sum_{m=0}^{k-3} p^{\sum_{l=0}^{m} s_l(k)}\sum_{l=1}^{s_{m+1}(k)} p^{s_{m+1}(k)-1-l}\tau_{n-m}(l);
\end{multline*}
\item supposons $p\nmid v_F(a_n)$. Soit $l\geq 1$. Si $s_n\geq l$ alors 
\begin{equation*}
\tau_n(l)= \begin{cases} 
t_{n,n+1}(r,l-1,\dots,l-1) & \text{si} \; l\geq 2 \\
1 \; & \text{si} \; l=1
\end{cases} .
\end{equation*}

 Si $s_n\leq l-1$ alors 
\begin{equation*}
\tau_n(l)=p^{s_n} t_{n,2}(r,l,\dots,l,0)+(1-p)\sum_{j=1}^{s_n} p^{s_n-j}t_{n,n+1}(r,j,\dots,j);
 \end{equation*}
 \item si $p\mid v_F(a_n)$ alors pour tout $s\geq 1$ et $r\geq s$, 
\begin{equation*} 
 t_{n,2}(r,s,\dots,s)= 1+\frac{2(p-1)p^n}{p^{n+1}-1}(p^{(n+1)(s-1)}-1);
 \end{equation*}
 \item si $p\nmid v_F(a_n)$ alors pour tous $s\geq 1$ et $r\geq s+1$,
\begin{equation*} 
 t_{n,n+1}(r,s,\dots,s)=1+\frac{2p^{n-1}(p-1)}{p^{n+1}-1}(p^{(n+1)s}-1);
 \end{equation*}
\item supposons $r\geq \max\{s_{n-1},s_n\}+2$ (resp. $r\geq s_n+1$) si $p\nmid v_F(a_n)$ (resp. $p\mid v_F(a_n)$). Alors, pour tout $j\in\{1,\dots,n\}$
\begin{multline*}
t_{n,1}(r,s_1,\dots,s_n) =p^{(\sum_{i=1}^{j-1} s_i)} t_{1,1}(r,s_j,\dots,s_n) \\ 
  +(1-p)\sum_{m=0}^{j-2} p^{\sum_{k=0}^m s_k} \sum_{l=1}^{s_{m+1}}p^{s_{m+1}-l}\tau_{n-m}(l).
 \end{multline*}
 \end{enumerate}
\end{thm}

\begin{proof}
$i)$ 
Le cas $k=2$ correspond au $i)$ du lemme \ref{lemme intermédiaire}. 
Supposons donc $k\geq 3$.
D'après le $ii)$ du Lemme \ref{lemme intermédiaire}, on a 
\begin{equation*}
t_{n,k}(r,s_1(k),\dots,s_n(k))=p^{s_1(k)}t_{n-1,k-1}(r,s_2(k),\dots,s_n(k))+(1-p)\sum_{l=1}^{s_1(k)} p^{s_1(k)-l} \tau_n(l).
\end{equation*}
 
En répétant le même argument que la preuve du $ii)$ du Lemme \ref{lemme intermédiaire} en remplaçant $s_1(k)$ par $s_2(k)$ (et ainsi de suite), on déduit par récurrence que 
\begin{multline*}
t_{n,k}(r,s_1(k),\dots,s_n(k)) = p^{\sum_{j=1}^{k-2} s_j(k)} t_{n-k+2,2}(r,s_{k-1}(k),\dots,s_n(k)) + \\
 (1-p)\sum_{m=0}^{k-3} p^{\sum_{l=0}^{m} s_l(k)}\sum_{l=1}^{s_{m+1}(k)} p^{s_{m+1}(k)-1-l}\tau_{n-m}(l).
\end{multline*}

Enfin, $t_{n-k+2,2}(r,s_{k-1}(k),\dots,s_n(k))= \tau_{n-k+2}(s_{k-1}(k))$ d'après le $i)$ du Lemme \ref{lemme intermédiaire}.
Ceci prouve $i)$ puisque $t_{n,k}(r,s_1,\dots,s_n)=t_{n,k}(r,s_1(k),\dots,s_n(k))$ d'après le corollaire \ref{Corollaire où je transforme les nombres plus grand que s n en sn n} appliqué à $m=k$. \\

$ii)$ Commençons par le cas où $s_n\geq l$.
Dans ce cas, on a $l(n+1)=l$.
 Si $l=1$, c'est clair d'après \eqref{unification de la rq et du cor}.
Supposons $l\geq 2$. 
Le $i)$ de la proposition \ref{Théorème fondamental 2.} montre que 
\begin{equation*}
\tau_n(l)=t_{n,2}(r,l,\dots,l,l)=t_{n,2}(r,l,\dots,l,l-1).
\end{equation*} 
En utilisant le $iii)$ de la proposition \ref{Théorème fondamental 2.} appliqué à $k=1$, on obtient que 
\begin{equation*}
t_{n,2}(r,l,\dots,l,l-2)=t_{n,n+1}(r,l-1,l,\dots,l,l-1).
\end{equation*} 
Ainsi, le corollaire \ref{unicité du saut} appliqué à $s=l-1$ et à $E=\{1\}$ montre que 
\begin{equation*}
t_{n,n+1}(r,l-1,l,\dots,l,l-1)=t_{n,2}(r,l,\dots,l,l-1).
\end{equation*}
De plus, le corollaire \ref{Corollaire où je transforme les nombres plus grand que s n en sn n} appliqué à $m=n+1$ donne
\begin{equation*}
t_{n,n+1}(r,l-1,l,\dots,l,l-1)=t_{n,n+1}(r,l-1,\dots,l-1).
\end{equation*}
On vient ainsi de prouver que $\tau_n(l)=t_{n,n+1}(r,l-1,\dots,l-1)$. 

Supposons $s_n\leq l-1$. 
Dans ce cas, $l(n+1)=s_n$.
En appliquant le $ii b)$ de la proposition \ref{Théorème fondamental 2.} à $k=2$,
\begin{equation*}
\tau_n(l)=p t_{n,2}(r,l,\dots,l,s_n-1)+(1-p)t_{n,n+1}(r,l,\dots,l,s_n).
\end{equation*}
Le corollaire \ref{Corollaire où je transforme les nombres plus grand que s n en sn n} appliqué à $m=n+1$, à $s_i=l$ pour $i\leq n-1$ et à $s'_i=s_n$ montre que $t_{n,n+1}(r,l,\dots,l,s_n)=t_{n,n+1}(r,s_n,\dots,s_n)$.
 Ainsi,
\begin{equation*}
\tau_n(l)=p t_{n,2}(r,l,\dots,l,s_n-1)+(1-p)t_{n,n+1}(r,s_n,\dots,s_n).
\end{equation*} 
 En répétant le même argument en remplaçant $s_n$ par $s_n-1$ (et ainsi de suite), on déduit par récurrence que 
\begin{equation*}
\tau_n(l)=p^{s_n} t_{n,2}(r,l,\dots,l,0)+(1-p)\sum_{j=1}^{s_n} p^{s_n-j}t_{n,n+1}(r,j,\dots,j).
 \end{equation*} \\

$iii)$ 
Procédons par récurrence sur $n$. 
Pour $n=1$, cela correspond au théorème \ref{cas où n=1 et v p a vaut 0}. 
 Supposons que ce soit vrai jusqu'au rang $n-1$ et montrons le au rang $n$. 
 Pour cela, procédons par récurrence sur $s$. 
 Le cas $s=1$ correspond à \eqref{unification de la rq et du cor}.
 Supposons que la formule soit vraie jusqu'au rang $s-1$ et montrons là au rang $s$. 
Le $i)$ de la proposition \ref{Théorème fondamental} appliqué à $l=1$, à $k=n$, à $\sigma=(1,n)$ et à $s_i=s$ pour tout $i$ montre que $\tau_n(s)=t_{n,n+1}(r,s,\dots,s)$.
 Le $ii)$ du Lemme \ref{lemme intermédiaire} appliqué à $k=n+1$ et à $s_i=s$ pour tout $i$ donne
\begin{equation} \label{formule de récurrence de u_n(s)}
\tau_n(s)=p^s \tau_{n-1}(s)+(1-p)\sum_{l=1}^s p^{s-l}\tau_n(l).
\end{equation}
En remarquant que $\tau_n(s)$ se trouve à gauche et à droite de cette expression, on a 
\begin{equation*}
\tau_n(s)=p^{s-1} \tau_{n-1}(s)+(1-p)\sum_{l=1}^{s-1} p^{s-1-l} \tau_n(l). 
\end{equation*} 
En utilisant l'hypothèse de récurrence sur $s$, il s'ensuit que
\begin{equation*}
\tau_n(s)  =p^{s-1} \tau_{n-1}(s)+(1-p)\sum_{l=1}^{s-1} p^{s-1-l}\left(1+\frac{2(p-1)p^{n}}{p^{n+1}-1}(p^{(n+1)(l-1)}-1)\right). 
\end{equation*}
Comme $(1-p)\sum_{l=1}^{s-1} p^{s-1-l}=1-p^{s-1}$, on obtient alors que
\begin{align*}
\tau_n(s) & = p^{s-1} \tau_{n-1}(s)+1-p^{s-1}+(1-p)\sum_{l=1}^{s-1} \frac{2(p-1)p^{n+s-1-l}}{p^{n+1}-1}(p^{(n+1)(l-1)}-1) \\
 & = p^{s-1} \tau_{n-1}(s)+1-p^{s-1}+(p^{s-1}-1)\frac{2(p-1)p^n}{p^{n+1}-1}-\frac{2(p-1)^2p^{s-2}}{p^{n+1}-1}\sum_{l=1}^{s-1} p^{nl}.
\end{align*}
Or, $\sum_{l=1}^{s-1} p^{nl}=p^n \frac{p^{n(s-1)}-1}{p^n-1}$ et par hypothèse de récurrence sur $n$, 
\begin{equation*}
p^{s-1} \tau_{n-1}(s)-p^{s-1}=\frac{2(p-1)p^{n+s-2}}{p^n-1}(p^{n(s-1)}-1).
\end{equation*}
On obtient ainsi que
\begin{align*}
\tau_n(s) & = 1+ 2p^n(p-1)\left(\frac{p^{s-2}(p^{n(s-1)}-1)}{p^n-1}+\frac{p^{s-1}-1}{p^{n+1}-1}+\frac{(1-p)p^{s-2}(p^{n(s-1)}-1)}{(p^{n+1}-1)(p^n-1)}\right) \\
& = 1+2p^n(p-1)\left(\frac{p^{s-2}(p^{n(s-1)}-1)}{p^n-1}\left(1+\frac{1-p}{p^{n+1}-1}\right)+\frac{p^{s-1}-1}{p^{n+1}-1}\right) \\
& = 1+\frac{2p^n(p-1)}{p^{n+1}-1}\left(p^{s-1}(p^{n(s-1)}-1)+p^{s-1}-1\right)
\end{align*}
et $iii)$ s'en déduit. \\

$iv)$ Notons $u_n(s)=t_{n,n+1}(r,s,\dots,s)$ avec la convention que $u_n(0)=1$ (ce qui est en accord avec la formule souhaitée). 

 Procédons par récurrence sur $n$. 
 Pour $n=1$, cela correspond au théorème \ref{Cas n=1 et v p a vaut 1}. 
 Supposons que la formule soit vraie jusqu'au rang $n-1$ et montrons la au rang $n$. 
 Pour cela, procédons par récurrence sur $s$. 
 
Le $ii)$ du Lemme \ref{lemme intermédiaire} appliqué à $k=n+1$ et à $s_i=s$ pour tout $i$ donne
\begin{equation*}
u_n(s)=p^s u_{n-1}(s)+(1-p)\sum_{l=1}^s p^{s-l}t_{n,2}(r,l,\dots,l).
\end{equation*}
Par le $ii)$, $t_{n,2}(r,l,\dots,l)=u_n(l-1)$ pour tout $l\geq 2$.
De plus, le $ia)$ de la proposition \ref{Théorème fondamental 2.} appliqué à $k=1$ montre que $\tau_n(1)=t_{n,2}(r,1,\dots,1,0)=1$ d'après \eqref{unification de la rq et du cor}.
Ainsi, l'égalité $t_{n,2}(r,l,\dots,l)=u_n(l-1)$ reste vraie si $l=1$ d'après notre convention.
Par conséquent, 
\begin{equation*}
u_n(s)=p^s u_{n-1}(s)+(1-p)\sum_{l=1}^s p^{s-l}u_n(l-1).
\end{equation*}
 Le cas $s=1$ donne 
\begin{equation} \label{relation de récurrence de u n 1} 
 u_n(1)=p u_{n-1}(1)+1-p.
 \end{equation} 
 Comme $r\geq 2$, on a $u_1(1)=t_{1,2}(r,1)=2p-1$ d'après le $iii)$ du théorème \ref{Cas n=1 et v p a vaut 1}.
 Une rapide récurrence prouve que $u_n(1)=2p^n-2p^{n-1}+1$, ce qui montre le cas $s=1$.
 
 On suppose désormais que $iv)$ soit vrai jusqu'au rang $s-1$ et montrons qu'il l'est encore au rang $s$.
 En utilisant l'hypothèse de récurrence sur $s$, 
\begin{equation*}
u_n(s) =p^s u_{n-1}(s)+(1-p)\sum_{l=1}^s p^{s-l}\left(1+\frac{2(p-1)p^{n-1}}{p^{n+1}-1}(p^{(n+1)(l-1)}-1)\right) 
\end{equation*}

Comme $(1-p)\sum_{l=1}^s p^{s-l}=1-p^s$, on en déduit que : 
\begin{align*}
u_n(s) & = p^s u_{n-1}(s)+1-p^s+(1-p)\sum_{l=1}^s \frac{2(p-1)p^{n+s-1-l}}{p^{n+1}-1}(p^{(n+1)(l-1)}-1) \\
& = p^s u_{n-1}(s)+1-p^s+(p^s-1)\frac{2(p-1)p^{n-1}}{p^{n+1}-1}-\frac{2(p-1)^2p^{s-2}}{p^{n+1}-1}\sum_{l=1}^s p^{nl}
\end{align*}

Or, $\sum_{l=1}^s p^{nl}=p^n \frac{p^{ns}-1}{p^n-1}$ et par hypothèse de récurrence sur $n$, on a que 
\begin{equation*}
p^s u_{n-1}(s)-p^s=\frac{2(p-1)p^{n+s-2}}{p^n-1}(p^{ns}-1).
\end{equation*}

On obtient ainsi que
\begin{align*}
u_n(s) & =1+2(p-1)p^{n-1}\left(\frac{p^{s-1}(p^{ns}-1)}{p^n-1}+\frac{p^s-1}{p^{n+1}-1}+\frac{p^{s-1}(1-p)(p^{ns}-1)}{(p^{n+1}-1)(p^n-1)}\right) \\
& = 1+2(p-1)p^{n-1}\left( \frac{p^{s-1}(p^{ns}-1)}{p^n-1}\left(1+\frac{1-p}{p^{n+1}-1}\right)+\frac{p^s-1}{p^{n+1}-1}\right) \\
& = 1+2(p-1)p^{n-1}\frac{p^s(p^{ns}-1)+p^s-1}{p^{n+1}-1} \\
& = 1+\frac{2(p-1)p^{n-1}}{p^{n+1}-1}(p^{(n+1)s}-1),
\end{align*}

ce qui conclut la récurrence et montre $iv)$. \\

   $v)$    
   Soit $l$ un entier.
   Commençons par le cas où $p\mid v_F(a_n)$. 
   Dans ce cas, $s_i(n)=s_i$ pour tout $i$ (on rappelle que par hypothèse, $s_1\leq \dots \leq s_n$).
   D'après le $iii)$, on a que $t_{n,2}(r,l,\dots,l)$ ne dépend pas de $r$.
    Si $r\geq s_{k-1}$, on déduit alors du $i)$ que $t_{n,k}(r,s_1,\dots,s_n)$ ne dépend pas de $r$.
   Comme $s_1\leq \dots \leq s_n$, alors $t_{n,k}(r,s_1,\dots,s_n)$ ne dépend pas de $r$ pour tout $k\in\{2,\dots,n+1\}$ si $r\geq s_n$.
Ainsi, pour tout $r\geq s_n+1$ et tout $k\in\{2,\dots,n+1\}$,
\begin{equation} \label{dsdfd}
t_{n,k}(r,s_1,\dots,s_n)=t_{n,k}(r-1,s_1,\dots,s_n).
\end{equation}
   
  Supposons maintenant $p\nmid v_F(a_n)$.
  Alors, $t_{n,n+1}(r,l-1,\dots,l-1)$ ne dépend pas de $r$ d'après $iv)$.
   Du $ii)$, on en déduit que $\tau_n(l)$ ne dépend pas de $r$.
Ainsi, si $r\geq \max\{s_{k-1}(k),s_{k-2}(k)\}+1$, il s'ensuit du $i)$ que $t_{n,k}(r,s_1,\dots,s_n)$ ne dépend pas de $r$. 
Comme $s_1\leq \dots \leq s_{n-1}$, on en conclut que $t_{n,k}(r,s_1,\dots,s_n)$ ne dépend pas de $r$ pour tout $k\geq 2$ si $r\geq \max\{s_{n-1},s_n\}+1$.
Donc, pour tout $r\geq \max\{s_{n-1},s_n\}+2$, 
\begin{equation} \label{qdggjhng}
t_{n,k}(r,s_1,\dots,s_n)=t_{n,k}(r-1,s_1,\dots,s_n).
\end{equation}
 
Supposons maintenant $r$ comme dans l'énoncé du $v)$ afin que \eqref{dsdfd} ou \eqref{qdggjhng} soit vérifié.
 D'après le $i)$ de la proposition \ref{calcul de la différente version appliquée.}, 
  \begin{equation*} 
\begin{aligned}
t_{n,1}(r,s_1,\dots,s_n) & = p (t_{n,1}(r,s_1-1,s_2,\dots,s_n)-t_{n,2}(r-1,s_1,\dots,s_n)) \\ 
& \qquad + t_{n,2}(r,s_1,\dots,s_n)\\
&= p t_{n,1}(r,s_1-1,s_2,\dots,s_n)+(1-p)\tau_n(s_1)
\end{aligned}
\end{equation*}
(on rappelle que $\tau_n(s_1)=t_{n,2}(r,s_1,\dots,s_n)$ d'après le $i)$ du Lemme \ref{lemme intermédiaire}).
En répétant le même argument en remplaçant $s_1$ par $s_1-1$ (et ainsi de suite), on déduit par récurrence que 
\begin{equation} \label{Z}
t_{n,1}(r,s_1,\dots,s_n)=p^{s_1}t_{n-1,1}(r,s_2,\dots,s_n)+(1-p)\sum_{l=1}^{s_1} p^{s_1-l} \tau_n(l).
\end{equation}
En répétant le même argument où l'on remplace $s_1$ par $s_2$ (et ainsi de suite), on déduit par récurrence que 

\begin{multline*}
t_{n,1}(r,s_1,\dots,s_n) =p^{(\sum_{i=1}^{j-1} s_i)} t_{1,1}(r,s_j,\dots,s_n) \\ 
  +(1-p)\sum_{m=0}^{j-2} p^{\sum_{k=0}^m s_k} \sum_{l=1}^{s_{m+1}}p^{s_{m+1}-l}\tau_{n-m}(l).
 \end{multline*}
 pour tout $j\geq 1$. 
 Ceci prouve le $v)$ et donc le théorème.
\end{proof}

\begin{rqu} \label{Remarque après le calcul des sauts dont j'ai besoin}
 Si $p\mid v_F(a_n)$, on sait calculer $t_{n,2}(r,s,\dots,s)$ pour $r\geq s$ d'après le théorème \ref{thm ou je calcule TOUT}. 
 En fait, il est possible de le calculer pour tout $r$. 
 En effet, le Lemme \ref{lemme intermédiaire} est valable pour tout $r$.
 Ainsi, il suffit de connaître $t_{1,2}(r,s)$ et $t_{n,2}(r,1,\dots,1)$ pour tout $r$ afin de pouvoir utiliser la double récurrence sur $n$ et $s$ présente dans la preuve du $iii)$ du théorème ci-dessus.
Le calcul de ces nombres se déduit du théorème \ref{cas où n=1 et v p a vaut 0} et de la Remarque \ref{Remarque sur le calcul de r,1,1}.
Il est donc possible d'obtenir une formule explicite de $t_{n,2}(r,s,\dots,s)$ pour $r<s$.

 Les mêmes arguments s'appliquent aussi pour le calcul de $t_{n,n+1}(r,s,\dots,s)$ dans le cas où $p\nmid v_F(a_n)$. 
 En effet, le Lemme \ref{lemme intermédiaire} est valable pour tout $r$, la valeur de $t_{1,2}(r,s)$ se déduit du théorème \ref{Cas n=1 et v p a vaut 1} et  celle de $t_{n,n+1}(r,1,\dots,1)$ (pour $r\geq 2$) se déduit du $iv)$ du théorème \ref{thm ou je calcule TOUT}.
 Pour $r=1$, on sait que $t_{1,2}(1,1)=p$ (cf le Lemme \ref{calcul de $t(1,1)$.}). 
 De \eqref{relation de récurrence de u n 1}, on en déduit $t_{n,n+1}(1,1,\dots,1)$.
  On peut ainsi appliquer la double récurrence sur $n$ et $s$ présente dans la preuve du $iv)$ afin d'en déduire $t_{n,n+1}(r,s,\dots,s)$ pour tout $r,s\geq 1$.
  
Le $ii)$ du théorème \ref{thm ou je calcule TOUT} montre alors que l'on peut calculer $\tau_n(s)$ pour tous $r,s\geq 1$. 
Enfin, le $i)$ du théorème \ref{thm ou je calcule TOUT} permet d'en déduire la valeur de $t_{n,k}(r,s_1,\dots,s_n)$ pour tous $r,s_1,\dots,s_n\geq 1$.
\end{rqu} 
 
 Pour pouvoir déterminer la suite des groupes de ramification, on aura également besoin de la formule \eqref{NONNO} ci-dessous.
 
\begin{rqu} \label{si je choisis de d'abord faire baisser le s n}
Supposons $p\nmid v_F(a_n)$. 
Plaçons nous dans les conditions du $v)$ du théorème \ref{thm ou je calcule TOUT} (on a ainsi \eqref{qdggjhng}).
Supposons en plus que $s_n\leq s_1$.
D'après le $i)$ de la proposition \ref{calcul de la différente version appliquée.}, 
  \begin{equation*} 
\begin{aligned}
t_{n,1}(r,s_1,\dots,s_n) & = p (t_{n,1}(r,s_1,\dots,s_n-1)-t_{n,n+1}(r-1,s_1,\dots,s_n)) \\ 
& \qquad + t_{n,n+1}(r,s_1,\dots,s_n)\\
&= p t_{n,1}(r,s_1,\dots,s_n-1)+(1-p)t_{n,n+1}(r,s_1,\dots,s_n)
\end{aligned}
\end{equation*}
d'après \eqref{qdggjhng}.
De plus, d'après le corollaire \ref{Corollaire où je transforme les nombres plus grand que s n en sn n} appliqué à $s_i'=s_n$ (par hypothèse, $s_n\leq s_1$ et  $s_1\leq s_i$ pour tout $i\leq n-1$), on en déduit que $t_{n,n+1}(r,s_1,\dots,s_n)=t_{n,n+1}(r,s_n,\dots,s_n)$. 
Ainsi, 
\begin{equation*}
t_{n,1}(r,s_1,\dots,s_n) = p t_{n,1}(r,s_1,\dots,s_n-1)+(1-p)t_{n,n+1}(r,s_n,\dots,s_n).
\end{equation*}
En répétant le même argument en remplaçant $s_n$ par $s_n-1$ (et ainsi de suite), on déduit par récurrence que 
\begin{equation} \label{NONNO}
t_{n,1}(r,s_1,\dots,s_n)=p^{s_n}t_{n,1}(r,s_1,\dots,s_{n-1},0)+(1-p)\sum_{l=1}^{s_n} p^{s_n-l} t_{n,n+1}(r,l,\dots,l).
\end{equation}
\end{rqu}

\section{Calcul des groupes de ramification de $\Gal(L_{r,s_1,\dots,s_n}/F)$.} \label{calcul des sauts}

On se propose de calculer explicitement la suite des groupes de ramification de $G:=\Gal(L_{r,s_1,\dots,s_n}/F)$ avec $r\geq \max\{s_1,\dots,s_n\}$ et $F,a_1,\dots,a_n$ vérifiant les conditions de l'hypothèse \ref{hypothèse 3}.
On rappelle que $G_0=G$ et $G_1=\Gal(L_{r,s_1,\dots,s_n}/F(\zeta_p))$ d'après  la proposition \ref{calculs de 0 et 1}.

Commençons par le cas où $p\mid v_F(a_n)$.
On rappelle que dans ce cas, $s_1\leq \dots \leq s_n$.
Ainsi, pour tous entiers $l$ et $k$ tels que $l\leq k$, on a $s_l(k)=s_l$ (pour la définition de $s_l(k)$, voir \eqref{definition de l'}).

\begin{lmm} \label{dernière minoration qu'il me manquait}
Pour tous entiers $s_1,\dots,s_n\geq 1$, on a  

\begin{enumerate} [i)]
\item $t_{n,1}(s_n+1,s_1,\dots,s_n) < t_{n,n+1}(s_n+1,s_1,\dots,s_{n-1},s_n+1)$;
\item $t_{n,n+1}(s_n,s_1,\dots,s_n) < t_{n,1}(s_n+1,s_1,\dots,s_n)$.
\end{enumerate}

\end{lmm}

\begin{proof}

D'après le $i)$ et le $iv)$ du théorème \ref{cas où n=1 et v p a vaut 0}, on obtient
\begin{align*}
t_{1,1}(s_n+1,s_n)- t_{1,2}(s_n+1,s_n+1) & = \frac{p-1}{p+1}(p^{2s_n}+p^{2s_n+1})-p^{2s_n+1} \\
& = -p^{2s_n}<0
 \end{align*}
et 
 \begin{align*}
t_{1,2}(s_n,s_n)-t_{1,1}(s_n+1,s_n) & = p^{2s_n-1}-\frac{p-1}{p+1}(p^{2s_n-1}+p^{2s_n}) \\
& = p^{2s_n-1}(2-p)<0.
\end{align*}

 En utilisant le $i)$ avec $k=n+1$ et le $v)$ avec $j=n$ du théorème \ref{thm ou je calcule TOUT}, on a
\begin{multline*}
t_{n,1}(s_n+1,s_1,\dots,s_n)-t_{n,n+1}(s_n+1,s_1,\dots,s_{n-1},s_n+1)  \\
= p^{\sum_{i=1}^{n-1} s_i} \big (t_{1,1}(s_n+1,s_n) - t_{1,2}(s_n+1,s_n+1)\big)<0, 
\end{multline*}
ce qui montre $i)$, et
\begin{multline*}
t_{n,n+1}(s_n,s_1,\dots,s_n)-t_{n,1}(s_n+1,s_1,\dots,s_n) \\
 = p^{\sum_{i=1}^{n-1} s_i} \big(t_{1,2}(s_n,s_n)-t_{1,1}(s_n+1,s_n)\big)<0,
\end{multline*} 
 ce qui montre $ii)$. 
 Le lemme s'ensuit.
\end{proof}

On peut maintenant prouver notre premier théorème, que l'on rappelle pour la commodité du lecteur. 
\begin{thm*} 
Soient $F$ et $a_1,\dots,a_n$ vérifiant l'hypothèse \ref{hypothèse 3}. 
 Supposons de plus que $p\mid v_F(a_n)$ et que $s_1\leq \dots \leq s_n$.
Alors, les sauts de $G$, différents de $0$, sont les 

\begin{enumerate} [i)]
\item $t_{n,1}(s,s_1,\dots,s_n)$ avec $s\in \{s_n+1,\dots,r\}$;
\item $t_{n,1}(s_n'+1,s_1',\dots,s_n')$ où $s_n'\in\{1,\dots,s_n-1\}$ et où $s_i'=\min\{s_i,s_n'\}$;
\item $t_{n,n+1}(s_n',s_1',\dots,s_n')$ où $s_n'\in\{1,\dots,s_n\}$ et où $s_i'=\min\{s_i,s_n'\}$.
\end{enumerate}

De plus, si on note $u_j$ le $(n+1)$-uplet tel que $t_j=t_{n,k}(u_j)$ désigne le $(j+1)$-ième plus petit saut de $G$ alors, 

\begin{enumerate} [i)]
\item $G_0=G$;
\item $G_{t_j}=\Gal(L_{r,s_1,\dots,s_n}/L_{u_{j-1}})$ pour $j\geq 1$.
\end{enumerate}
\end{thm*}

\begin{proof} 

Notons $K_i=L_{r^{(i)},s_1^{(i)},\dots,s_n^{(i)}}$ la suite définie par 

\begin{enumerate} [i)]
\item $K_1=L_{1,0,\dots,0}$;
\item si $s_n^{(i)}= s_n$, alors $K_{i+1}=L_{\min\{r,r^{(i)}+1\},s_1,\dots,s_n}$;
\item si $r^{(i)}=s_n^{(i)}\neq s_n$ alors $ K_{i+1}=L_{r^{(i)}+1,s_1^{(i)},\dots,s_n^{(i)}}$;
\item si $r^{(i)}>s_n^{(i)}\neq s_n$, alors $K_{i+1}=L_{r^{(i)},\min\{s_1,s_1^{(i)}+1\},\dots,\min\{s_n,s_n^{(i)}+1\}}$.
\end{enumerate}

Il est clair qu'il existe $m$ tel que $K_m=L_{r,s_1,\dots,s_n}$ et que $K_i=K_m$ pour tout $i\geq m$.
De plus, pour tout $i$, $K_i/K_1$ est galoisienne et $K_{i+1}/K_i$ admet un unique saut d'après le corollaire \ref{unicité du saut}. 
On le notera $t_i$. 
On souhaite montrer que $t_1< \dots < t_{m-1}$.

Supposons $s_n^{(i)}=s_n$ (et donc $s_k^{(i)}=s_k$ pour tout $k$ car $s_1\leq \dots \leq s_n$) avec $r^{(i)}<r-1$.
Alors,
\begin{equation*}
K_{i+1}=L_{r^{(i)}+1,s_1,\dots,s_n} \; \text{et} \; K_{i+2}=L_{r^{(i)}+2,s_1,\dots,s_n}.
 \end{equation*}
Il en résulte donc que 
\begin{equation*}
t_i=t_{n,1}(r^{(i)}+1,s_1,\dots,s_n) \; \text{et} \; t_{i+1}=t_{n,1}(r^{(i)}+2,s_1,\dots,s_n).
\end{equation*}
En appliquant la proposition  \ref{croissance des suites} à $K=K_i$ et à $\alpha=\zeta_{p^{r^{(i)}}}$, on obtient que $t_i< t_{i+1}$.

Notons $j$ le plus petit entier tel que $s_n^{(j)}=s_n$. 
Le lecteur pourra remarquer que si $i\leq \lceil j/2\rceil$ (la partie entière supérieure de $j/2$) alors, 
\begin{equation*}
r^{(2i)}=r^{(2i-1)}=r^{(2i-2)}+1=s^{(2i)}=s^{(2i-1)}+1.
\end{equation*}
Donc, $r^{(2i)}=s^{(2i)}$ et $r^{(2i-1)}=s^{(2i-1)}+1$ si $s_n^{(i)}\neq s_n$. 
D'où $r^{(i)}\in\{s_n^{(i)}, s_n^{(i)}+1\}$ pour tout $i$ tel que $s_n^{(i)}\neq s_n$.

Supposons $r^{(i)}=s_n^{(i)}\neq s_n$.
Dans ce cas, 
\begin{equation*}
K_{i+1}=L_{s_n^{(i)}+1,s_1^{(i)},\dots,s_n^{(i)}} \; \text{et} \; K_{i+2}=L_{s_n^{(i)}+1,s_1^{(i+1)},\dots,s_{n-1}^{(i+1)},s_n^{(i)}+1}
 \end{equation*}
avec $s_j^{(i+1)}=\min\{ s_j, s_j^{(i)}+1\}$.
On obtient alors 
\begin{equation*}
t_i=t_{n,1}(s_n^{(i)}+1,s_1^{(i)},\dots,s_n^{(i)}) \; \text{et} \; t_{i+1}= t_{n,n+1}(s_n^{(i)}+1,s_1^{(i+1)},\dots,s_{n-1}^{(i+1)}, s_n^{(i)}+1).
\end{equation*}
Notons $E=\{ j \; \vert \; s_j\neq s_j^{(i)}\}$. 
Par construction de $E$ et des $K_i$, on en déduit que $s_j=s_j^{(i)}=s_j^{(i+1)}$ si $j\notin E$ et que $s_j^{(i)}=s_n^{(i)}<s_j$ pour tout $j\in E$. 
Ainsi, $s_j^{(i+1)}=s_j^{(i)}+1=s_n^{(i)}+1$ pour tout $j\in E$. 
D'après le corollaire \ref{unicité du saut 3} appliqué à $s=s_n^{(i)}$ et à $i=k=n$, on en déduit que $t_{i+1}=t_{n,n+1}(s_n^{(i)}+1,s_1^{(i)},\dots,s_{n-1}^{(i)}, s_n^{(i)}+1)$.
D'après le $i)$ du Lemme \ref{dernière minoration qu'il me manquait}, on en déduit que $t_i<t_{i+1}$.

Supposons $r^{(i)}>s_n^{(i)}\neq s_n$.
Alors, $r^{(i)}=s_n^{(i)}+1$.
D'où
\begin{equation*}
K_{i+1}=L_{s_n^{(i)}+1,s_1^{(i+1)},\dots,s_{n-1}^{(i+1)},s_n^{(i)}+1}
\end{equation*}
et
\begin{equation*}
K_{i+2}=L_{s_n^{(i)}+2,s_1^{(i+1)},\dots,s_{n-1}^{(i+1)},s_n^{(i)}+1}
 \end{equation*}
avec $s_j^{(i+1)}=\min\{s_j,s_j^{(i)}+1\}$. 
On a alors 
\begin{equation*}
t_i=t_{n,n+1}(s_n^{(i)}+1,s_1^{(i+1)},\dots,s_{n-1}^{(i+1)},s_n^{(i)}+1)
\end{equation*}
et 
\begin{equation*}
t_{i+1}= t_{n,1}(s_n^{(i)}+2,s_1^{(i+1)},\dots,s_{n-1}^{(i+1)}, s_n^{(i)}+1).
\end{equation*}
D'après le $ii)$ du Lemme \ref{dernière minoration qu'il me manquait}, on en déduit que $t_i<t_{i+1}$. 

On a donc montré que $t_1<\dots <t_{m-1}$.
Le corollaire \ref{utile pour déterminer la suite des groupes de rami} montre que les sauts de $G$ qui sont non nuls sont les $t_1,\dots,t_{m-1}$ et que $G_{t_i}=\Gal(K_m/K_i)$.
Enfin, on remarque que les $t_i$ correspondent bien aux sauts énoncés dans le théorème.
  \end{proof}
    
  \begin{ex}
  \rm{On souhaite déterminer la suite des groupes de ramification de $G=\Gal(L_{4,1,2,3}/F)$ avec $p\mid v_F(a_4)$ (avec $F$ et $a_1,a_2,a_3,a_4$ vérifiant les conditions de l'hypothèse \ref{hypothèse 3}). 
  D'après le théorème ci-dessus, les sauts sont $t_0=0$ si $p\neq 2$ et    
\begin{align*}
 & t_1=t_{4,4}(1,1,1,1), \; t_2=t_{4,1}(2,1,1,1), \; t_3=t_{4,4}(2,1,2,2) \\
 & t_4=t_{4,1}(3,1,2,2),\; t_5=t_{4,4}(3,1,2,3), \; t_6=t_{4,1}(4,1,2,3).
 \end{align*}
 
  De plus, la suite des groupes de ramification de $G$ est 
\begin{align*}  
  & G_0=G, G_{t_1}=\Gal(L_{4,1,2,3}/L_{1,0,0,0}), G_{t_2}=\Gal(L_{4,1,2,3}/L_{1,1,1,1}) \\ 
  & G_{t_3}=\Gal(L_{4,1,2,3}/L_{2,1,1,1}), \; G_{t_4}=\Gal(L_{4,1,2,3}/L_{2,1,2,2}) \\
  & G_{t_5}=\Gal(L_{4,1,2,3}/L_{3,1,2,2})\; \text{et} \; G_{t_6}=\Gal(L_{4,1,2,3}/L_{3,1,2,3}).
  \end{align*}}
  
  \end{ex}
Traitons à présent le cas $p\nmid v_F(a_n)$.
On rappelle que dans ce cas, $s_1\leq \dots \leq s_{n-1}$.

\begin{lmm} \label{majorations utiles dans le cas où la valuation n'est pas divisaible pas p}
Soient $s\geq 1$ et $s_1',\dots,s_{n-1}'$ des entiers tels que $ s_i'\leq s$ pour tout $i$.
Alors, 

\begin{enumerate} [i)]
\item $t_{n,1}(s+2,s_1',\dots,s_{n-1}',s)>t_{n,n+1}(s+1,s_1',\dots,s_{n-1}',s)$ ;
\item $t_{n,n+1}(s+1,s_1',\dots,s_{n-1}',s)>t_{n,1}(s+1,s_1',\dots,s_{n-1}',s-1)$.
\end{enumerate}
\end{lmm}

\begin{proof}
$i)$ En utilisant le $v)$, le $vi)$ et le $iii)$ du théorème \ref{Cas n=1 et v p a vaut 1}, 
\begin{align*}
t_{1,1}(s+2,s)-t_{1,2}(s+1,s-1) & = p^{2s+1}-\frac{2p^{2s+1}-p+1}{p+1}-\frac{2p^{2s}+p-1}{p+1} \\
& = p^{2s+1} -2p^{2s}>0. 
\end{align*}
En appliquant le $v)$ à $j=n$ et le $i)$ du théorème \ref{thm ou je calcule TOUT} à $k=n+1$, on obtient
\begin{multline*}
t_{n,1}(s+2,s_1',\dots,s_{n-1}',s)-t_{n,n+1}(s+1,s_1',\dots,s_{n-1}',s-1) \\
= t_{1,1}(s+2,s)-t_{1,2}(s+1,s-1)>0,
\end{multline*}
car $\min\{s'_l, s\}=s'_l$ pour tout $l\leq n-1$ par hypothèse.
Ceci prouve $i)$. \\

$ii)$  En utilisant respectivement le $iii)$ et le $v)$ du théorème \ref{Cas n=1 et v p a vaut 1}, 
\begin{align*}
t_{1,2}(s+1,s)-t_{1,1}(s+1,s-1) & = \frac{2p^{2s}+p-1}{p+1}-p^{2s-1}+\frac{2p^{2s+1}-p+1}{p+1} \\
& = 2p^{2s}-p^{2s-1}>0.
\end{align*}
En appliquant le $v)$ à $j=n$ et le $i)$ du théorème \ref{thm ou je calcule TOUT} à $k=n+1$, on obtient
\begin{multline*}
t_{n,n+1}(s+1,s_1',\dots,s_{n-1}',s)-t_{n,1}(s+1,s_1',\dots,s_{n-1}',s-1) \\
= t_{1,2}(s+1,s)-t_{1,1}(s+1,s-1)>0,
\end{multline*}
car $\min\{s'_l,s\}=s'_l$ pour tout $l\leq n-1$ par hypothèse.
Ceci prouve $ii)$. 
\end{proof}

Nous sommes maintenant en mesure de prouver notre second théorème, que l'on rappelle pour la commodité du lecteur. 
On rappelle également que $s_1\leq \dots \leq s_{n-1}$.

\begin{thm*} 
Soient $F$ et $a_1,\dots,a_n$ vérifiant l'hypothèse \ref{hypothèse 3}.
Supposons de plus que $p\nmid v_F(a_n)$ et que $s_1\leq \dots\leq s_{n-1}$.
Alors, les sauts de $G$, différents de $0$, sont les 

\begin{enumerate} [i)]
\item $t_{n,n+1}(r,s_1,\dots,s_n) $ si $r=s_n$;
\item $t_{n,1}(s_n'+1,s_1',\dots,s_{n-1}', \min\{s_n'-1,s_n\})$ où $s_n'\in\{1,\dots,r\}$ et $s_i'=\min\{s_i,s_n'\}$;
\item $t_{n,n}(s_{n-1}',s_1',\dots,s_{n-1}',s_n)$ où $s_{n-1}'\in\{s_n+2,\dots,s_{n-1}\}$ et $s_i'=\min\{s_i,s_{n-1}'\}$;
\item $t_{n,n+1}(s_n'+1,s_1',\dots,s_n')$ où $s_n'\in\{1,\dots,s_n\}$ et $s_i'=\min\{s_i,s_n'+1\}$.
\end{enumerate}

De plus, si on note $u_j$ le $(n+1)$-uplet tel que $t_j=t_{n,k}(u_j)$ désigne le $(j+1)$-ième plus petit saut de $G$ alors,

\begin{enumerate} [i)]
\item $G_0=G$;
\item $G_{t_j}=\Gal(L_{r,s_1,\dots,s_n}/L_{u_{j-1}})$ pour $j\geq 1$.
\end{enumerate}
\end{thm*}

\begin{proof}
 Pour tout $i$, notons $K_i=L_{r^{(i)},s_1^{(i)},\dots,s_n^{(i)}}$ la suite définit par 

\begin{enumerate} [i)]
\item $K_1=L_{1,0,\dots,0}$ et $K_2=L_{1,1,\dots,1,0}$;
\item Si $s_l^{(i)}=s_l$ pour tout $l$ alors 
\begin{equation*}
K_{i+1}=L_{\min\{r,r^{(i)}+1\},s_1,\dots,s_n};
\end{equation*}
\item Si $s_l^{(i)}\neq s_l$ pour un certain $l$ et si $r^{(i)}=r$, alors 
\begin{equation*}
K_{i+1}=L_{r,\min\{s_1,s_1^{(i)}+1\},\dots,\min\{s_n,s_n^{(i)}+1\}}.
\end{equation*}
Notons
\begin{equation*}
S^{(i)}=\begin{cases} 
s^{(i)}_{n-1} \; \text{si} \; s^{(i)}_n<s^{(i)}_{n-1} \\
s^{(i)}_n+1 \; \text{sinon} 
\end{cases}.
\end{equation*}
\item Si $s_l^{(i)}\neq s_l$ pour un certain $l$, si $r^{(i)}>S^{(i)}$ et si $r^{(i)}\neq r$ alors 
\begin{equation*}
K_{i+1}=L_{r^{(i)}, \min\{s_1,s_1^{(i)}+1\},\dots,\min\{s_n,s_n^{(i)}+1\}};
\end{equation*}
\item Si $s_l^{(i)}\neq s_l$ pour un certain $l$ et que $r^{(i)}=S^{(i)}\neq r$, alors 
\begin{equation*}
K_{i+1}=L_{\min\{r,r^{(i)}+1\},s_1^{(i)},\dots,s_n^{(i)}}.
\end{equation*}
\end{enumerate}
Par construction, $K_i/K_1$ est galoisienne et $K_{i+1}/K_i$ admet un unique saut (cf corollaire \ref{unicité du saut ?}) pour tout $i$.
On le nommera $t_i$. 
Remarquons qu'il existe un entier $m$ tel que $K_m=L_{r,s_1,\dots,s_n}$ et que $K_i=K_m$ pour tout $i\geq m$.
Soient $i\in\{2,\dots, m-2\}$ et $l$ tels que $s_l^{(i)}\neq s_l$.
Si $s_k^{(i)}\neq s_k$ pour un entier $k\neq l$, on pourra remarquer que 
\begin{equation} \label{pour montrer je sais pas quoi}
\begin{cases}
s_k^{(i)}=s_l^{(i)} \; \text{si}\; k,l\neq n \\
s_k^{(i)}=s_n^{(i)}+1 \; \text{si}\; l= n \; \text{et} \; k\neq n
\end{cases}.
\end{equation}
Supposons que $l\leq n-1$.
 Par construction, on a alors $s_l^{(i)}\geq s_{n-1}^{(i)}$.
Or, $s_l^{(i)}<s_l$ et $s_l\leq s_{n-1}$. 
Par conséquent, $s_{n-1}>s_{n-1}^{(i)}$ et donc, $s_{n-1}^{(i)}\neq s_{n-1}$. 
On a ainsi montré que si $s_l^{(i)}\neq s_l$ pour un certain $l\leq n-1$, alors $s_{n-1}^{(i)}\neq s_{n-1}$.

Supposons que $s_n^{(i)}\neq s_n$.
Alors,
\begin{equation} \label{Y}
S^{(i)}=s_n^{(i)}+1.
\end{equation}
En effet, soit $s_{n-1}^{(i)}\neq s_{n-1}$ et dans ce cas, $s_{n-1}^{(i)}=s_n^{(i)}+1$ d'après \eqref{pour montrer je sais pas quoi} et donc $S^{(i)}=s_n^{(i)}+1$, soit $s_{n-1}^{(i)}=s_{n-1}$ et dans ce cas, $s_n\geq s_{n-1}$ et donc $S^{(i)}=s_n^{(i)}+1$ par définition de $S^{(i)}$.

Supposons maintenant que $s_{n-1}^{(i)}\neq s_{n-1}$.
 Alors, 
\begin{equation}\label{X}
S^{(i)}=s_{n-1}^{(i)}.
\end{equation}
En effet, soit $s_n^{(i)}\neq s_n$ et dans ce cas, $s_{n-1}^{(i)}=s_n^{(i)}+1$ d'après \eqref{pour montrer je sais pas quoi} et donc $S^{(i)}=s_{n-1}$, soit $s_n^{(i)}=s_n$ et dans ce cas, $s_{n-1}> s_n$ et donc $S^{(i)}=s_{n-1}^{(i)}$ par définition de $S^{(i)}$. 

Par définition, on a $\max\{s_{n-1}^{(i)}, s_n^{(i)}+1\}\geq S^{(i)}\geq \min\{s_{n-1}^{(i)}, s_n^{(i)}+1\}$. 
De plus,s'il existe $i$ et $l$ tels que $s_l^{(i)}\neq s_l$ alors, $1+\min\{s_{n-1}^{(i)}, s_n^{(i)}+1\}\geq r^{(i)} \geq \max\{s_{n-1}^{(i)}, s_n^{(i)}+1\}$ par construction. 
On en déduit alors que $r^{(i)}\in\{S^{(i)}, S^{(i)}+1\}$ pour tout entier $i$ tel que $s_{n-1}^{(i)}\neq s_{n-1}$ ou $s_n^{(i)}\neq s_n$.

On souhaite montrer que $t_1<\dots <t_{m-1}$.
Pour cela, on procédera en quatre étapes indépendantes résumées ci-dessous (avec $i\leq m-2$ afin que $K_{i+1}\neq K_m$) :

\textit{Première étape :} $t_i<t_{i+1}$ si $s_l^{(i)}=s_l$ pour tout $l$; 

\textit{Deuxième étape :} $t_i<t_{i+1}$ si $s_l^{(i)}\neq s_l$ pour un certain $l$ et si $r^{(i)}=r$; 

\textit{Troisième étape :} $t_i<t_{i+1}$ si $r^{(i)}=S^{(i)}+1\neq r$; 

\textit{Quatrième étape :} $t_i<t_{i+1}$ si $r^{(i)}=S^{(i)}\neq r$. \\

\textit{Première étape :}

Supposons que $s^{(i)}_l=s_l$ pour tout $l$. 
Alors,
\begin{equation*}
K_{i+1}=L_{r^{(i)}+1,s_1^{(i)},\dots,s_n^{(i)}}\neq K_m \; \text{et} \; K_{i+2}=L_{r^{(i)}+2,s_1^{(i)},\dots,s_n^{(i)}}.
\end{equation*}
Ainsi, 
\begin{equation*}
t_i=t_{n,1}(r^{(i)}+1,s_1,\dots,s_n) \; \text{et} \; t_{i+1}=t_{n,1}(r^{(i)}+2,s_1,\dots,s_n).
\end{equation*}
Le Lemme \ref{croissance des suites} appliqué à $K=L_{1,s_1,\dots,s_n}$ et à $\alpha=\zeta_p$ montre que $t_i<t_{i+1}$. \\

\textit{Deuxième étape :}

Commençons par le cas où $r^{(i)}=S^{(i)}=r$.
 Par hypothèse, il existe $l$ tel que $s_l^{(i)}\neq s_l$.
Supposons par l'absurde que $l\leq n-1$.
Alors, $s_{n-1}^{(i)}\neq s_{n-1}$ et donc, $S^{(i)}=s_{n-1}^{(i)}$ d'après \eqref{X}.
Ceci est absurde car 
\begin{equation*}
r=r^{(i)}=S^{(i)}<s_{n-1}\leq r.
\end{equation*}  
Ainsi, $s_{n-1}^{(i)}=s_{n-1}$ (et donc $s_l^{(i)}=s_l$ pour tout $l\leq n-1$). 
Il s'ensuit donc que $s_n^{(i)}\neq s_n$.
Dans ce cas, $S^{(i)}=s_n^{(i)}+1$ d'après \eqref{Y}.
Par conséquent, $r=s_n^{(i)}+1$.
Comme $r\geq s_n$ et que $s_n^{(i)}<s_n$, il s'ensuit que  $s_n\leq r <s_n+1$.
Ainsi, $r=s_n$.
Comme $s_n^{(i)}=r-1=s_n-1$, on en déduit que 
\begin{equation*}
K_i= L_{r,s_1,\dots,s_{n-1},s_n-1} \; \text{et} \; K_{i+1}=L_{r,s_1,\dots,s_n}=K_m,
 \end{equation*}
 ce qui est absurde.

Ceci montre donc que $r^{(i)}=S^{(i)}+1$.
Supposons que $s_{n-1}^{(i)}\neq s_{n-1}$. 
Dans ce cas, $S^{(i)}=s_{n-1}^{(i)}$ d'après \eqref{X} et donc $r^{(i)}=r=s_{n-1}^{(i)}+1$.
Comme $r\geq s_{n-1}$ et que $r=s_{n-1}^{(i)}+1<s_{n-1}+1$, il s'ensuit que $r=s_{n-1}=s_{n-1}^{(i)}+1$. 
Par conséquent, \[K_{i+1}= L_{r, s_1^{(i+1)},\dots,s_{n-2}^{(i+1)},s_{n-1}, s_n}.\]
Comme $s_{n-1}^{(i+1)}=s_{n-1}$, il en résulte que $s_l^{(i+1)}=s_l$ pour tout $l\leq n-1$. 
Par conséquent, $K_{i+1}=K_m$, ce qui est absurde. 

Ainsi, $s_{n-1}^{(i)}=s_{n-1}$ (et donc $s_l^{(i)}= s_l$ pour tout $l\leq n-1$) et $s_n^{(i)}\neq s_n$.
D'après \eqref{Y}, on obtient que $r=r^{(i)}=S^{(i)}+1=s_n^{(i)}+2$.
On en déduit que $s_n^{(i)}=r-2$. 
Comme $r\geq s_n > s_n^{(i)}$, on en conclut que $s_n\in\{r-1,r\}$. 
Si $s_n=r-1$ alors, comme $K_i=L_{r,s_1,\dots,s_{n-1},r-2}$, on en déduirait que $K_{i+1}=K_m$, ce qui est absurde. 
Donc, $s_n=r$ et on obtient que \[K_{i+1}=L_{r,s_1,\dots,s_{n-1}, s_n-1} \; \text{et} \; K_{i+2}=K_m.\]
On en déduit donc que \[t_{m-1}=t_{n,n+1}(r,s_1,\dots,s_n) \; \text{et} \; t_{m-2}=t_{n,n+1}(r,s_1,\dots,s_{n-1},s_n-1).\]
Le fait que $t_{m-2}< t_{m-1}$ découle de la proposition \ref{croissance des suites} appliquée à $K=L_{r,s_1,\dots,s_{n-1},0}$ et à $\alpha=a_n$. \\

\textit{Troisième étape :} 
Supposons que $s_n^{(i)}\neq s_n$.
Ainsi, $r^{(i)}=S^{(i)}+1=s_n^{(i)}+2$ d'après \eqref{Y}. 
On en déduit donc que
\begin{equation*}
K_{i+1}=L_{r^{(i)}, \min\{s_1,s_1^{(i)}+1\},\dots,\min\{s_{n-1},s_{n-1}^{(i)}+1\}, s_n^{(i)}+1}.
\end{equation*}
Comme, $s_n^{(i+1)}=s_n^{(i)}+1$, on a $r^{(i+1)}=r^{(i)}=S^{(i)}+1=S^{(i+1)}$. 
Ainsi,
\begin{equation*}
K_{i+2}=L_{r^{(i)}+1, \min\{s_1,s_1^{(i)}+1\},\dots,\min\{s_{n-1},s_{n-1}^{(i)}+1\}, s_n^{(i)}+1}.
\end{equation*}
On obtient alors 
\begin{equation*}
t_i=t_{n,n+1}(r^{(i)},s_1^{(i+1)},\dots,s_{n-1}^{(i+1)},s_n^{(i)}+1)
\end{equation*}
et
\begin{equation*}
t_{i+1}=t_{n,1}(r^{(i)}+1,s_1^{(i+1)},\dots,s_{n-1}^{(i+1)},s_n^{(i)}+1).
\end{equation*}

Notons $E=\{j \; \vert s_j^{(i)}\neq s_j\}$. 
Alors, $s_j^{(i+1)}=s_j^{(i)}$ si $j\notin E$ et $s_j^{(i+1)}=s_j^{(i)}+1$ si $j\in E$. 
De plus, $s_j^{(i)}+1=s_n^{(i)}+2$ si $j\in E\backslash\{n\}$ d'après \eqref{pour montrer je sais pas quoi}.
D'après le corollaire \ref{unicité du saut} appliqué à $i=n$ et où on remplace $s_n+1$ par $s_n^{(i)}+2$, on en déduit que
\[t_{n,n+1}(r^{(i)}+1,s_1^{(i+1)},\dots,s_{n-1}^{(i+1)},s_n^{(i)}+1)=t_{n,n+1}(r^{(i)}+1,s_1^{(i)},\dots,s_{n-1}^{(i)},s_n^{(i)}+1).\]
Par construction, on a que $s_k^{(i)}\leq s_n^{(i)}+1$ (en fait, c'est même égal si $s_k^{(i)}\neq s_k$ d'après \eqref{pour montrer je sais pas quoi}).
Ainsi, $t_i<t_{i+1}$ d'après le $i)$ du Lemme \ref{majorations utiles dans le cas où la valuation n'est pas divisaible pas p} appliqué à $s=s_n^{(i)}+1$ (on rappelle que $r^{(i)}=s_n^{(i)}+2$).

Supposons maintenant que $s_{n-1}^{(i)}\neq s_{n-1}$ et $s_n^{(i)}=s_n$.
Ainsi, $r^{(i)}=S^{(i)}+1=s_{n-1}^{(i)}+1$ d'après \eqref{X}.
De plus,
\begin{equation*}
K_{i+1}=L_{r^{(i)}, \min\{s_1,s_1^{(i)}+1\},\dots,\min\{s_{n-2},s_{n-2}^{(i)}+1\}, s_{n-1}^{(i)}+1,s_n}.
\end{equation*}
Comme $s_{n-1}^{(i+1)}= s_{n-1}^{(i)}+1$, on en conclut que $r^{(i+1)}=r^{(i)}=S^{(i)}+1=S^{(i+1)}$. 
Ainsi,
\begin{equation*}
K_{i+2}=L_{r^{(i)}+1, , \min\{s_1,s_1^{(i)}+1\},\dots,\min\{s_{n-2},s_{n-2}^{(i)}+1\}, s_{n-1}^{(i)}+1,s_n}.
\end{equation*}
On a alors 
\begin{equation*}
t_i=t_{n,n}(r^{(i)},s_1^{(i+2)},\dots,s_{n-2}^{(i+2)},s_{n-1}^{(i)}+1,s_n)
\end{equation*}
et  \[t_{i+1}=t_{n,1}(r^{(i)}+1,s_1^{(i+2)},\dots,s_{n-2}^{(i+2)},s_{n-1}^{(i)}+1,s_n).\]
D'après le $i)$ appliqué à $k=n$ (et donc $\min\{s_j^{(i+2)}, s_{n-1}^{(i+2)}\}=s_j^{(i+2)}$ pour tout $j\leq n-1$ car $s_j\leq s_{n-1}$ et donc que $s_j^{(i+2)}\leq s_{n-1}^{(i+2)}$ par construction) et le $v)$ appliqué à $j=n-1$ du théorème \ref{thm ou je calcule TOUT} montre que 
\begin{equation*}
t_i-t_{i+1}=p^{\sum_{j=1}^{n-2} s_j^{(i+2)}}\big( t_{2,2}(r^{(i)},s_{n-1}^{(i)}+1,s_n)-t_{2,1}(r^{(i)}+1,s_{n-1}^{(i)}+1,s_n)\big).
\end{equation*}
Comme $s_n^{(i)}=s_n$ et que $s_{n-1}^{(i)}\neq s_{n-1}$, on en déduit, par construction des $K_i$, que $s_n< s_{n-1}^{(i)}$.
Ainsi, $t_{2,2}(r^{(i)}, s_{n-1}^{(i)}+1,s_n)=\tau_2(s_{n-1}^{(i)}+1)$ (pour la définition de $\tau_2$, voir le début de la section \ref{calculs chiants}). 
En utilisant le $ii)$ du théorème \ref{thm ou je calcule TOUT} et la Remarque \ref{si je choisis de d'abord faire baisser le s n}, on en déduit que
\begin{equation*}
t_i-t_{i+1}=p^{s_n+\sum_{j=1}^{n-2}s_j^{(i+2)}}\big( t_{1,2}(r^{(i)},s_{n-1}^{(i)}+1)-t_{1,1}(r^{(i)}+1,s_{n-1}^{(i)}+1)\big)<0
\end{equation*}
d'après le $ii)$ du Lemme \ref{dernière minoration qu'il me manquait} (on rappelle que $r^{(i)}=S^{(i)}+1=s_{n-1}^{(i)}+1)$. \\

\textit{Quatrième étape :}

Commençons par supposer que $s_n^{(i)}\neq s_n$.
Alors, $r^{(i)}=S^{(i)}=s_n^{(i)}+1$ d'après \eqref{X}. 
Ainsi, 
\begin{equation*}
K_{i+1}=L_{r^{(i)}+1, s^{(i)}_1,\dots,s^{(i)}_n}.
\end{equation*}
Comme $r^{(i+1)}=r^{(i)}+1=S^{(i)}+1=S^{(i+1)}+1$, cela montre donc que
\begin{equation*}
K_{i+2}=L_{r^{(i)}+1, \min\{s_1,s_1^{(i)}+1\},\dots,\min\{s_n,s_n^{(i)}+1\}}.
\end{equation*}
On obtient alors 
\begin{equation*}
t_i=t_{n,1}(r^{(i)}+1,s^{(i)}_1,\dots,s_n^{(i)}) \; \text{et} \; t_{i+1}=t_{n,n+1}(r^{(i)}+1,s_1^{(i+2)},\dots,s_{n-1}^{(i+2)},s_n^{(i)}+1).
\end{equation*}
Notons $E=\{j \; \vert s_j^{(i)}\neq s_j\}$. 
Alors, $s_j^{(i+2)}=s_j^{(i)}$ si $j\notin E$ et $s_j^{(i+2)}=s_j^{(i)}+1$ si $j\in E$. 
De plus, $s_j^{(i)}+1=s_n^{(i)}+2$ si $j\in E$ d'après \eqref{pour montrer je sais pas quoi}.
D'après le corollaire \ref{unicité du saut} appliqué à $i=n$ et où on a remplacé $s_n+1$ par $s_n^{(i)}+2$, on en déduit que
\[t_{n,n+1}(r^{(i)}+1,s_1^{(i+2)},\dots,s_{n-1}^{(i+2)},s_n^{(i)}+1)=t_{n,n+1}(r^{(i)}+1,s_1^{(i)},\dots,s_{n-1}^{(i)},s_n^{(i)}+1).\]
Par construction, on a que $s_k^{(i)}\leq s_n^{(i)}+1$ (en fait, c'est même égal si $s_k^{(i)}\neq s_k$ d'après \eqref{pour montrer je sais pas quoi}).
Ainsi, $t_i<t_{i+1}$ d'après le $ii)$ du Lemme \ref{majorations utiles dans le cas où la valuation n'est pas divisaible pas p} appliqué à $s=s_n^{(i)}+1$ (on rappelle que $r^{(i)}=s_n^{(i)}+1$). 

Supposons maintenant que $s_n^{(i)}=s_n$ et que $s_{n-1}^{(i)}\neq s_{n-1}$.
On a donc
\begin{equation*}
K_{i+1}=L_{r^{(i)}+1, s^{(i)}_1,\dots,s^{(i)}_n}.
\end{equation*}
Comme $r^{(i+1)}=r^{(i)}+1=S^{(i)}+1=S^{(i+1)}+1=s_{n-1}^{(i)}+1$ d'après \eqref{X}, cela montre que
\begin{equation*}
K_{i+2}=L_{r^{(i)}+1, \min\{s_1,s_1^{(i)}+1\},\dots,\min\{s_{n-2},s_{n-2}^{(i)}+1\}, s_{n-1}^{(i)}+1,s_n}.
\end{equation*}
On a alors 
\begin{equation*}
t_i=t_{n,1}(r^{(i)}+1,s^{(i)}_1,\dots,s_{n-1}^{(i)},s_n)
\end{equation*}
et
\begin{equation*}
t_{i+1}=t_{n,n}(r^{(i)}+1,s_1^{(i+2)},\dots,s_{n-2}^{(i+2)},s_{n-1}^{(i)}+1,s_n).
\end{equation*}

Notons $E=\{j\; \vert \; s_j^{(i)}\neq s_j\}$. 
Ainsi, $s_j^{(i+2)}=s_j^{(i)}$ si $j\notin E$ et $s_j^{(i+2)}=s_j^{(i)}+1$ si $j\in E$. 
D'après le corollaire \ref{unicité du saut 3} appliqué à $s=s_{n-1}^{(i)}$ et à $k=i=n-1$, on a \[ t_{n,n}(r^{(i)},s_1^{(i+2)},\dots,s_{n-2}^{(i+2)},s_{n-1}^{(i)}+1,s_n)= t_{n,n}(r^{(i)}+1,s^{(i)}_1,\dots,s_{n-2}^{(i)},s_{n-1}^{(i)}+1,s_n).\]
D'après le $i)$ appliqué à $k=n$ (et donc $\min\{s_j^{(i)}, s_{n-1}^{(i)}\}=s_j^{(i)}$ pour tout $j\leq n-1$ car $s_j\leq s_{n-1}$ et donc que $s_j^{(i)}\leq s_{n-1}^{(i)}$ par construction) et le $v)$ appliqué à $j=n-1$ du théorème \ref{thm ou je calcule TOUT}, on en déduit que
\begin{equation*}
t_i-t_{i+1}=p^{\sum_{j=1}^{n-2} s_j^{(i)}}\big( t_{2,1}(r^{(i)}+1,s_{n-1}^{(i)},s_n)-t_{2,2}(r^{(i)}+1,s_{n-1}^{(i)}+1,s_n)).
\end{equation*}
Comme $s_n< s_{n-1}^{(i)}$, on en déduit que $t_{2,2}(r^{(i)}, s_{n-1}^{(i)}+1,s_n)=\tau_2(s_{n-1}^{(i)}+1)$. 
De plus, en utilisant le $ii)$ du théorème \ref{thm ou je calcule TOUT} et la Remarque \ref{si je choisis de d'abord faire baisser le s n}, 
\begin{equation*}
t_i-t_{i+1}=p^{s_n+\sum_{j=1}^{n-2} s_j^{(i)}}\big( t_{2,1}(r^{(i)}+1,s_{n-1}^{(i)})-t_{2,2}(r^{(i)}+1,s_{n-1}^{(i)}+1)\big)<0
\end{equation*}
d'après le $i)$ du Lemme \ref{dernière minoration qu'il me manquait} (on rappelle que $r^{(i)}=S^{(i)}=s_{n-1}^{(i)}$). \\

On a ainsi montré que $t_1<\dots<t_{m-1}$.
Comme $K_m=L_{r,s_1,\dots,s_n}$, le corollaire \ref{utile pour déterminer la suite des groupes de rami} montre que les sauts de $G$ non nuls sont les $t_1,\dots,t_{m-1}$ et que $G_{t_i}=\Gal(K_m/K_i)$.
On conclut en remarquant que les $t_1,\dots,t_{m-1}$ correspondent bien aux sauts énoncés dans le théorème.
\end{proof}

\begin{ex}
Les sauts de $G=\Gal(L_{4,1,2,3}/F)$ qui sont différents de $0$ sont :
\begin{align*}
& t_1=t_{3,4}(1,1,1,0); \; t_2=t_{3,1}(2,1,1,0) \\
& t_3=t_{3,4}(2,1,2,1); \; t_4=t_{3,1}(3,1,2,1); \; t_5=t_{3,4}(3,1,2,2) \\
& t_6= t_{3,1}(4,1,2,2); \; t_7=t_{3,4}(4,1,2,3).
\end{align*}
De plus, la suite des groupes de ramification est 
\begin{align*}
& G_{t_1}=\Gal(L_{4,1,2,3}/L_{1,0,0,0});\;  G_{t_2}=\Gal(L_{4,1,2,3}/L_{1,1,1,0}) \\
& G_{t_3}=\Gal(L_{4,1,2,3}/L_{2,1,1,0}); \; G_{t_4}=\Gal(L_{4,1,2,3}/L_{2,1,2,1}) \\
& G_{t_5}=\Gal(L_{4,1,2,3}/L_{3,1,2,1}); \; G_{t_6}=\Gal(L_{4,1,2,3}/L_{3,1,2,2}) \\
& G_{t_7}=\Gal(L_{4,1,2,3}/L_{4,1,2,2}).
\end{align*}
\end{ex}

\section{Une généralisation du théorème \ref{calcul des sauts, cas où la valuation est divisible par p} et du théorème \ref{calcul des sauts, cas où la valuation n'est pas divisible par p}.} \label{Une généralisation des Théorème}

  Soit $L/F$ une extension galoisienne, finie et radicale.
On souhaite calculer la suite des groupes de ramification de $L/F$ dans le cas $p\neq 2$.
Si $p=2$, on les calculera aussi, à condition de faire une hypothèse supplémentaire. 
Notons $m$ l'exposant de $G:=\Gal(L/F)$.
Si $m$ est premier à $p$, la suite des groupes de ramification est facile à calculer puisque $G_1=\{1\}$. 
On suppose donc que $m$ soit divisible par $p:=p_0$.
Dans la section précédente, on les a calculés dans le cas où $m$ est une puissance de $p$. 
Supposons donc que $m$ ne soit pas non plus une puissance de $p$ et notons $m=p^r\prod_{j=1}^l p_j^{r_j}$ la décomposition de $m$ en nombres premiers.

 Pour tout $j\in \{0,\dots,l\}$, il existe deux $n_j$-uplet $(a_{i,j})_{i=1}^{n_j}\in \left(\mathcal{O}_F\backslash \mathcal{O}_F^{p_j}\right)^{n_j}$ et $(t_{i,j})_{i=1}^{n_j}\in(\N^*)^{n_j}$ tels que $L$ soit égal au compositum de tous les $F\left(\zeta_{p_j^{r_i}}, a_{i,j}^{1/p_j^{t_{i,j}}}\right)$ avec $j\in \{0,\dots,l\}$ et $i\in\{1,\dots,n_j\}$.

Afin d'alléger les notations, on notera $a_i:=a_{i,0}$ et $t_{i,0}=s_i$ pour tout $i\geq 0$. 
Comme pour le cas où $m$ est une puissance de $p$, on supposera, uniquement si $p=2$, que 

\begin{equation} \label{condition supplémentaire pour la section 9}
[F(\zeta_{p^r},a_1^{1/p^r},\dots,a_n^{1/p^r}) : F(\zeta_{p^r})]=p^{nr}.
\end{equation}

Par les mêmes arguments que la section \ref{section avec tous les calculs chiants}, on peut de nouveau se réduire au cas où $F$ et les $a_1,\dots,a_n$ vérifient les conditions de l'hypothèse \ref{hypothèse 3}.

Soient $j\in\{1,\dots,l\}$ et $i\in\{1,\dots,n_j\}$. 
Par la théorie de Kummer, les sous-corps de $F\left(\zeta_{p_j}^{r_j}, a_{i,j}^{1/p_j^{t_{i,j}}}\right)$ contenant $F(\zeta_{p_j}^{r_j})$ sont les $F\left(\zeta_{p_j}^{r_j},a_{i,j}^{1/p_j^t}\right)$ où $t\in\{0,\dots,t_{i,j}\}$. 
 Notons $f_{i,j}$ l'unique entier tel que $F\left(\zeta_{p_j}^{r_j},a_{i,j}^{1/p_j^{f_{i,j}}}\right)/F(\zeta_{p_j}^{r_j})$ soit l'extension maximale non-ramifiée de $F\left(\zeta_{p_j}^{r_j}, a_{i,j}^{1/p_j^{t_{i,j}}}\right)/F(\zeta_{p_j}^{r_j})$. 
 Notons également $F^{nr}$ l'extension maximale non-ramifiée de $F$ incluse dans $L$, $s_{i,j}=t_{i,j}-f_{i,j}$ et $\alpha_{i,j}=a_{i,j}^{1/p_j^{f_{i,j}}}$.

Comme $\zeta_{p_j^{r_j}}\in F^{nr}$, on en déduit que $F(\zeta_{p_j^{r_j}}, \alpha_{i,j})\subset F^{nr}$. 
De plus, on a $\alpha_{i,j}^{1/p_j}\notin F^{nr}$. 
Comme le compositum de deux extensions non-ramifiées est non-ramifiée, on en déduit que $L$ est le compositum de $F^{nr}(\zeta_{p^r},a_1^{1/p^{s_1}},\dots,a_n^{1/p^{s_n}})$ et de tous les $F^{nr}\left(\alpha_{i,j}^{1/{p_j^{s_{i,j}}}}\right)$ avec $j\in \{1,\dots,l\}$ et $i\in\{1,\dots,n_j\}$.

Supposons, sans perte de généralité, que $s_{i,j}>0$. 
Notons  $N=n+\sum_{j=1}^l n_j$. 
Les définitions suivantes sont les analogues naturels de la Définition \ref{defn de L r s} et de la Définition \ref{définition  du saut}.

\begin{defn}
Soit $\overline{s}=(s_{1,j}, \dots, s_{n_j,j})_{j=1,\dots,l}$ un $N$-uplet et $r,s_1,\dots,s_n$ des entiers positifs. 

\begin{enumerate} [i)]
\item notons $L_{r,s_1,\dots,s_n,\overline{s}}$ le corps \[ F^{nr}\left(\zeta_{p^r},a_1^{1/p^{s_1}},\dots,a_n^{1/p^{s_n}},\alpha_1^{1/p^{s_{1,1}}},\dots,\alpha_{n_1,1}^{1/p^{s_{n_1,1}}},\dots, \alpha_{1,l}^{1/p^{s_{1,l}}},\dots,\alpha_{n_l,l}^{1/p^{s_{n_l,l}}}\right);\]

\item soit $j\leq N+1$. 
On note \[\overline{s}(j)=(s_{1,1}-\delta_{n+2,j}, \dots,s_{n_1,1}-\delta_{n+n_1+1,j},s_{1,2}-\delta_{n+n_1+2,j},\dots,s_{n_l,l}-\delta_{N+1,j})\]
où $\delta_{i,j}$ désigne le symbole de Kronecker;

\item soit $k\leq N+1$.
On note $t_{N,k}(r,s_1,\dots,s_n,\overline{s})$ l'unique saut de l'extension $L_{r,s_1,\dots,s_n,\overline{s}}/L_{r(k),s_1(k)\dots,s_n(k),\overline{s}(k)}$ où $ r(k)$ et $s_i(k)$ sont définis comme dans \eqref{définition de r k} et \eqref{définition de s k}.  

\end{enumerate}
\end{defn}

On a alors le résultat suivant :

\begin{thm} \label{théorème qui se ramène aux groupes de ramification d'un premier}
 On a $t_{N,k}(r,s_1,\dots,s_n,s_{i,j})=0$ si $k\geq n+2$ ou si $k=1$ et $r=1$. Dans le cas contraire, 
\begin{equation*}
t_{N,k}(r,s_1,\dots,s_n,\overline{s}) = t_{N,k}(r,s_1,\dots,s_n,0,\dots,0)  \prod_{j=1}^l \prod_{i=1}^{n_j} p_j^{s_{i,j}}.
\end{equation*} 

\end{thm}

\begin{proof}
 Si $r=k=1$, alors $L_{1,s_1,\dots,s_n,\overline{s}}/L_{0,s_1,\dots,s_n,\overline{s}}$ est de degré $p-1$.
 Si $k\geq n+2$ alors, l'extension dont le saut est, par définition, $t_{N,k}(r,s_1,\dots,s_n,\overline{s})$ est de degré $p_m$ pour un certain $m\in \{1,\dots,l\}$.
  La première assertion s'ensuit car $\# G_1$ est une puissance de $p$.
  
   Supposons désormais que $k\leq n+1$ et que $r\neq 1$ si $k=1$.
Notons 
\begin{equation*} 
K_1=L_{r(k),s_1(k),\dots,s_n(k),\overline{s}}, K_2=L_{r,s_1,\dots,s_n,\overline{s}(n+2)} \; \text{et} \; K=K_1\cap K_2.
\end{equation*}   
   
En appliquant la proposition \ref{calcul de la différente.} avec $K, K_1$ et $K_2$ comme ci-dessus, on a, avec les notations de la proposition \ref{calcul de la différente.},
\begin{equation*} 
 \begin{aligned}
 & d_1=p_1, \; d_2=p, \\
 & t_1=0, \; t_2=t_{N,k}(r,s_1,\dots,s_n,\overline{s}(n+2)) \\
 & t'_1=0 \; \text{et} \; t'_2=t_{N,k}(r,s_1,\dots,s_n,\overline{s}).
 \end{aligned}
 \end{equation*}
Ainsi, le $i)$ de la proposition \ref{calcul de la différente.} montre que
\begin{equation*}
t_{N,k}(r,s_1,\dots,s_n,\overline{s})=p_1t_{N,k}(r,s_1,\dots,s_n,\overline{s}(n+2)).
\end{equation*}

 En répétant le même argument en remplaçant $s_{1,1}$ par $s_{1,1}-1$ (et ainsi de suite), on déduit par récurrence que 
\begin{equation*}
t_{N,k}(r,s_1,\dots,s_n,\overline{s}) =p_1^{s_{1,1}} t_{N,k}(r,s_1,\dots,s_n,0,\overline{s'}).
\end{equation*}
En répétant le même argument en remplaçant $s_{1,1}$ par $s_{1,2}$ (et ainsi de suite), on déduit par récurrence que 
\begin{equation*}
t_{N,k}(r,s_1,\dots,s_n,\overline{s}) =  t_{N,k}(r,s_1,\dots,s_n,0,\dots,0)  \prod_{j=1}^l \prod_{i=1}^{l_j} p_j^{s_{i,j}},
\end{equation*}
ce qui montre le théorème.

\end{proof}

Ainsi, pour calculer explicitement $t_{N,k}(r,s_1,\dots,s_n,\overline{s})$ avec $k\leq n+1$, il suffit de calculer explicitement $t_{N,k}(r,s_1,\dots,s_n,0,\dots,0)$.
Or, si on montre que $F^{nr}$ et $a_1,\dots,a_n$ vérifient les conditions de l'hypothèse \ref{hypothèse 3} alors, on pourra ainsi utiliser le théorème \ref{thm ou je calcule TOUT} pour en déduire la valeur de $t_{N,k}(r,s_1,\dots,s_n,0,\dots,0)$.
Montrons que c'est le cas.

Il est clair que les conditions $iii)$ et $iv)$ sont vérifiées. 
Comme $i)$ est une conséquence directe du $ii)$, il reste plus qu'à montrer $ii)$. 
Remarquons que si $p=2$, cela se déduit de notre unique restriction \eqref{condition supplémentaire pour la section 9}. 
Supposons donc que $p\neq 2$. 
D'après le Fait \ref{fait}, il suffit de montrer que pour tout $i\geq 2$, on a $[a_i]_{F^{nr}}\notin \langle [a_1]_{F^{nr}},\dots,[a_{i-1}]_{F^{nr}}\rangle$.
  Supposons par l'absurde que ce ne soit pas le cas. 
  Alors, il existe des entiers $x_1,\dots,x_{i-1}$  et $c\in F^{nr}$ tels que $a_i= c^p\prod_{l=1}^{i-1} a_l^{x_l}$.
   Mais alors, $a_i\left(\prod_{l=1}^{i-1} a_l^{x_l}\right)^{-1}\in F\cap (F^{nr})^p$. 
   Ainsi, $\left(a_i\prod_{l=1}^{i-1} a_l^{-x_l}\right)^{1/p}\in F^{nr}\cap F(\zeta_{p^r}, a_1^{1/p^{s_1}},\dots,a_n^{1/p^{s_n}})$. 
Comme $F^{nr}/F$ est non ramifiée et que $F(\zeta_{p^r}, a_1^{1/p^{s_1}},\dots,a_n^{1/p^{s_n}})/F$ est totalement ramifiée, on en conclut que $\left(a_i\prod_{l=1}^{i-1} a_l^{-x_l}\right)^{1/p}\in F$ et donc que $[a_i]_F\in \langle [a_1]_F,\dots,[a_{i-1}]_F\rangle$, ce qui contredit le $ii)$ de l'hypothèse \ref{hypothèse 3} (appliqué cette fois-ci à $F$).
Ainsi, on peut appliquer tous les résultats des sections précédentes à $F^{nr}$ au lieu de $F$. 
 
 Rappelons que $\overline{s}=(s_{1,j}, \dots, s_{n_j,j})_{j=1,\dots,l}$.
Notons $D=\prod_{j=1}^l \prod_{i=1}^{n_j} p^{s_{i,j}}$ et, si $k\leq n+1$, $t_{N,k}(r,s_1,\dots,s_n)$ le saut $t_{N,k}(r,s_1,\dots,s_n,0,\dots,0)$. 
Comme $D$ est une constante ne dépendant pas des $s_1,\dots,s_n$, on déduit du théorème \ref{théorème qui se ramène aux groupes de ramification d'un premier}, en recopiant la preuve du théorème \ref{calcul des sauts, cas où la valuation est divisible par p} et du théorème \ref{calcul des sauts, cas où la valuation n'est pas divisible par p} présentes dans la section \ref{calcul des sauts}, que 

\begin{thm}
Soient $F$ et $a_1,\dots,a_n$ vérifiant l'hypothèse \ref{hypothèse 3}. 
 Supposons de plus que $p\mid v_F(a_n)$ et que $s_1\leq \dots \leq s_n$.
Alors, les sauts de $G$, différents de $0$, sont les 

\begin{enumerate} [i)]
\item $t_{n,1}(s,s_1,\dots,s_n)D$ avec $s\in \{s_n+1,\dots,r\}$;
\item $t_{n,1}(s_n'+1,s_1',\dots,s_n')D$ où $s_n'\in\{1,\dots,s_n-1\}$ et où $s_i'=\min\{s_i,s_n'\}$;
\item $t_{n,n+1}(s_n',s_1',\dots,s_n')D$ où $s_n'\in\{1,\dots,s_n\}$ et où $s_i'=\min\{s_i,s_n'\}$.
\end{enumerate}

De plus, si on note $u_j$ le $(n+1)$-uplet tel que $t_j=t_{n,k}(u_j,\overline{s})$ désigne le $(j+1)$-ième plus petit saut de $G$ alors, 

\begin{enumerate} [i)]
\item $G_0=G$;
\item $G_{t_j}=\Gal(L_{r,s_1,\dots,s_n}/L_{u_{j-1},\overline{s}})$ pour $j\geq 1$.
\end{enumerate}

\end{thm}

\begin{thm} 
Soient $F$ et $a_1,\dots,a_n$ vérifiant l'hypothèse \ref{hypothèse 3}.
Supposons de plus que $p\nmid v_F(a_n)$ et que $s_1\leq \dots\leq s_{n-1}$.
Alors, les sauts de $G$, différents de $0$, sont les 

\begin{enumerate} [i)]
\item $t_{n,n+1}(r,s_1,\dots,s_n)D$ si $r=s_n$;
\item $t_{n,1}(s_n'+1,s_1',\dots,s_{n-1}', \min\{s_n'-1,s_n\})D$ où $s_n'\in\{1,\dots,r\}$ et $s_i'=\min\{s_i,s_n'\}$;
\item $t_{n,n}(s_{n-1}',s_1',\dots,s_{n-1}',s_n)D$ où $s_{n-1}'\in\{s_n+2,\dots,s_{n-1}\}$ et $s_i'=\min\{s_i,s_{n-1}'\}$;
\item $t_{n,n+1}(s_n'+1,s_1',\dots,s_n')D$ où $s_n'\in\{1,\dots,s_n\}$ et $s_i'=\min\{s_i,s_n'+1\}$.
\end{enumerate}

De plus, si on note $u_j$ le $(n+1)$-uplet tel que $t_j=t_{n,k}(u_j,\overline{s})$ désigne le $(j+1)$-ième plus petit saut de $G$ alors,

\begin{enumerate} [i)]
\item $G_0=G$;
\item $G_{t_j}=\Gal(L_{r,s_1,\dots,s_n}/L_{u_{j-1},\overline{s}})$ pour $j\geq 1$.
\end{enumerate}
\end{thm}

\bibliographystyle{plain}

\end{document}